\documentclass[11pt,leqno]{article}
\usepackage{amsfonts}

\include{psfig}
\include{epsf}

\usepackage{color}
\usepackage{amssymb} 

\newtheorem{theorem}{Theorem}[section]
\newtheorem{lemma}{Lemma}[section]

\newtheorem{definite}{Definition}[section]

\newtheorem{remark}{Remark}[section]

\catcode`\@=11
\renewcommand{\section}{
         \setcounter{equation}{0}
         \@startsection {section}{1}{\z@}{-3.5ex plus -1ex minus
         -.2ex}{2.3ex plus .2ex}{\normalsize\bf}
}
\renewcommand{\subsection}{
         \@startsection {subsection}{1}{\z@}{-3.5ex plus -1ex minus
         -.2ex}{2.3ex plus .2ex}{\normalsize\bf}
} \catcode`\@=12

\def\reals{{\rm\vrule depth0ex width.4pt\kern-.08em R}}
\def\bbbz{{\mathchoice {\hbox{$\sf\textstyle Z\kern-0.4em Z$}}
{\hbox{$\sf\textstyle Z\kern-0.4em Z$}} {\hbox{$\sf\scriptstyle
Z\kern-0.3em Z$}} {\hbox{$\sf\scriptscriptstyle Z\kern-0.2em Z$}}}}

\newcommand{\nc}{\newcommand}
\nc{\W}{{\bf W}} \nc{\A}{{\bf A}} \nc{\bL}{{\bf L}} \nc{\bH}{{\bf
H}} \nc{\C}{{\cal C}}

\def\eq#1{(\ref{e:#1})}

\def\elabel#1{\label{e:#1}}

\begin{document}
\begin{center}
\Large\bf Unified Systems of FB-SPDEs/FB-SDEs with Jumps/Skew
Reflections and Stochastic Differential Games~\footnote{This work
was presented in parts or as a whole at 8th World Congress in
Probability and Statistics, Istanbul (2012), American Institute of
Mathematical Sciences Annual Conference, Madrid (2014), 8th
International Congress of Industrial and Applied Mathematics, 
Beijing (2015), and at conferences in Suzhou (2011), Minneapolis
(2013), Chengdu (2013), Shanghai (2014), Jinan (2014), Nanjing
(2014), Beijing (2015), Chongqing (2015), Changsha (2015), Kunming
(2015), Shenyang (2015), etc. The presentations as invited plenary
talks in SCETs 2014/2015 and ICPDE 2015 were video-taped by Wanfang
Data and can be downloaded online. The author thanks the helpful
comments from the participants. An earlier version on the unified
B-SPDE was posted via arxiv in May of 2011 and on the unified
systems of FB-SPDEs/FB-SDEs with skew reflections was posted via
arxiv on June 16 of 2015.}
\end{center}
\begin{center}
\large\bf Wanyang Dai~\footnote{Supported by National Natural
Science Foundation of China with Grant No. 10971249 and Grant No.
11371010.}
\end{center}
\begin{center}
\small Department of Mathematics and State Key Laboratory of Novel
Software Technology\\
Nanjing University, Nanjing 210093, China\\
Email: nan5lu8@netra.nju.edu.cn\\
Date: 14 September 2015
\end{center}

\vskip 0.1 in
\begin{abstract}
We study four systems and their interactions. First, we formulate a
unified system of coupled forward-backward stochastic {\it partial}
differential equations (FB-SPDEs) with L\'evy jumps, whose drift,
diffusion, and jump coefficients may involve partial differential
operators. A solution to the FB-SPDEs is defined by a 4-tuple
general dimensional random vector-field process evolving in time
together with position parameters over a domain (e.g., a hyperbox or
a manifold). Under an infinite sequence of generalized local linear
growth and Lipschitz conditions, the well-posedness of an adapted
4-tuple strong solution is proved over a suitably constructed
topological space. Second, we consider a unified system of FB-SDEs,
a special form of the FB-SPDEs, however, with {\it skew} boundary
reflections. Under randomized linear growth and Lipschitz conditions
together with a general completely-${\cal S}$ condition on
reflections, we prove the well-posedness of an adapted 6-tuple weak
solution with {\it boundary regulators} to the FB-SDEs by the
Skorohod problem and an oscillation inequality. Particularly, if the
spectral radii in some sense for reflection matrices are strictly
less than the unity, an adapted 6-tuple strong solution is
concerned. Third, we formulate a stochastic differential game (SDG)
with general number of players based on the FB-SDEs. By a solution
to the FB-SPDEs, we get a solution to the FB-SDEs under a given
control rule and then obtain a Pareto optimal Nash equilibrium
policy process to the SDG. Fourth, we study the applications of the
FB-SPDEs/FB-SDEs in queueing systems and quantum statistics while
we use them to motivate the SDG.\\



\end{abstract}

\section{Introduction}

We study four systems and their interactions: a unified system of
coupled forward-backward stochastic partial differential equations
(FB-SPDEs) with L\'evy jumps; a unified system of FB-SDEs, a special
form of the FB-SPDEs, however, with skew reflections; a stochastic
differential game (SDG) problem with general number of players based
on the FB-SDEs; and a system of queues and their associated
reflecting diffusion approximations. More precisely, there are four
interconnected and streamlined aims involved in our discussions.

The first aim is to study the adapted 4-tuple strong solution
$(U,V,\bar{V},\tilde{V})$ to the unified system of coupled FB-SPDEs
with L\'evy jumps with respect to time-position parameter $(t,x)\in
R_{+}\times D$,
\begin{eqnarray}
\;\;\;\;\;\;\;\;\;\;\left\{\begin{array}{ll}
U(t,x)&=G(x)+\int_{0}^{t}{\cal L}(s^{-},x,U,V,\bar{V},\tilde{V})ds\\
&\;\;\;\;\;\;\;\;\;\;\;\;+\int_{0}^{t}{\cal J}(s^{-},x,U,V,\bar{V},\tilde{V})dW(s)\\
&\;\;\;\;\;\;\;\;\;\;\;\;+\int_{0}^{t}\int_{{\cal Z}^{h}}{\cal
I}(s^{-},x,U,V,\bar{V},\tilde{V},z)\tilde{N}(\lambda ds,dz),\\
V(t,x)&=H(x)+\int_{t}^{\tau}\bar{{\cal L}}(s^{-},x,U,V,\bar{V},\tilde{V})ds\\
&\;\;\;\;\;\;\;\;\;\;\;\;\;+\int_{t}^{\tau}\bar{{\cal
J}}(s^{-},x,U,V,\bar{V},\tilde{V})dW(s)\\
&\;\;\;\;\;\;\;\;\;\;\;\;\;+\int_{t}^{\tau}\int_{{\cal
Z}^{h}}\bar{{\cal
I}}(s^{-},x,U,V,\bar{V},\tilde{V},z)\tilde{N}(\lambda ds,dz),
\end{array}
\right. \elabel{fbspdef}
\end{eqnarray}
where, $t\in[0,\tau]$ and $\tau\in[0,T]$ is a stopping time with
regard to a filtration defined later in the paper, ${\cal
Z}^{h}=R^{h}-\{0\}$ or $R_{+}^{h}$ for a positive integer $h$, and
$s^{-}$ denotes the corresponding left limit at time point $s$. In
particular, $D\in R^{p}$ with a given $p\in{\cal N}=\{1,2,...\}$ is
a connected domain, for examples, a $p$-dimensional box, a
$p$-dimensional ball (or a general manifold), a $p$-dimensional
sphere (or a general Riemannian manifold), or the whole Euclidean
space $R^{p}$ of real numbers itself. The F-SPDE in \eq{fbspdef} is
with the given initial random vector-field $G$, while the B-SPDE in
\eq{fbspdef} has
the known terminal random vector-field $H$. In \eq{fbspdef}, $U$ and
$V$ are $r$-dimensional and $q$-dimensional random vector-field
processes respectively, $W$ is a standard $d$-dimensional Brownian
motion, and $\tilde{N}$ is a $h$-dimensional centered L\'evy jump
process (or centered subordinator). Furthermore, the partial
differential operators of $r$-dimensional vector ${\cal L}$,
$r\times d$-dimensional matrix ${\cal J}$, and $r\times
h$-dimensional matrix ${\cal I}$ are functionals of $U$, $V$,
$\bar{V}$, $\tilde{V}$, and their partial derivatives of up to the
$k$th order for $k\in\{0,1,2,3,...\}$. So do the partial
differential operators of $q$-dimensional vector $\bar{{\cal L}}$,
$q\times d$-dimensional matrix $\bar{{\cal J}}$, and $q\times
h$-dimensional matrix $\bar{{\cal I}}$. More precisely, for each
${\cal A}\in\{{\cal L},{\cal J}, \bar{{\cal L}},\bar{{\cal J}}\}$,
\begin{eqnarray}
{\cal A}(s,x,U,V,\bar{V},\tilde{V})&\equiv&{\cal
A}(s,x,(U,\frac{\partial U}{\partial
x_{1}},...,\frac{\partial^{k}U}{\partial x_{1}^{i_{1}}...{\partial
x_{p}^{i_{p}}}})(s,x),
\elabel{operatorsex}\\
&&\;\;\;\;\;\;\;\;\;\;\;\;(V,\frac{\partial V}{\partial
x_{1}},...,\frac{\partial^{k}V}{\partial x_{1}^{i_{1}}...{\partial
x_{p}^{i_{p}}}})(s,x),
\nonumber\\
&&\;\;\;\;\;\;\;\;\;\;\;\;(\bar{V},\frac{\partial\bar{V}}{\partial
x_{1}},...,\frac{\partial^{k}\bar{V}}{\partial
x_{1}^{i_{1}}...{\partial x_{p}^{i_{p}}}})(s,x),
\nonumber\\
&&\;\;\;\;\;\;\;\;\;\;\;\;(\tilde{V},\frac{\partial\tilde{V}}{\partial
x_{1}},...,\frac{\partial^{k}\tilde{V}}{\partial
x_{1}^{i_{1}}...{\partial x_{p}^{i_{p}}}})(s,x,\cdot),\cdot),
\nonumber
\end{eqnarray}
where the dot ``$\cdot$" in $\tilde{V}(s,x,\cdot)$ and its
associated partial derivatives denotes the integration in terms of
the so-called L\'evy measure. However, if ${\cal A}\in\{{\cal
I},\bar{{\cal I}}\}$, the last line on the right-hand side of
\eq{operatorsex} should be changed to the form,
\begin{eqnarray}
&&(\tilde{V},\frac{\partial\tilde{V}}{\partial
x_{1}},...,\frac{\partial^{k}\tilde{V}}{\partial
x_{1}^{i_{1}}...{\partial x_{p}^{i_{p}}}})(s,x,z),z,\cdot).
\nonumber
\end{eqnarray}
\begin{figure}[tbh]
\begin{center}
\epsfxsize=3.5in\epsfbox{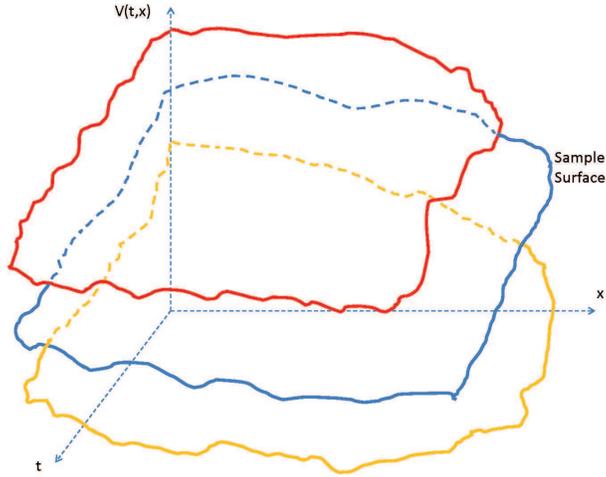}
\caption{\small Sample Surface Solution to the FB-SPDEs}
\label{samplesurface}
\end{center}
\end{figure}

Note that our partial differential operators presented in
\eq{operatorsex} can be general-nonlinear and high-order, e.g.,
\begin{eqnarray}
&&{\cal A}(s,x,U,V,\bar{V},\tilde{V})=f(\frac{\partial^{k}
V(t,x)}{\partial x_{1}^{i_{1}}...\partial x_{p}^{i_{p}}})
\nonumber
\end{eqnarray}
for  a general nonlinear functional $f$, where $r,i_{1},...,i_{p}$
are nonnegative integers satisfying $i_{1}+...+i_{p}=k$ with
$k\in\{0,1,2,3,...\}$. Furthermore, the initial random vector-field
$G$, the terminal random vector-filed $H$, and the 4-tuple solution
process $(U,V,\bar{V},\tilde{V})$ can be complex-valued.

Under an infinite sequence of generalized local linear growth and
Lipschitz conditions, we prove the existence and uniqueness of an
adapted 4-tuple strong solution to the FB-SPDEs in a suitably
constructed functional topological space. The solution to the
unified system in \eq{fbspdef} can be interpreted in a sample
surface manner with time-position parameter $(t,x)$ (see, e.g.,
$V(t,x)$ in Figure~\ref{samplesurface} for such an example). The
quite involved technical proof developed in this paper is extended
from our earlier work summarized in Dai~\cite{dai:newcla} (arxiv,
2011) for a unified B-SPDE.

More precisely, the newly unified system of coupled FB-SPDEs in
\eq{fbspdef} covers many existing forward and/or backward SDEs/SPDEs
as special cases, where the partial differential operators are taken
to be special forms. For examples, specific single-dimensional
strongly nonlinear F-SPDE and B-SPDE driven solely by Brownian
motions can be respectively derived for the purpose of
optimal-utility based portfolio choice (see, e.g, Musiela and
Zariphopoulou~\cite{muszar:stopar}). Here, the strong nonlinearity
is in the sense addressed by Lions and
Souganidis~\cite{liosou:notaux} and Pardoux~\cite{par:stopar}.
Furthermore, the single-dimensional stochastic
Hamilton-Jacobi-Bellman (HJB) equations are also examples of our
unified system in \eq{fbspdef}, which are specific B-SPDEs (see,
e.g., $\emptyset$ksendal {\em et al.}~\cite{okssul:stohjb} and
references therein). Note that the proof of the well-posedness
concerning solution to the B-SPDE derived in Musiela and
Zariphopoulou~\cite{muszar:stopar} and solution to the HJB equation
derived in $\emptyset$ksendal {\em et al.}~\cite{okssul:stohjb} is
covered by the study in Dai~\cite{dai:newcla} (arxiv, 2011) although
the authors in both \cite{muszar:stopar} and $\emptyset$ksendal {\em
et al.}~\cite{okssul:stohjb} claim it as an open problem. The proof
of the well-posedness about solution to the F-SPDE derived in
Musiela and Zariphopoulou~\cite{muszar:stopar} is covered by the
even more unified discussion for the coupled FB-SPDEs in
\eq{fbspdef} of this paper. Actually, partial motivations to enhance
the unified B-SPDE in Dai~\cite{dai:newcla} (arxiv, 2011) to the
coupled FB-SPDEs in \eq{fbspdef} are from the conference
discussion~\cite{daizar:dis} during 45 minutes invited talk
presented by Zariphopoulou in ICM 2014, where the current author
claimed that the well-posedness of solution to the F-SPDE in
\cite{muszar:stopar} can be proved by the method developed in
Dai~\cite{dai:newcla} (arxiv, 2011). Besides these existing
examples, our motivations to study the coupled FB-SPDEs in
\eq{fbspdef} are also from optimal portfolio management in finance
(see, e.g., Dai~\cite{dai:meavar,dai:meahed}), and multi-channel (or
multi-valued) image regularization such as color images in computer
vision and network applications (see, e.g., Caselles {\em et
al.}~\cite{cassap:vecmed}). In this part, we also show the usages of
our unified system in \eq{fbspdef} in heat diffusions and quantum
Hall/anomalous Hall effects as two illustrative examples to support
our first aim. Mathematically, we refine a stochastic
Dirichlet-Poisson problem from heat diffusions and use stochastic
Schr$\ddot{o}$dinger equation as model for Hall effects in quantum
statistics.

It is worth to point out that the proving methodology developed in
the current paper and its early version in Dai~\cite{dai:newcla}
(arxiv, 2011) is aimed to provide a general theory and framework to
show the well-posedness of a unified general system class of the
coupled FB-SPDEs in \eq{fbspdef}. However, some specific forms of
the FB-SPDEs in \eq{fbspdef} (either in forward manner or in
backward manner) may be solved by alternative techniques, e.g., the
author in his Fields Metal awarded work (Hairer~\cite{hai:solkpz}
and ICM 2014) solves the KPZ equation by rough path technology, and
furthermore, the related rough path theory can deal with the lack of
either temporal or spatial regularity (see, e.g.,
Hairer~\cite{hai:solkpz} and reference therein).

The second aim of the paper is to prove the well-posedness of an
adapted 6-tuple weak solution $((X,Y),(V,\bar{V},\tilde{V},F))$ with
2-tuple boundary regulator $(Y,F)$ to the (possible) non-Markovian
system of coupled FB-SDEs with L\'evy jumps and skew reflections
under a given control rule $u$,
\begin{eqnarray}
&&\left\{\begin{array}{ll} \left\{\begin{array}{ll}
       X(t)&=\;\;b(t^{-},X(t^{-}),V(t^{-}),\bar{V}(t^{-}),\tilde{V}(t^{-},\cdot),u(t^{-},X(t^{-}),\cdot)dt\\
           &\;\;\;+\sigma(t^{-},X(t^{-}),V(t^{-}),\bar{V}(t^{-}),\tilde{V}(t^{-},\cdot),u(t^{-},X(t^{-})),\cdot)dW(t)\\
           &\;\;\;+\int_{{\cal Z}^{h}}\eta(t^{-},X(t^{-}),V(t^{-}),\bar{V}(t^{-}),\tilde{V}(t^{-},z),
           u(t^{-},X(t^{-})),z,\cdot)\tilde{N}(dt,dz)\\
           &\;\;\;+RdY(t),\\
       X(0)&=\;\;x,\\
       Y_{i}(t)&=\;\;\int_{0}^{t}I_{D_{i}}(X(s))dY_{i}(s);
      \end{array}
\right.\\
\left\{\begin{array}{ll}
        V(t)&=\;\;c(t^{-},X(t^{-}),V(t^{-}),\bar{V}(t^{-}),\tilde{V}(t^{-},\cdot),u(t^{-},X(t^{-}),\cdot)dt\\
             &\;\;\;-\alpha(t^{-},X(t^{-}),V(t^{-}),\bar{V}(t^{-}),\tilde{V}(t^{-},\cdot),u(t^{-},X(t^{-})),\cdot)dW(t)\\
             &\;\;\;-\int_{{\cal Z}^{h}}\zeta(t^{-},X(t^{-}),V(t^{-}),\bar{V}(t^{-}),\tilde{V}(t^{-},z),
             u(t^{-},X(t^{-})),z,\cdot)\tilde{N}(dt,dz)\\
             &\;\;\;-SdF(t),\\
        V(T)&=\;\;H(X(T),\cdot),\\
        F_{i}(t)&=\;\;\int_{0}^{t}I_{\bar{D}_{i}}(V(s))dF_{i}(s).
      \end{array}
\right.
\end{array}
\right.
\elabel{bsdehjb}
\end{eqnarray}
In \eq{bsdehjb}, $X$ is a $p$-dimensional process governed by the
F-SDE with skew reflection matrix $R$ and $V$ is a $q$-dimensional
process governed by the B-SDE with skew reflection matrix $S$.
Furthermore, $Y$ can increase only when $X$ is on a boundary $D_{i}$
with $i\in\{1,...,b\}$ and $F$ can increase only when $V$ is on a
boundary $\bar{D}_{i}$ with $\in\{1,...,\bar{b}\}$, where $b$ and
$\bar{b}$ are two nonnegative integers.
Both $Y$ and $F$ are the regulating processes with possible jumps to
push $X$ and $V$ back into the state spaces $D$ and $\bar{D}$
respectively. They are parts of the 6-tuple solution to \eq{bsdehjb}
and determined by solution pairs to the well-known Skorohod problem
(see, e.g., Dai~\cite{dai:broapp}, Dai and Dai~\cite{daidai:heatra},
or Section~\ref{gfbproof} of the current paper for such a
definition). Thus, we call them as Skorohod regulators (see,
Figure~\ref{reflection} for such an example).
Note that, comparing with the unified system in \eq{fbspdef}, the
coefficients appeared in \eq{bsdehjb} do not contain any partial
derivative operator but the FB-SDEs themselves involve skew boundary
reflections. The proof for the well-posedness of an adapted 6-tuple
weak solution to the FB-SDEs is based on two general conditions. The
first one is a general completely-${\cal S}$ condition (see, e.g.,
Dai~\cite{dai:broapp}, Dai and Dai~\cite{daidai:heatra}, and
Figure~\ref{reflection} for an illustration). The non-uniqueness of
solution to an associated Skorohod problem under this condition is
one of the major difficulties in the proof.
\begin{figure}[tbh]
\begin{center}
\epsfxsize=3.5in\epsfbox{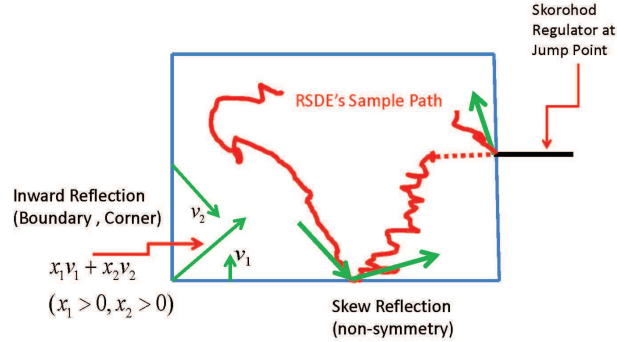}
\caption{\small Skew and Inward Reflection with Skorohod Regulator
under Completely-${\cal S}$ Condition} \label{reflection}
\end{center}
\end{figure}
The second one is the generalized linear growth and Lipschitz
conditions, where the conventional growth and Lipschitz constant is
replaced by a possible unbounded but mean-squarely integrable
adapted stochastic process (see, e.g.,
Dai~\cite{dai:meavar,dai:meahed}). In particular, if the
completely-${\cal S}$ condition becomes more strict, e.g., with
additional requirements that the spectral radii in certain sense for
both reflection matrices are strictly less than the unity, a unique
adapted 6-tuple strong solution will be concerned.

Concerning coupled FB-SDEs, it motivates a hot research area (see,
e.g., $\emptyset$ksendal {\em et al.}~\cite{okssul:stohjb} about the
discussion of coupled FB-SDEs with no boundary reflection, Karatzas
and Li~\cite{karli:bsdapp} about the study of Brownian motion driven
B-SDE with reflection, and references therein). However, to our best
knowledge, the coupled system in \eq{bsdehjb} with double skew
reflection matrices and the well-posedness study in terms of an
adapted 6-tuple weak solution with L\'evy jumps and under a general
completely-${\cal S}$ condition through the Skorohod problem are new
and for the first time in this area.

The third aim of the paper involves two folds. On the one hand, we
use the 4-tuple solution to the coupled FB-SPDEs in \eq{fbspdef} to
obtain an adapted 6-tuple solution to the system in \eq{bsdehjb}; On
the other hand, we use the obtained adapted 6-tuple solution to
determine a Pareto optimal Nash equilibrium policy process to a
non-zero-sum SDG problem in \eq{gameopto}, which is newly formulated
by the FB-SDEs in \eq{bsdehjb}. In this game, there are $q$-players
and each player $l\in\{1,...,q\}$ has his own value function
$V^{u}_{l}$ subject to the system in \eq{bsdehjb} under an
admissible control policy $u$. Every player $l$ chooses an optimal
policy to maximize his own value function over an admissible policy
set ${\cal C}$ while the summation of all value functions is also
maximized, i.e.,
\begin{eqnarray}
&&\sup_{u\in{\cal C}}V^{u}_{l}(0)=V^{u^{*}}_{l}(0) \elabel{gameopto}
\end{eqnarray}
for each $l\in\{0,1,...,q\}\}$, where,
\begin{eqnarray}
&&V^{u}_{0}(t)=\sum_{l=1}^{q}V^{u}_{l}(t). \elabel{gameoptoI}
\end{eqnarray}
Note that the total value function $V^{u}_{0}(0)$ does not have to
be a constant (e.g., zero), or in other words, the game is not
necessarily a zero-sum one.

The contribution and literature review of the study associated with
the game in \eq{gameopto}-\eq{gameoptoI} for the third aim can be
summarized as follows. One of the important solution methods for SDE
based optimal control is the dynamic programming. In general, this
method is related to a special case of the unified system in
\eq{fbspdef} (or its earlier unified B-SPDE form in
Dai~\cite{dai:newcla} (arxiv, 2011)), e.g., the specific B-SPDE with
$q=1$ (called stochastic HJB equation) in Peng~\cite{pen:stoham}
with no jumps and $\emptyset$ksendal {\em et
al.}~\cite{okssul:stohjb} with jumps. Here, we extend the
discussions in Peng~\cite{pen:stoham} and $\emptyset$ksendal {\em et
al.}~\cite{okssul:stohjb} to a system of generalized coupled
forward-backward oriented stochastic HJB equations with jumps
corresponding to the case that $q>1$. More importantly, this system
provides an effective way to resolve a non-zero-sum SDG problem with
jumps and general number of $q$ players, which subjects to a
non-Markovian system of coupled FB-SDEs with L\'evy jumps and skew
reflections (see, e.g., Figure~\ref{gameI} for such a game platform
(partially adapted from Dai~\cite{dai:traneu})).
\begin{figure}[tbh]
\begin{center}
\epsfxsize=3.5in\epsfbox{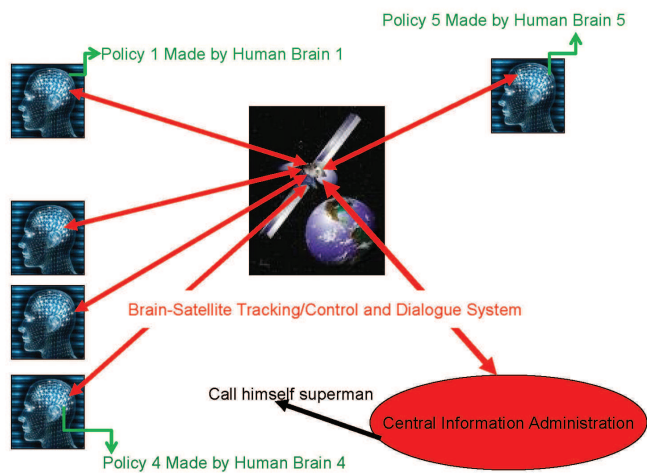}
\caption{\small A $5$-player game platform based on brain and
satellite communication} \label{gameI}
\end{center}
\end{figure}
By a solution to the FB-SPDEs in \eq{fbspdef}, we determine a
solution to the FB-SDEs in \eq{bsdehjb} under a given control rule
and then obtain a Pareto optimal Nash equilibrium policy process to
the non-zero-sum SDG problem in \eq{gameopto}. Note that, the
concept and technique concerning the non-zero-sum SDG and Pareto
optimality used in this paper is refined and generalized from
Dai~\cite{dai:optrat} and Karatzas and Li~\cite{karli:bsdapp}.

The fourth aim of the paper involves three folds. First,we study
some queueing networks (see, e.g.,  Figure~\ref{queueI})
\begin{figure}[tbh]
\begin{center}
\epsfxsize=3.5in\epsfbox{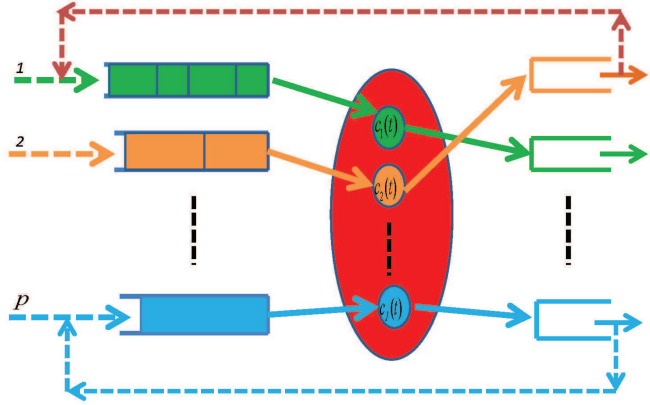}
\caption{\small A queueing network system with $p$-job classes}
\label{queueI}
\end{center}
\end{figure}
whose dynamics (e.g., queue length process) is governed by specific
forms of the FB-SDEs in \eq{bsdehjb}. These forms can be a L\'evy
driven SDE, a $p$-dimensional reflecting Brownian motion (RBM) (see,
e.g., Dai~\cite{dai:broapp}, Dai and Dai~\cite{daidai:heatra}, Dai
and Jiang~\cite{daijia:stoopt}), or a reflecting diffusion with
regime switching (RDRS) (see, e.g., Dai~\cite{dai:optrat}). The
reflecting diffusion is the functional limit of a sequence of
physical queueing processes under diffusive scaling, a general
completely-${\cal S}$ boundary reflection constraint, and a
well-known heavy traffic condition (an analogous treatment as the
one for ``infinite constant" in the KPZ equation (see, e.g.,
Hairer~\cite{hai:solkpz})). In reality, the characteristics of
L\'evy driven networks may be used to model or approximate more
general batch-arrival and batch-service queueing networks. Second,
we discuss how to use the queueing systems and their associated
reflecting diffusion approximations to motivate the SDG problem. The
criterion for the players in the game can be the queue length based
performance optimization ones or queueing related cost/profit
optimization ones. Third, we study the applications of the FB-SPDEs
presented by \eq{fbspdef} in the queueing networks. There are two
types of equations involved. One is the Kolmogorov's equation or
Fokker-Planck's formula oriented PDEs/SPDEs, which are corresponding
to the distributions of queueing length processes under given
network control rules. This type of equations are mainly used to
estimate the performance measures of the queueing networks. Another
type of equations are the HJB equation oriented PDES/SPDES, which
are mainly used to obtain optimal control rules over the set of
admissible strategies for the queueing networks.

The remainder of the paper is organized as follows. In
Section~\ref{uniexist}, we introduce suitable functional topological
space and state conditions required for our main theorems to
guarantee the well-posedness of an adapted 4-tuple strong solution
to the unified system of coupled FB-SPDEs in \eq{fbspdef}. In
Section~\ref{fbsdes}, we study the unified system of coupled FB-SDEs
with L\'evy jumps and skew reflections in \eq{bsdehjb} and present
the well-posedness theorem. In particular, we establish the solution
connection between the FB-SPDEs and the FB-SDEs. In
Section~\ref{queuesdg}, we formally formulate the FB-SDEs based SGE
problem in \eq{gameopto} and determine the Pareto optimal Nash
equilibrium policy process by a system of generalized stochastic HJB
equations (a particular form of the coupled FB-SPDEs). Related
applications in queueing networks are also discussed. Finally, in
Sections~\ref{uniextproof}-\ref{gtheoremproof}, we develop theory to
prove our main theorems.

\section{The Unified System of Coupled FB-SPDEs with L\'evy Jumps}\label{uniexist}

First of all, let $(\Omega,{\cal F},P)$ be a fixed complete
probability space. Then, we define a standard $d$-dimensional
Brownian motion $W\equiv\{W(t),t\in[0,T]\}$ for a given
$T\in[0,\infty)$ with $W(t)=(W_{1}(t),...,W_{d}(t))'$ and a
$h$-dimensional general L\'evy pure jump process (or special
subordinator) $L\equiv\{L(t),t\in[0,T]\}$ with
$L(t)\equiv(L_{1}(t),...,L_{h}(t))'$ on the space (see, e.g.,
Applebaum~\cite{app:levpro}, Bertoin~\cite{ber:levpro}, and
Sato~\cite{sat:levpro}). Note that the prime appeared in this paper
is used to denote the corresponding transpose of a matrix or a
vector. Furthermore, $W$, $L$, and their components are supposed to
be independent of each other. For each
$\lambda=(\lambda_{1},...\lambda_{h})'>0$, which is called a
reversion rate vector in many applications, we let $L(\lambda
s)=(L_{1}(\lambda_{1}s),...,L_{h}(\lambda_{h}s))'$. Then, we denote
a filtration by $\{{\cal F}_{t}\}_{t\geq 0}$ with ${\cal
F}_{t}\equiv\sigma\{{\cal G},W(s),L(\lambda s): 0\leq s\leq t\}$ for
each $t\in[0,T]$, where ${\cal G}$ is $\sigma$-algebra independent
of $W$ and $L$. In addition, let $I_{A}(\cdot)$ be the index
function over the set $A$ and $\nu_{i}$ for each $i\in\{1,...,h\}$
be a L\'{e}vy measure. Then, we use $N_{i}((0,t]\times
A)\equiv\sum_{0<s\leq t} I_{A}(L_{i}(s)-L_{i}(s^{-}))$ to denote a
Poisson random measure with a deterministic, time-homogeneous
intensity measure $ds\nu_{i}(dz_{i})$. Thus, each $L_{i}$ can be
represented by (see, e.g., Theorem 13.4 and Corollary 13.7 in pages
237 and 239 of Kallenberg~\cite{kal:foumod})
\begin{eqnarray}
&&L_{i}(t)=a_{i}(t)+\int_{(0,t]}\int_{{\cal
Z}}z_{i}N_{i}(\lambda_{i}ds,dz_{i}), \;t\geq 0. \elabel{subordrep}
\end{eqnarray}
For convenience, we take the constant $a_{i}$ to be zero.

In the subsequent two subsections, we first study the unified system
in \eq{fbspdef}, which is real-valued with a closed position
parametric domain. Then, we extend the discussion to a
complex-valued system with an open position parametric domain.

\subsection{The Real-Valued System with Closed Position Parametric Domain}\label{boundedd}

We let $D\in R^{p}$ be a closed position parametric domain and use
$C^{k}(D,R^{l})$ for each $k\in{\cal N}$ and $l\in\{r,q\}$ to denote
the Banach space of all functions $f$ having continuous derivatives
up to the order $k$ with the uniform norm for each $f$ in this
space,
\begin{eqnarray}
&&\|f\|_{C^{k}(D,l)}=\max_{c\in\{0,1,...,k\}}\max_{j\in\{1,...,r(c)\}}
\sup_{x\in D}\left|f^{(c)}_{j}(x)\right|. \elabel{fckupd}
\end{eqnarray}
The $r(c)$ in \eq{fckupd} for each $c\in\{0,1,...,k\}$ is the total
number of the partial derivatives of the order $c$
\begin{eqnarray}
&&f^{(c)}_{r,i_{1}...i_{p}}(x)=\frac{\partial^{c}f_{r}(x)}{\partial
x_{1}^{i_{1}}...\partial x_{p}^{i_{p}}} \elabel{difoperatorI}
\end{eqnarray}
with $i_{l}\in\{0,1,...,c\}$, $l\in\{1,...,p\}$, $r\in\{1,...,l\}$,
and $i_{1}+...+i_{p}=c$. Here, we remark that, whenever the partial
derivative on the boundary ${\partial D}$ is concerned, it is
defined in a one-side manner. In addition, let
\begin{eqnarray}
f_{i_{1},...,i_{p}}^{(c)}&\equiv&(f_{1,i_{1},...,i_{p}}^{(c)},...,
f_{q,i_{1},...,i_{p}}^{(c)}),
\elabel{difoperatoro}\\
f^{(c)}(x)&\equiv&(f_{1}^{(c)}(x),...,f_{r(c)}^{(c)}(x)),
\elabel{difoperator}
\end{eqnarray}
where each $j\in\{1,...,r(c)\}$ corresponds to a $p$-tuple
$(i_{1},...,i_{p})$ and a $r\in\{1,...,l\}$. Then, we use
$C^{\infty}(D,R^{l})$ to denote the Banach space
\begin{eqnarray}
&&C^{\infty}(D,R^{l})\equiv
\left\{f\in\bigcap_{c=0}^{\infty}C^{c}(D,R^{l}),
\|f\|_{C^{\infty}(D,l)}<\infty\right\}, \elabel{cinfity}
\end{eqnarray}
where
\begin{eqnarray}
&&\|f\|^{2}_{C^{\infty}(D,q)}=\sum_{c=0}^{\infty}\xi(c)
\|f\|^{2}_{C^{c}(D,l)} \elabel{inftynorm}
\end{eqnarray}
for some discrete function $\xi(c)$ in terms of $c\in\{0,1,2,...\}$,
which is fast decaying in $c$. For convenience, we take
$\xi(c)=\frac{1}{((c^{10})!)(\eta(c)!)e^{c}}$ with
\begin{eqnarray}
&&\eta(c)=[\max\{|x_{1}|+...+|x_{p}|,x\in D\}]^{c}, \nonumber
\end{eqnarray}
where the notation $[]$ denotes the summation of the unity and the
integer part of a real number.

Next, let $L^{2}_{{\cal F}}([0,T],C^{\infty}(D;R^{l}))$ denote the
set of all $R^{l}$-valued (or called $C^{\infty}(D;R^{l})$-valued)
measurable random vector-field processes $Z(t,x)$ adapted to
$\{{\cal F}_{t},t\in[0,T]\}$ for each $x\in D$, which are in
$C^{\infty}(D,R^{l})$ for each fixed $t\in[0,T]$), such that
\begin{eqnarray}
&&E\left[\int_{0}^{T}\|Z(t)\|^{2}_{C^{\infty}(D,l)}dt\right]<\infty.
\elabel{adaptednormI}
\end{eqnarray}
In particular, let $L^{2}_{{\cal
G}_{l}}(\Omega,C^{\infty}(D;R^{l}))$ with $l\in\{r,q\}$ denote the
set of all $R^{l}$-valued random vector-fields $\zeta(x)$ that are
${\cal G}_{l}$-measurable for each $x\in D$ and satisfy
\begin{eqnarray}
&&\|\zeta\|^{2}_{L^{2}_{{\cal G}}(\Omega,C^{\infty}(D,R^{l}))}\equiv
E\left[\|\zeta\|^{2}_{C^{\infty}(D,l)}\right]<\infty,
\elabel{initialx}
\end{eqnarray}
where ${\cal G}_{r}={\cal G}$ and ${\cal G}_{q}={\cal F}_{T}$. In
addition, let $L^{2}_{p}([0,T]\times {\cal Z}^{h},$
$C^{\infty}(D,R^{l\times h}))$ be the set of all $R^{l\times
h}$-valued random vector-field processes denoted by
$\tilde{V}(t,x,z)=$ $(\tilde{V}_{1}(t,x,z_{1}),$ $...,$
$\tilde{V}_{h}(t,x,z_{h}))$, which are predictable for each $x\in D$
and $z\in{\cal Z}^{h}$ and are endowed with the norm
\begin{eqnarray}
&&E\left[\sum_{i=1}^{h}\int_{0}^{T}\int_{{\cal Z}}
\left\|\tilde{V}_{i}(t,z_{i})\right\|^{2}_{C^{\infty}(D,l)}
\nu_{i}(dz_{i})dt\right]<\infty. \elabel{normsI}
\end{eqnarray}
Thus, we can define
\begin{eqnarray}
{\cal Q}^{2}_{{\cal F}}([0,T]\times D) &\equiv&\;\;\;L^{2}_{{\cal
F}}([0,T],C^{\infty}(D,R^{r}))\elabel{bunisoI}\\
&&\times L^{2}_{{\cal F}}([0,T],C^{\infty}(D,R^{q}))
\nonumber\\
&&\times L^{2}_{{\cal F},p}([0,T],C^{\infty}(D,R^{q\times d}))
\nonumber\\
&&\times L^{2}_{p}([0,T]\times{\cal Z}^{h},C^{\infty}(D,R^{q\times
h})). \nonumber
\end{eqnarray}
Finally, let
\begin{eqnarray}
&&L^{2}_{\nu}({\cal Z}^{h},C^{c}(D,R^{q\times h}))
\elabel{lvrqh}\\
&&\equiv\left\{\tilde{v}: {\cal Z}^{h}\rightarrow C^{c}(D,R^{q\times
h}),\;\sum_{i=1}^{h}\int_{{\cal Z}}
\left\|\tilde{v}_{i}(z_{i})\right\|^{2}_{C^{c}(D,q)}
\nu_{i}(dz_{i})<\infty\right\} \nonumber
\end{eqnarray}
that is endowed with the norm
\begin{eqnarray}
&&\|\tilde{v}\|_{\nu,c}^{2}\equiv\sum_{i=1}^{h}\int_{{\cal Z}}
\left\|\tilde{v}_{i}(z_{i})\right\|^{2}_{C^{c}(D,q)}
\lambda_{i}\nu_{i}(dz_{i}) \elabel{bunisoIn}
\end{eqnarray}
for any $\tilde{v}\in L^{2}_{\nu}({\cal Z}^{h},C^{c}(D,R^{q\times
h}))$ and $c\in\{0,1,...,\infty\}$. Furthermore, define
\begin{eqnarray}
{\cal V}^{\infty}(D)
&\equiv&C^{\infty}(D,R^{r})
\elabel{primitivespace}\\
&&\times C^{\infty}(D,R^{q})
\nonumber\\
&&\times C^{\infty}(D,R^{q\times d})
\nonumber\\
&&\times \bar{L}^{2}_{\nu}({\cal Z}^{h},C^{\infty}(D,R^{q\times
h})). \nonumber
\end{eqnarray}

In the sequel, we let $\|A\|$ be the largest absolute value of
entries (or components) of the given matrix (or vector) $A$.
Furthermore, for each $s\in[0,T]$ and $z\in{\cal Z}^{h}$, we let
\begin{eqnarray}
\tilde{N}(\lambda ds,dz)&=&(\tilde{N}_{1}(\lambda_{1}ds,dz_{1}),...,
\tilde{N}_{h}(\lambda_{h}ds,dz_{h}))', \elabel{barvvII}
\end{eqnarray}
where
\begin{eqnarray}
\tilde{N}_{i}(\lambda_{i}ds,dz_{i})
=N_{i}(\lambda_{i}ds,dz_{i})-\lambda_{i}ds\nu_{i}(dz_{i})
\elabel{centeredpm}
\end{eqnarray}
for each $i\in\{1,...,h\}$. Then, we impose some conditions to
guarantee the unique existence of an adapted 4-tuple strong solution
to the unified system in \eq{fbspdef}.

First, for each ${\cal A}\in\{{\cal L},\bar{\cal L},{\cal
J},\bar{{\cal J}}\}$, every $c\in\{0,1,2,...,\}$, and any
$(u^{i},v^{i},\bar{v}^{i},\tilde{v}^{i})\in{\cal V}^{\infty}(D)$
with $i\in\{1,2\}$, we define
\begin{eqnarray}
&&\Delta{\cal
A}^{(c)}(s,x,u^{1},v^{1},\bar{v}^{1},\tilde{v}^{1},u^{2},v^{2},
\bar{v}^{2},\tilde{v}^{2})\nonumber\\
&&\equiv {\cal
A}^{(c)}(s,x,u^{1},v^{1},\bar{v}^{1},\tilde{v}^{1})-{\cal
A}^{(c)}(s,x,u^{2},v^{2},\bar{v}^{2},\tilde{v}^{2}). \nonumber
\end{eqnarray}
Then, we assume that the generalized local Lipschitz condition is
true almost surely (a.s.),
\begin{eqnarray}
&&\left\|\Delta{\cal
A}^{(c+l+o)}(s,x,u^{1},v^{1},\bar{v}^{1},\tilde{v}^{1},u^{2},v^{2},
\bar{v}^{2},\tilde{v}^{2})\right\|
\elabel{blipschitz}\\
&&\leq
K_{D,c}\left(\|u^{1}-u^{2}\|_{C^{k+c}(D,r)}+\|v^{1}-v^{2}\|_{C^{k+c}(D,q)}\right.
\nonumber\\
&&\;\;\;\;\;\;\;\;\;\;\;\;\;\left.+\|\bar{v}^{1}-\bar{v}^{2}\|_{C^{k+c}(D,qd)}
+\|\tilde{v}^{1}-\tilde{v}^{2}\|_{\nu,k+c}\right). \nonumber
\end{eqnarray}
Note that $K_{D,c}$ in \eq{blipschitz} with each $c\in\{0,1,2,...\}$
is a nonnegative constant. It depends on the domain $D$ and the
differential order $c$ and may be unbounded as $c\rightarrow\infty$
and $D\rightarrow R^{p}$. $l\in\{0,1,2\}$ denotes the $l$th order of
partial derivative of $\Delta{\cal
A}^{(c)}(s,x,u,v,\bar{v},\tilde{v})$ in time variable $t$.
$o\in\{0,1,2\}$ denotes the $o$th order of partial derivative of
$\Delta{\cal A}^{(c+l)}(s,x,u,v,\bar{v},\tilde{v})$ in terms of a
component of $u$, $v$, $\bar{v}$, or $\tilde{v}$. Furthermore, for
each ${\cal A}\in\{{\cal I},\bar{{\cal I}}\}$, we suppose that
\begin{eqnarray}
&&\sum_{i=1}^{h}\int_{{\cal Z}}\left\|\Delta{\cal
A}_{i}^{(c+l+o)}(s,x,u^{1},v^{1},\bar{v}^{1},\tilde{v}^{1},u^{2},v^{2},
\bar{v}^{2},\tilde{v}^{2},z_{i})\right\|^{2}\lambda_{i}\nu_{i}(dz_{i})
\elabel{nlipo}\\
&&\leq
K_{D,c}\left(\|u^{1}-u^{2}\|^{2}_{C^{k+c}(D,r)}+\|v^{1}-v^{2}\|^{2}_{C^{k+c}(D,q)}\right.
\nonumber\\
&&\;\;\;\;\;\;\;\;\;\;\;\;\;\left.+\|\bar{v}^{1}-\bar{v}^{2}\|^{2}_{C^{k+c}(D,qd)}
+\|\tilde{v}^{1}-\tilde{v}^{2}\|^{2}_{\nu,k+c}\right), \nonumber
\end{eqnarray}
where ${\cal A}_{i}$ is the $i$th column of ${\cal A}$.

Second, for each ${\cal A}\in\{{\cal L},\bar{\cal L},{\cal
J},\bar{{\cal J}}\}$, every $c\in\{0,1,2,...,\}$, and any
$(u,v,\bar{v},\tilde{v})\in{\cal V}^{\infty}(D)$, we suppose that
the generalized local linear growth condition holds
\begin{eqnarray}
&&\left\|{\cal A}^{(c+l+o)}(s,x,u,v,\bar{v},\tilde{v})\right\|
\elabel{fblipschitzoI}\\
&&\leq
K_{D,c}\left(\delta_{0c}+\|u\|_{C^{k+c}(D,r)}+\|v\|_{C^{k+c}(D,q)}\right.
\nonumber\\
&&\;\;\;\;\;\;\;\;\;\;\;\;\;
\left.+\|\bar{v}\|_{C^{k+c}(D,qd)}+\|\tilde{v}\|_{\nu,k+c}\right),
\nonumber
\end{eqnarray}
where, $\delta_{0c}=1$ if $c=0$ and $\delta_{0c}=0$ if $c>0$.
Similarly, for each ${\cal A}\in\{{\cal I},\bar{{\cal I}}\}$, we
suppose that
\begin{eqnarray}
&&\sum_{i=1}^{h}\int_{{\cal Z}}\left\|{\cal
A}_{i}^{(c+l+o)}(s,x,u,v,\bar{v},\tilde{v},z_{i})\right\|^{2}\lambda_{i}\nu(dz_{i})
\elabel{nlipoI}\\
&&\leq
K_{D,c}\left(\delta_{0c}+\|u\|^{2}_{C^{k+c}(D,r)}+\|v\|^{2}_{C^{k+c}(D,q)}\right.
\nonumber\\
&&\;\;\;\;\;\;\;\;\;\;\;\;\;
\left.+\|\bar{v}\|^{2}_{C^{k+c}(D,qd)}+\|\tilde{v}\|^{2}_{\nu,k+c}\right).
\nonumber
\end{eqnarray}

Then, we can state our main theorem of this subsection as follows.
\begin{theorem}\label{bsdeyI}
Suppose that $(G,H)\in L^{2}_{{\cal G}}(\Omega,
C^{\infty}(D;R^{r}))\times L^{2}_{{\cal F}_{T}}(\Omega,
C^{\infty}(D;R^{q}))$ and conditions in \eq{blipschitz}-\eq{nlipoI}
are true. Furthermore, assume that each ${\cal A}\in\{{\cal
L},\bar{{\cal L}},{\cal J},\bar{{\cal J}}, {\cal I},\bar{{\cal
I}}\}$ is $\{{\cal F}_{t}\}$-adapted for every fixed $x\in D$,
$z\in{\cal Z}^{h}$, and any given $(u,v,\bar{v},\tilde{v})\in{\cal
V}^{\infty}(D)$ with
\begin{eqnarray}
&&{\cal L}(\cdot,x,0)\in L^{2}_{{\cal
F}}\left([0,T],C^{\infty}(D,R^{r})\right),
\elabel{flipic}\\
&&{\cal J}(\cdot,x,0)\in L^{2}_{{\cal
F}}\left([0,T],C^{\infty}(D,R^{r\times d})\right),
\elabel{flipicADI}\\
&&{\cal I}(\cdot,x,0,\cdot)\in L^{2}_{{\cal F}}\left([0,T]\times
R_{+}^{h},C^{\infty}(D,R^{r\times h})\right),
\elabel{flipicADII}\\
&&\bar{{\cal L}}(\cdot,x,0)\in L^{2}_{{\cal
F}}\left([0,T],C^{\infty}(D,R^{q})\right),
\elabel{blipic}\\
&&\bar{{\cal J}}(\cdot,x,0)\in L^{2}_{{\cal
F}}\left([0,T],C^{\infty}(D,R^{q\times d})\right),
\elabel{blipicADI}\\
&&\bar{{\cal I}}(\cdot,x,0,\cdot)\in L^{2}_{{\cal
F}}\left([0,T]\times R_{+}^{h},C^{\infty}(D,R^{q\times h})\right).
\elabel{blipicADII}
\end{eqnarray}
Then, there exists a unique adapted 4-tuple strong solution to the
system in \eq{fbspdef}, i.e.,
\begin{eqnarray}
&&(U,V,\bar{V},\tilde{V})\in {\cal Q}^{2}_{{\cal F}}([0,T]\times D),
\elabel{buniso}
\end{eqnarray}
and $(U,V)(\cdot,x)$ is c\`adl\`ag for each $x\in D$ almost surely
(a.s.).
\end{theorem}

The proof of Theorem~\ref{bsdeyI} is provided in
Section~\ref{uniextproof}.

\subsection{The Complex-Valued System with Open Position Parametric Domain}

In this subsection, we generalize the study in
Subsection~\ref{boundedd} to the case corresponding to an open (or
partially open) position parametric domain $D$ (e.g., $R^{p}$ or
$R_{+}^{p}$). More exactly, we assume that there exists a sequence
of nondecreasing closed and connected sets
$\{D_{n},n\in\{0,1,...\}\}$ such that
\begin{eqnarray}
&&D=\bigcup_{n=0}^{\infty}D_{n}. \elabel{unionset}
\end{eqnarray}
Furthermore, let $C^{\infty}(D,\mathbb{C}^{l})$ with $l\in\{r,q\}$
be the Banach space endowed with the norm for each $f$ in the space
\begin{eqnarray}
&&\|f\|^{2}_{C^{\infty}(D,l)}\equiv\sum_{n=0}^{\infty}\xi(n+1)
\|f\|^{2}_{C^{\infty}(D_{n},l)}, \elabel{irbspacen}
\end{eqnarray}
where $\mathbb{C}^{l}$ is the $l$-dimensional complex Euclidean
space and the norm $\|f\|^{2}_{C^{\infty}(D_{n},l)}$ in
\eq{irbspacen} is interpreted in the corresponding complex-valued
sense. In addition, define $\bar{Q}^{2}_{{\cal F}}([0,\tau]\times
D)$ to be the corresponding space in \eq{bunisoI} if the terminal
time $T$ is replaced by the stopping time $\tau$ in \eq{fbspdef} and
the norm in \eq{inftynorm} is substituted by the one in
\eq{irbspacen}. Finally, we use the same way to interpret the spaces
$L^{2}_{{\cal G}}(\Omega, C^{\infty}(D;R^{r}))$ and $L^{2}_{{\cal
F}_{\tau}}(\Omega, C^{\infty}(D;R^{q}))$. Then, we have the
following theorem.
\begin{theorem}\label{infdn}
Suppose that $(G,H)\in L^{2}_{{\cal G}}(\Omega,
C^{\infty}(D;R^{r}))\times L^{2}_{{\cal F}_{\tau}}(\Omega,
C^{\infty}(D;R^{q}))$ and the system in \eq{fbspdef} satisfies the
conditions in \eq{blipschitz}-\eq{nlipoI} over $D_{n}$ for each
$n\in\{0,1,...\}$ with associated (local) linear growth and Lipshitz
constant $K_{D_{n},c}$. Furthermore, assume that each ${\cal
A}\in\{{\cal L},\bar{{\cal L}},{\cal J},\bar{{\cal J}}, {\cal
I},\bar{{\cal I}}\}$ is $\{{\cal F}_{t}\}$-adapted for every fixed
$x\in D$, $z\in{\cal Z}^{h}$, and any given
$(u,v,\bar{v},\tilde{v})\in{\cal V}^{\infty}(D)$ with conditions in
\eq{flipic}-\eq{blipicADII} being true. Then, the system in
\eq{fbspdef} has a unique adapted 4-tuple strong solution
\begin{eqnarray}
&&(U,V,\bar{V},\tilde{V})\in\bar{{\cal Q}}^{2}_{{\cal
F}}([0,\tau]\times D), \elabel{obuniso}
\end{eqnarray}
and $(U,V)(\cdot,x)$ is c\`adl\`ag for each $x\in D$ a.s.
\end{theorem}

The proof of Theorem~\ref{infdn} is provided in
Section~\ref{uniextproof}.

\subsection{Illustrative Examples}

In this subsection, we present two illustrative examples related to
heat diffusions and quantum Hall/anomalous Hall effects.
Mathematically, the heat diffusions are modeled as a stochastic
Dirichlet-Poisson problem and the Hall effects are presented by a
stochastic Schr$\ddot{o}$dinger equation. Examples related to
queueing systems and SDGs will be presented in
Section~\ref{queuesdg} after studying the system of coupled FB-SDEs
in \eq{bsdehjb}.

\vskip 0.4cm
\begin{itemize}
\item {\bf Heat Diffusions: Stochastic Dirichlet-Poisson Problem}
\end{itemize}

\vskip 0.4cm For this example, we consider a special B-SPDE (a
specific backward part of the system in \eq{fbspdef}). More
precisely, the associated partial differential operators are given
by
\begin{eqnarray}
\bar{{\cal
L}}(t,x,U,V,\bar{V},\tilde{V})&\equiv&-g(t,x)+\frac{1}{2}\sum_{j=1}^{p}\frac{\partial^{2}
V(t,x)}{\partial x^{2}_{j}},
\nonumber\\
\bar{{\cal J}}(t,x,U,V,\bar{V},\tilde{V})&\equiv&0,
\nonumber\\
\bar{{\cal I}}(t,x,U,V,\bar{V},\tilde{V},z)&\equiv&0, \nonumber
\end{eqnarray}
where $V$ is single-dimensional (i.e., $q=1$) and $g(t,x)$ is some
integrable function. Then, we can obtain the corresponding B-SPDE
with jumps as follows,
\begin{eqnarray}
V(t,x)&=&H(x)
+\int_{t}^{\tau}\left(-g(t,x)+\frac{1}{2}\sum_{j=1}^{p}\frac{\partial^{2}
V(t,x)}{\partial x^{2}_{j}}\right)ds
\elabel{dirichletpoisson}\\
&&\;\;\;\;\;\;\;\;\;-\int_{t}^{\tau}\bar{V}(s,x)dW(s)
\nonumber\\
&&\;\;\;\;\;\;\;\;\;-\int_{t}^{\tau}\int_{z>0}\tilde{V}(s^{-},x,z)
\tilde{N}(\lambda ds,dz).
\nonumber
\end{eqnarray}
When the terminal value $H(x)$ is a boundary based condition over
$D$, we will call the related resolution problem {\it a stochastic
Dirichlet-Poisson problem with jumps}, which is a randomized general
form of the classical Dirichlet-Poisson problem (see, e.g., the
definition of a classic case in
$\emptyset$ksendal~\cite{oks:stodif}). An explanation about how to
use this problem to estimate the inner or surface temperature of
certain material or an object (e.g., Sun) is displayed in
Figure~\ref{dirichlet}.
\begin{figure}[tbh]
\begin{center}
\epsfxsize=4.5in\epsfbox{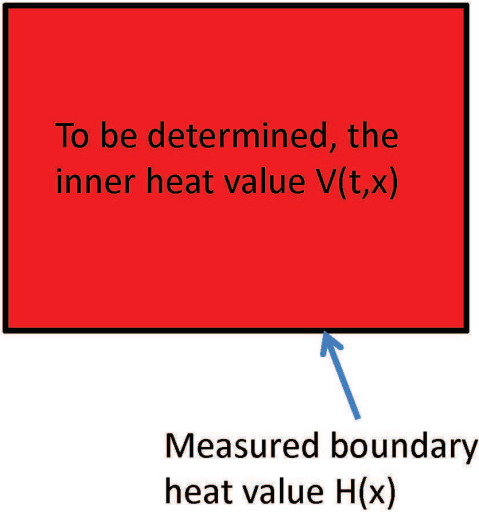}
\caption{\small A heat diffusion system by a solution to the
Dirichlet-Poisson problem}
\label{dirichlet}
\end{center}
\end{figure}
Physically, the randomized heat equation in \eq{dirichletpoisson} is
derived from a particle system just like its classic counterpart,
which can be modeled by a diffusion process as follows,
\begin{eqnarray}
&&dX_{t}=b(X_{t})dt+\sigma(X_{t})dW(t). \elabel{adif}
\end{eqnarray}
Now, we use $\tau_{D}^{x}$ to denote the first exit time (a stopping
time) of the process $X_{t}$ from the time-space domain $[0,T]\times
D$, i.e.,
\begin{eqnarray}
&&\tau_{D}^{x}=inf\{t>0,(t,X_{t})\notin [0,T]\times D\},
\elabel{adifstop}
\end{eqnarray}
where the upper index $x$ means that $X_{t}$ starts from $x\in D$.
Then, we can impose a terminal-boundary condition with
$\tau=\tau_{D}^{x}$ as follows,
\begin{eqnarray}
&&\lim_{t\rightarrow\tau_{D}^{x}(\omega)}V(t,X_{t})
=H(\tau_{D}^{x}(\omega)) =H_{D}(x,\omega)\;\;\;\;\;a.s.\;\;Q^{x}
\elabel{imbcon}
\end{eqnarray}
for all $(t,x)\in [0,T]\times D$. In the case that $H_{D}(x,\omega)$
is a random variable independent of $x$, the required smooth
condition for the well-posedness of the B-SPDE in
\eq{dirichletpoisson} is satisfied. In general, $H_{D}(x,\omega)$
can be approximated by sufficiently smooth function in $x$ as
required.

\vskip 0.4cm
\begin{itemize}
\item {\bf Hall Effects: Stochastic Schr$\ddot{o}$dinger Equation}
\end{itemize}

\vskip 0.4cm In quantum physics and statistical mechanics, some
phenomena such as Hall/anomalous Hall effects (see, e.g.,
Hall~\cite{hal:newact}, Karplus and Luttinger~\cite{karlut:haleff},
and the summarized description at Wikipedia website) are major
concerns. By the definition of Hall effect, the movements of quantum
particles (e.g., electrons) within a semiconduct/superconduct are
along regular paths (see Figure~\ref{hallano} for such an example)
if the Lorentz force generated by an external magnetic field with a
perpendicular component is imposed.
\begin{figure}[tbh]
\begin{center}
\epsfxsize=3.5in\epsfbox{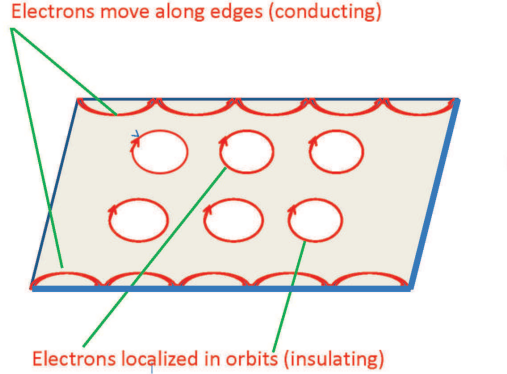}
\caption{\small Edge states carry the current}
\label{hallano}
\end{center}
\end{figure}
When this phenomenon happens, the collisions of quantum particles
will be significantly reduced and the performance of the
semiconduct/superconduct will be largely improved. However, in a
real application, imposing an external magnetic field is frequently
expensive. Thus, people try to develop some magnetic material based
semiconduct/superconduct in order that the Hall effect happens
naturally (see, e.g., Karplus and Luttinger~\cite{karlut:haleff},
Chang {\em et al.}~\cite{chazha:expobs}). This phenomenon is called
anomalous Hall effect.

Besides observing the Hall/anomalous Hall effects by physical
experiments (see, e.g., Hall~\cite{hal:newact} and Chang {\em et
al.}~\cite{chazha:expobs}), one can also analytically study and
simulate these effects through a Schr$\ddot{o}$dinger equation (see,
e.g., Thouless~\cite{tho:quahal}, Chai~\cite{cha:denfun,cha:theass},
and the summarized descriptions about density functional theory and
time-dependent density functional theory at Wikipedia website). The
Schr$\ddot{o}$dinger equation used in most existing studies is a
form of the Fokker-Planck's formula (see, e.g.,
$\emptyset$ksendal~\cite{oks:stodif}). Here, by taking a form of
${\cal L}$ in the forward part of the system in \eq{fbspdef}, we can
unify these Schr$\ddot{o}$dinger equations (see, e.g., Bouard and
Debussche~\cite{boudeb:stonon,boudeb:stononI},
Thouless~\cite{tho:quahal}, Chai~\cite{cha:denfun,cha:theass}) into
the {\em generalized stochastic nonlinear Schr$\ddot{o}$dinger
equation} with absorbing boundaries for each $i\in\{1,...,2p\}$),
\begin{eqnarray}
&&idV(t,x)={\cal L}(t^{-},x,V)dt+{\cal
J}(t^{-},x,V)dW(t)+\int_{z>0}{\cal I}(t^{-},x,V,z)\tilde{N}(\lambda
dt,dz), \elabel{hallbspdefhall}
\end{eqnarray}
where $V$ is single-dimensional (i.e., $q=1$), $i$ is the imaginary
number, and ${\cal L}$ is a form of the operator,
\begin{eqnarray}
{\cal
L}(t,x,V,\cdot)&=&\sum_{j=1}^{p}a_{jj}(x)\frac{\partial^{2}V(t,x)}{\partial
x_{j}^{2}}+\sum_{j=1}^{p}b_{j}(x)\frac{\partial V(t,x)}{\partial
x_{j}}+c(x,V)V(t,x). \elabel{schrodingerop}
\end{eqnarray}
Note that $c(x,V)$ in \eq{schrodingerop} is the potential, which may
depend on external temperature and/or external magnetic field. For
example, the recent discovery about the Anomalous Hall Effect (see,
e.g., Chang {\em et al.}~\cite{chazha:expobs}) is based on a lower
temperature and without imposing external magnetic field.
Furthermore, the related Schr$\ddot{o}$dinger equation based studies
can be found in Chai~\cite{cha:denfun,cha:theass}, etc.

Now, if the densities appeared in the Hall/Anomalous Hall Effects
are the target stationary distributions (i.e., the terminal values
$H(T,x)$ in \eq{fbspdef} are given), we can take $\bar{{\cal L}}$ in
the system of \eq{fbspdef} to be a form of the operator in
\eq{schrodingerop}. Then, we can find the initial and transient
distributions of quantum particles by the backward part of the
system in \eq{fbspdef}. From physical viewpoint, this study provides
insights about how to characterize and manufacture the magnetic
material based semiconductor/superconduct.

\section{The Coupled FB-SDEs with L\'evy Jumps and Skew
Reflections}\label{fbsdes}

\subsection{The Coupled FB-SDEs and Its Well-Posedness}

In this section, we suppose that the process $X$ governed by the
forward SDE in \eq{bsdehjb} lives in a state space $D$ (e.g., a
$p$-dimensional positive orthant or a $p$-dimensional rectangle).
Furthermore, let $D_{i}=\{x\in R^{p},x\cdot n_{i}=b_{i}\}$ for
$i\in\{1,...,b\}$ be the $i$th boundary face of $D$, where $b_{i}=0$
for $i\in\{1,...,p\}$, $b_{i}$ is some positive constant for
$i\in\{p+1,...,b\}$, and $n_{i}$ is the inward unit normal vector on
the boundary face $D_{i}$. For convenience, we define
$N=(n_{1},...,n_{b})$. In addition, let $R$ in \eq{bsdehjb} be a
$p\times b$ matrix with $b\in\{p,2p\}$, whose $i$th column denoted
by $p$-dimensional vector $v_{i}$ is the reflection direction of $X$
on $D_{i}$. The process $Y$ in \eq{bsdehjb} is a nondecreasing
predictable process with $Y(0)=0$ and boundary regulating property
as explained in \eq{bsdehjb}.
In queueing system, this process is called boundary idle time or
blocking process.

Similarly, we assume that $V$ takes value in a region $\bar{D}$ with
boundary face $\bar{D}_{i}=\{v\in
R^{q},v\cdot\bar{n}_{i}=\bar{b}_{i}\}$ for $i\in\{1,...,\bar{b}\}$,
where $\bar{n}_{i}$ is the inward unit normal vector on the boundary
face $\bar{D}_{i}$. For convenience, we define
$\bar{N}=(\bar{n}_{1},...,\bar{n}_{\bar{b}})$. In finance, the given
constant $\bar{b}_{i}$ is called early exercise reward. Furthermore,
$S$ in \eq{bsdehjb} is supposed to be a $q\times \bar{b}$ matrix for
a known $\bar{b}\in\{q,2q\}$. In addition, $F(\cdot)$ in
\eq{bsdehjb} is a nondecreasing predictable process with $F(0)=0$
and boundary regulating property as explained in \eq{bsdehjb}.

To guarantee the existence and uniqueness of an adapted 6-tuple weak
solution to the coupled FB-SDEs in \eq{bsdehjb}, we need to
introduce the completely-${\cal S}$ condition on the reflection
matrix $R$ (and similarly on $S$).
\begin{definite}
A $p\times p$ square matrix $R$ is called completely-${\cal S}$ if
and only if there is $x>0$ such that $\tilde{R}x>0$ for each
principal sub-matrix $\tilde{R}$ of $R$, where the vector
inequalities are to be interpreted componentwise. Furthermore, a
$p\times b$ matrix $R$ is called completely-${\cal S}$ if and only
if each $p\times p$ square sub-matrix of $N'R$ is completely-{\cal
S}.
\end{definite}

Note that the completely-${\cal S}$ condition on the reflection
matrices guarantees that the coupled FB-SDEs are of inward
reflection on each boundary and corner of the orthant or the
rectangle (see, e.g., Figure~\ref{reflection} and
Dai~\cite{dai:broapp}). Furthermore, the reflection appeared here is
called skew reflection that is a generalization of the conventional
mirror (or called symmetry) reflection.

Now, the coefficient functions given in \eq{bsdehjb} are assumed to
be $\{{\cal F}_{t}\}$-predictable and are detailed as follows,
\begin{eqnarray}
b(t,x,u)\equiv b(t,x,v,\bar{v},\tilde{v},u,\cdot)&:&[0,T]\times
R^{p}\times R^{q}\times R^{q\times d}\times R^{q\times h}\times
U\rightarrow
R^{p},\nonumber\\
\sigma(t,x,u)\equiv\sigma(t,x,v,\bar{v},\tilde{v},u,\cdot)&:&[0,T]\times
R^{p}\times R^{q}\times R^{q\times d}\times R^{q\times h}\times
U\rightarrow
R^{p\times d},\nonumber\\
\eta(t,x,u)\equiv\eta(t,x,v,\bar{v},\tilde{v},u,z,\cdot)&:&[0,T]\times
R^{p}\times R^{q}\times R^{q\times d}\times R^{q\times h}\times
U\times{\cal Z}^{h}\rightarrow
R^{p\times h},\nonumber\\
c(t,x,u)\equiv c(t,x,v,\bar{v},\tilde{v},u,\cdot)&:&[0,T]\times
R^{p}\times R^{q}\times R^{q\times d}\times R^{q\times h}\times
U\rightarrow R^{q}, \nonumber\\
\alpha(t,x,u)\equiv\sigma(t,x,v,\bar{v},\tilde{v},u,\cdot)&:&[0,T]\times
R^{p}\times R^{q}\times R^{q\times d}\times R^{q\times h}\times
U\rightarrow
R^{q\times d},\nonumber\\
\zeta(t,x,u)\equiv\gamma(t,x,v,\bar{v},\tilde{v},u,z,\cdot)&:&[0,T]\times
R^{p}\times R^{q}\times R^{q\times d}\times R^{q\times h}\times
U\times{\cal Z}^{h}\rightarrow R^{q\times h}.
\nonumber
\end{eqnarray}
For $f,f^{1},f^{2}\in\{b,\sigma,c,\alpha\}$, we suppose that
\begin{eqnarray}
\left\|f(u)\right\|&\leq&L(t,\omega)
\left(1+\left\|x\right\|+\|v\|+\|\bar{v}\|+\|\tilde{v}\|_{\nu}\right),
\elabel{fbconI}\\
\left\|f^{2}(u)-f^{1}(u)\right\|&\leq&L(t,\omega)
\left(\left\|x^{2}-x^{1}\right\|+\left\|v^{2}-v^{1}\right\| \right.
\elabel{fbconII}\\
&&\;\;\;\;\;\;\;\;\;\;\;\;\left.+\left\|\bar{v}^{2}-\bar{v}^{1}\right\|
+\left\|\tilde{v}^{2}-\tilde{v}^{1}\right\|_{\nu}\right). \nonumber
\end{eqnarray}
Furthermore, for each $f,f^{1},f^{2}\in\{\gamma,\zeta\}$ and
$z\in{\cal Z}^{h}$, we suppose that
\begin{eqnarray}
&&\sum_{i=1}^{h}\int_{{\cal Z}}\left\|f_{i}(u,z_{i})\right\|^{2}
\lambda_{i}\nu_{i}(dz_{i})
\elabel{afbconI}\\
&&\leq
L^{2}(t,\omega)\left(1+\left\|x\right\|^{2}+\|v\|^{2}+\|\bar{v}\|^{2}
+\|\tilde{v}\|^{2}_{\nu}\right), \nonumber
\end{eqnarray}
where $f_{i}$ is the $i$th column of $f$, and
\begin{eqnarray}
&&\sum_{i=1}^{h}\int_{{\cal
Z}}\left\|f^{2}_{i}(u,z_{i})-f^{1}_{i}(u,z_{i})\right\|^{2}
\lambda_{i}\nu_{i}(dz_{i})
\elabel{afbconII}\\
&&\leq
L^{2}(t,\omega)\left(\left\|x^{2}-x^{1}\right\|^{2}+\left\|v^{2}-v^{1}\right\|^{2}
+\left\|\bar{v}^{2}-\bar{v}^{1}\right\|^{2}
+\left\|\tilde{v}^{2}-\tilde{v}^{1}\right\|^{2}_{\nu}\right).
\nonumber
\end{eqnarray}
In addition, we assume that the terminal value $H(x)\equiv
H(x,\cdot)$ satisfies the condition,
\begin{eqnarray}
&&\left\|H(x)\right\|\leq L(t,\omega)(1+\|x\|). \elabel{termc}
\end{eqnarray}
Finally, $L$ in \eq{fbconI}-\eq{afbconII} and \eq{termc} is assumed
to be a known non-negative stochastic process that is $\{{\cal
F}_{t}\}$-adapted and mean-squarely integrable, i.e.,
\begin{eqnarray}
&&E\left[\int_{0}^{T}L^{2}(t)dt\right]<\infty. \elabel{fbconIII}
\end{eqnarray}
\begin{theorem}\label{fbexist}
Under conditions \eq{fbconI}-\eq{fbconIII}, the following two claims
are true:
\begin{enumerate}
\item If $S$ and $R$ satisfy the completely-${\cal S}$ condition,
there exists a unique adapted 6-tuple weak solution to the system in
\eq{bsdehjb} when at least one of the forward and backward SDEs has
reflection boundary;
\item Furthermore, if each $q\times q$ sub-principal matrix of
$\bar{N}'S$ and each $p\times p$ sub-principal matrix of $N'R$ are
invertible or if both of the SDEs have no reflection boundaries,
there is a unique adapted 6-tuple strong solution to the system in
\eq{bsdehjb}.
\end{enumerate}
\end{theorem}

Due to the length, the proof of Theorem~\ref{fbexist} is postponed
to Section~\ref{gfbproof}.

\subsection{Resolution via Coupled FB-SPDEs}

In this subsection, we consider a particular case of the coupled
FB-SPDEs in \eq{fbspdef} but with an additional equation, which
corresponds to the special forms of partial differential operators
$\bar{\cal L}$, $\bar{\cal J}$, and $\bar{\cal I}$. More precisely,
for each $l\in\{0,1,...,q\}$, we define
\begin{eqnarray}
&&\bar{{\cal L}}_{l}(t,x,U,V,\bar{V},\tilde{V},u)
\elabel{paretoopto}\\
&\equiv&\sum_{i,j=1}^{p}(\sigma\sigma')_{ij}(t,x,u)
\frac{\partial^{2}V_{l}(t,x)}{\partial x_{i}\partial
x_{j}}+\sum_{i=1}^{p}\left(b_{i}(t,x,u)+\sum_{j=1}^{b}v_{ij}\gamma_{j}(t,x)\right)\frac{\partial
V_{l}(t,x)}{\partial
x_{i}}\nonumber\\
&&+\sum_{j=1}^{d}\sum_{i=1}^{p}\sigma_{ji}(t,x,u)\frac{\partial
\alpha_{lj}(t,x,u)}{\partial x_{i}}-c_{l}(t,x,u)
+\sum_{k=1}^{q}s_{lk}\beta_{k}(t,x)\nonumber\\
&&-\sum_{j=1}^{h}\int_{{\cal
Z}}\left(V_{l}(t,x+\eta_{j}(t,x,u,z_{j}))-V_{l}(t,x)
-\sum_{i=1}^{p}\frac{\partial V_{l}(t,x)}{\partial
x_{i}}\eta_{ij}(t,x,u,z_{j})\right)\nu_{j}(dz_{j})
\nonumber\\
&&-\sum_{j=1}^{h}\int_{{\cal Z}}
\left(\tilde{\zeta}_{lj}(t,x+\eta_{j}(t,x,u,z_{j})),u,z_{j})
-\tilde{\zeta}_{lj}(t,x,u,z_{j})\right)\nu_{j}(dz_{j}), \nonumber
\end{eqnarray}
where $\eta_{ij}$ and $\eta_{j}$ for $i\in\{1,...,p\}$ and
$j\in\{1,...,h\}$ are the $(i,j)$th entry and the $j$th column of
$\eta$ respectively. Furthermore,
\begin{eqnarray}
c_{0}(t,x,u)&=&\sum_{l=1}^{q}c_{l}(t,x,u), \elabel{coexp}\\
\tilde{\zeta}_{0j}(t,x,u,z_{j}))&=&\sum_{l=1}^{q}\zeta_{lj}(t,x,u,z_{j})),
\elabel{coexpI}
\end{eqnarray}
and $\zeta_{lj}$ for $l\in\{1,...,q\}$ and $j\in\{1,...,h\}$ is the
$(i,j)$th entry of $\zeta$. In addition, $\gamma_{j}(t,x)$ for
$j\in\{1,...,b\}$ and $\beta_{k}(t,x)$ for $k\in\{1,...,q\}$ are
some functions in $t$ and $x$.

Note that, the partial derivative
\begin{eqnarray}
&&\frac{\partial \alpha_{lj}(t,x,u)}{\partial x_{i}}\;\;\mbox{for
each}\;\;l\in\{0,1,...,q\},\;i\in\{1,...,p\},\;j\in\{1,...,d\}
\nonumber
\end{eqnarray}
should be interpreted according to chain rule since $\alpha(t,x)$ is
also a function in $x$ through $(V,\bar{V},\tilde{V})(t,x)$ and
$u(t,x)$, where
\begin{eqnarray}
&&\alpha_{0j}(t,x,u)=\sum_{l=1}^{q}\alpha_{lj}(t,x,u).\elabel{alphazero}
\end{eqnarray}
Finally, we define
\begin{eqnarray}
\bar{{\cal J}}(t,x,U,V,\bar{V},\tilde{V})&=&-\bar{V}(t,x),
\elabel{paretooptI}\\
\bar{{\cal I}}(t,x,U,V,\bar{V},\tilde{V},z)&=&-\tilde{V}(t,x),
\elabel{paretooptII}\\
V(T,x)&=&H(x), \elabel{paretooptIII}
\end{eqnarray}
where, we assume that $H\in L^{2}_{{\cal F}_{T}}(\Omega,
C^{\infty}(D;R^{q}))$. Then, we have the following definition.
\begin{definite}\label{admissibled}
${\cal C}$ is called the admissible set of adapted control processes
if $\{\bar{{\cal L}}_{l}(t,x,U,V$,
$\bar{V},\tilde{V},u),l\in\{0,1,...,q\}\}$ together with $\{{\cal
L},{\cal J},{\cal I},\bar{{\cal J}},\bar{{\cal I}}\}$ satisfy the
conditions as stated in Theorem~\ref{bsdeyI} (or
Theorem~\ref{infdn}).
\end{definite}
\begin{theorem}\label{gtheoremo}
Let $(U(t,x),V(t,x),\bar{V}(t,x),\tilde{V}(t,x,\cdot))$ be the
unique adapted 4-tuple strong solution to the $(r,q+1)$-dimensional
coupled FB-SPDEs in \eq{fbspdef}, which corresponds to specific
$\{\bar{{\cal L}},\bar{{\cal J}},\bar{{\cal I}}\}$ in
\eq{paretoopto}-\eq{paretooptII}, terminal condition in
\eq{paretooptIII}, and a control process $u\in{\cal C}$. If $S$ and
$R$ satisfy the completely-${\cal S}$ condition, the following
claims in are true:
\begin{enumerate}
\item There exists a unique adapted 6-tuple weak solution $((X(t),Y(t))$,
$(V(t),\bar{V}(t),\tilde{V}(t,z),F(t)))$ to the system in
\eq{bsdehjb} when at least one of the SDEs has reflection boundary,
where
\begin{eqnarray}
V_{l}(t)&=&V_{l}(t,X(t)),\elabel{fbeqI}\\
\bar{V}_{lj}(t)&=&-\left(\alpha_{lj}(t,X(t),u)
+\sum_{i=1}^{p}\sigma_{li}(t,X(t),u)\frac{\partial
V_{l}(t,X(t))}{\partial x_{i}}\right),
\elabel{fbeqII}\\
\tilde{V}_{lj}(t,z))
&=&-\left(V_{l}(t,X(t)+\eta_{j}(t,X(t),u,z_{j}))-V_{l}(t,X(t))\right)
\elabel{fbeqIII}\\
&&-\zeta_{lj}(t,X(t)+\eta_{j}(t,X(t),u,z_{j}),u,z_{j}) \nonumber
\end{eqnarray}
for $l\in\{0,1,...,q\}$ and $j\in\{1,...,h\}$, where
\begin{eqnarray}
\;\;\;\;\;\;\;\;\alpha(t,X(t),u)&=&\alpha(t,X(t),V(t,X(t)),\bar{V}(t,X(t)),
\tilde{V}(t,X(t),\cdot),u(t,X(t)),\cdot),
\elabel{alphae}\\
\;\;\;\;\;\;\;\;\eta(t,X(t),u)&=&\eta(t,X(t),V(t,X(t)),\bar{V}(t,X(t)),
\tilde{V}(t,X(t),\cdot),u(t,X(t)),\cdot),
\elabel{alphae}\\
\;\;\;\;\;\;\;\;\zeta(t,X(t),u)&=&\zeta(t,X(t),V(t,X(t)),\bar{V}(t,X(t)),
\tilde{V}(t,X(t),\cdot),u(t,X(t)),\cdot); \elabel{zetae}
\end{eqnarray}
\item There is a unique adapted 6-tuple strong solution to the system in
\eq{bsdehjb} when each $q\times q$ sub-principal matrix of
$\bar{N}'S$ and each $p\times p$ sub-principal matrix of $N'R$ are
invertible or when both of the SDEs have no reflection boundaries.
\end{enumerate}
\end{theorem}

\section{Connections to Non-Zero-Sum SDGs and Queues}\label{queuesdg}

\subsection{Non-Zero-Sum SDGs}

By Theorem~\ref{gtheoremo}, we suppose that the $4$-tuple
$(X,V,\bar{V},\tilde{V})$ in \eq{utilityu} is part of a solution
$(X,Y,V,\bar{V},\tilde{V},F)$ to the non-Markvian system of coupled
FB-SDEs with L\'evy jumps and skew reflections in \eq{bsdehjb}.
Then, let $u(\cdot)$ be the corresponding $B$-valued ($B\subset
R^{q}$) and $\{{\cal F}_{t}\}$-adapted control process, whose $l$th
component $u_{l}(\cdot)$ for each $l\in\{1,...,q\}$ is the $l$th
player's control policy. Furthermore, we assume that the utility
function for each player $l\in\{1,...,q\}$ is defined by
\begin{eqnarray}
&&\left\{\begin{array}{ll}
        c_{l}(t,X(t),u)&\equiv
        c_{l}(t,X(t),V(t,X(t)),\bar{V}(t,X(t)),\tilde{V}(t,X(t)),u(t,X(t))),\\
        c_{0}(t,X(t),u)&\equiv\sum_{l=1}^{q}c_{l}(t,X(t),u),
       \end{array}
\right.
\elabel{utilityu}
\end{eqnarray}
Thus, it follows from \eq{alphae}-\eq{zetae},
\eq{coexpI}-\eq{alphazero}, and \eq{utilityu} that the value
functions $\{V^{u}_{l}(0),l\in\{1,...,q\}\}$ in \eq{gameopto} are
now well defined. Then, we can introduce the following concepts.
\begin{definite}
By a non-zero-sum SDG to the system in \eq{bsdehjb}, we mean that
each player $l\in\{1,...,q\}$ chooses an optimal policy to maximize
his own value function expressed in \eq{gameopto}. Furthermore, the
value functions $\{V^{u}_{l}(0),l\in\{1,...,q\}\}$ do not have to
add up to a constant (e.g., zero), or in other words, the SDG is not
necessarily a zero-sum one.
\end{definite}
\begin{definite}
$u^{*}(\cdot)$ is called a Pareto optimal Nash equilibrium policy
process if, the process is also an optimal one to the sum of all the
$q$ players' value functions at time zero; no player will profit by
unilaterally changing his own policy when all the other players'
policies keep the same. Mathematically,
\begin{eqnarray}
&V_{0}^{u^{*}}(0)\geq V_{0}^{u}(0),\;\;\;V_{l}^{u^{*}}(0)\geq
V_{l}^{u^{*}_{-l}}(0) \elabel{nequilibrium}
\end{eqnarray}
for each $l\in\{0,1,...,q\}$ and any given admissible control policy $u$,
where
\begin{eqnarray}
&u^{*}_{-l}=(u_{1}^{*},...,u^{*}_{l-1},u_{l},u^{*}_{l+1},...,u^{*}_{q}).
\nonumber
\end{eqnarray}
\end{definite}
\begin{definite}
$\{\bar{{\cal
L}}_{l}(t,x,U,V,\bar{V},\tilde{V},u),l\in\{0,1,...,q\}\}$ together
with $\{{\cal L},{\cal J},{\cal I}\}$ are called satisfying the
comparison principle in terms of $u$ if, for any two $u^{i}\in{\cal
C}$ with $i\in\{1,2\}$ and any two ${\cal F}_{T}$-measurable $H^{i}$
with associated two solutions
$(U^{i},V^{i},\bar{V}^{i},\tilde{V}^{i})(t,x)$, respectively, of
\eq{bsdehjb} such that
\begin{eqnarray}
\bar{{\cal
L}}_{l}(t,x,U^{1},V^{1},\bar{V}^{1},\tilde{V}^{1},u^{1})&\leq&\bar{{\cal
L}}_{l}(t,x,U^{2},V^{2},\bar{V}^{2},\tilde{V}^{2},u^{2}),\nonumber\\
H^{1}(x)&\leq& H^{2}(x) \nonumber
\end{eqnarray}
for all $(t,x)\in[0,T]\times D$, we have
\begin{eqnarray}
&V^{1}(t,x)\leq V^{2}(t,x).
\nonumber
\end{eqnarray}
\end{definite}
\begin{theorem}\label{gtheorem}
Let $(U(t,x),V(t,x),\bar{V}(t,x),\tilde{V}(t,x,\cdot))$ be the
unique adapted 4-tuple strong solution to the $(r,q+1)$-dimensional
FB-SPDEs in \eq{fbspdef}, which corresponds to specific
$\{\bar{{\cal L}},\bar{{\cal J}},\bar{{\cal I}}\}$ in
\eq{paretoopto}-\eq{paretooptII}, terminal condition in
\eq{paretooptIII}, and a control process $u\in{\cal C}$. Suppose
that $S$ and $R$ satisfy the completely-${\cal S}$ condition. If
$\{\bar{{\cal
L}}_{l}(t,x,U,V,\bar{V},\tilde{V},u),l\in\{0,1,...,q\}\}$ together
with $\{{\cal L},{\cal J},{\cal I},\bar{{\cal J}},\bar{{\cal I}}\}$
for suitably chosen $\gamma(t,x)$ and $\beta(t,x)$ satisfy the
comparison principle in terms of $u$, the following two claims are
true:
\begin{enumerate}
\item There is a Pareto optimal Nash equilibrium
point $u^{*}(t,X(t))$ to the non-zero-sum SDG problem in
\eq{gameopto} when both of the SDEs in \eq{bsdehjb} have no
reflection boundaries and if $\gamma(t,x)=\beta(t,x)\equiv 0$;
\item There is an approximated Pareto optimal Nash
equilibrium point $u^{*}(t,X(t))$ to the non-zero-sum SDG problem in
\eq{gameopto} when at least one of the SDEs in \eq{bsdehjb} has
reflection boundary and if $\gamma(t,x)$, $\beta(t,x)$ are taken to
be infinitely smooth approximated functions of $\frac{dF}{dt}$(t,x)
and $\frac{dY}{dt}(t,x)$ in $x$.
\end{enumerate}
\end{theorem}

\subsection{Queues and Reflecting Diffusions}

Queueing networks widely appear in real-world applications such as
those in service, cloud computing, and communication systems. They
typically consist of arrival processes, service processes, and
buffer storages with certain kind of service regime and network
architecture (see, e.g., an example with $p$-job classes in
Figure~\ref{queueI}). The major performance measure for this system
is the {\em queue length process} denoted by
$Q(\cdot)=(Q_{1}(\cdot),...,Q_{p}(\cdot))'$, where $Q_{i}(t)$ is the
number of $i$th class jobs stored in the $i$th buffer for each
$i\in\{1,...,p\}$ at time $t$. Let $Q(0)$ be the initial queue
length for the system. Then, the queueing dynamics of the system can
be presented by
\begin{eqnarray}
&&Q(t)=Q(0)+A(t)-D(t),\elabel{dqueue}
\end{eqnarray}
where, the $i$th component $A_{i}(t)$ of $A(t)$ for each
$i\in\{1,...,p\}$ is the total number of jobs arrived to buffer $i$
by time $t$, and the $i$th component $D_{i}(t)$ of $D(t)$ is the
total number of jobs departed from buffer $i$ by time $t$. In the
following discussions, we use two generalized ways to characterize
the arrival and departure processes.

First, we assume that each $A_{i}(\cdot)$ for $i\in\{1,...,p\}$ is a
time-inhomogeneous L\'evy process with intensity measure
$a_{i}(t,Q(t),z_{i})dt\nu_{i}(dz_{i})$ that is the job arrival rate
to buffer $i$ at time $t$ and depends on the queue state at time
$t$. Similarly, we assume that each $D_{i}(\cdot)$ is also a
time-inhomogeneous L\'evy process with intensity measure
$d_{i}(t,Q(t),z_{i})dt\nu_{i}(dz_{i})$ that is the assigned service
rate to buffer $i$ at time $t$. Furthermore, we assume that the
routing proportion from buffer $j$ to buffer $i$ for jobs finishing
service at buffer $j$ is $p_{ji}(t,Q(t),z_{j})$. Then, by the F-SDE
in \eq{bsdehjb} and the discussions in Applebaum~\cite{app:levpro},
the queue length process in \eq{dqueue} for this case can be further
expressed by
\begin{eqnarray}
dQ_{i}(t) &=&\left(\int_{{\cal
Z}}a_{i}(t,Q(t),z_{i})\nu_{i}(dz_{i})\right.
\elabel{lqueue}\\
&&+\sum_{j\neq i}\int_{{\cal
Z}}p_{ji}(t,Q(t),z_{j})d_{j}(t,Q(t),z_{j})I_{\{Q_{j}(t)>0\}}\nu_{j}(dz_{j})
\nonumber\\
&&\left.-\int_{{\cal
Z}}d_{i}(t,Q(t),z_{i})I_{\{Q_{i}(t)>0\}}\nu_{i}(dz_{i})\right)dt
\nonumber\\
&&+\int_{{\cal Z}}a_{i}(t,Q(t),z_{i})\tilde{N}_{i}(dt,dz_{i})
\nonumber\\
&&+\sum_{j\neq i}\int_{{\cal
Z}}p_{ji}(t,Q(t),z_{j})d_{j}(t,Q(t),z_{j})I_{\{Q_{j}(t)>0\}}\tilde{N}_{j}(dt,dz_{j})
\nonumber\\
&&-\int_{{\cal
Z}}d_{i}(t,Q(t),z_{i})I_{\{Q_{i}(t)>0\}}\tilde{N}_{i}(dt,dz_{i})
\nonumber\\
&&+\sum_{j=1}^{b}R_{ij}(t,Q(t))dY_{j}(t), \nonumber
\end{eqnarray}
where, ${\cal Z}=R_{+}$, $Y_{j}(t)$ in \eq{lqueue} for each
$j\in\{1,...,b\}$ is the Skorohod regulator process and it can
increase only at time $t$ when $Q_{j}(t)=0$. Note that $R(t,Q(t))$
is a reflection matrix that may be time and queue state dependent,
and the coefficients in \eq{lqueue} may be discontinuous at the
queue state $Q_{i}(t)=0$. However, since the system in \eq{lqueue}
is designed in a controllable manner, the service rate
$d_{i}(s,Q(s))$ can always be set to be zero when $Q_{i}(t)=0$,
which implies that the reflection part in \eq{lqueue} can be
removed. Hence, the generalized Lipschitz and linear growth
conditions in \eq{fbconI}-\eq{afbconII} may be reasonably imposed to
the system in \eq{lqueue}. Thus, the system derived in \eq{lqueue}
can be well-posed. Furthermore, the optimal policies in terms of
cost, profit, and system performance can be designed and analyzed
(see, e.g., the related illustration in the coming
Subsection~\ref{queuegame}). Interested readers can also find some
specific formulations of the queueing system \eq{lqueue} in
Mandelbaum and Massey~\cite{manmas:strapp}, Mandelbaum and
Pats~\cite{manpat:stadep}, and Konstantopoulos {\it et
al.}~\cite{konlas:clalev}, etc.

Second, we assume that both the arrival and service processes are
described by renewal processes, renewal reward processes, or doubly
stochastic renewal processes. In this case, the driven processes for
the queueing system do not have the nice statistical properties such
as memoryless and stationary increment ones. Thus, it is usually
impossible to conduct exact analysis concerning the distribution of
$Q(\cdot)$. However, under certain conditions (e.g., the arrival
rates close to the associated service rates), one can show that the
corresponding sequence of diffusion-scaled queue length processes
converges in distribution to a $p$-dimensional reflecting Brownian
motion (RBM) (see, e.g., Dai~\cite{dai:broapp}, Dai and
Dai~\cite{daidai:heatra}, Dai and Jiang~\cite{daijia:stoopt}), or
more generally, a reflecting diffusion with regime switching (RDRS)
(see, e.g., Dai~\cite{dai:optrat}). In other words, we have that
\begin{eqnarray}
&&\hat{Q}^{r}(\cdot)\equiv\frac{1}{r}Q(r^{2}\cdot)\Rightarrow
\hat{Q}(\cdot)\;\;\;\mbox{along}\;\;r\in\{1,2,...\}, \elabel{weakc}
\end{eqnarray}
where ``$\Rightarrow$" means ``converges in distribution" and
$\hat{Q}(\cdot)$ is a RBM or a RDRS.

To be simple, we consider the case that the limit $\hat{Q}(\cdot)$
in \eq{weakc} is a RBM living in the state space $D$ introduced in
Section~\ref{fbsdes}. Furthermore, let $\theta$ be a vector in
$R^{p}$ and $\Gamma$ be a $p\times p$ symmetric and positive
definite matrix. Then, we can introduce the definition of a RBM
(see, e.g, Dai~\cite{dai:broapp}) as follows.
\begin{definite}\label{srbm}
A semimartingale RBM associated with the data $({\bf
S},\theta,\Gamma,R)$ that has initial distribution $\pi$ is a
continuous, $\{{\cal F}_{t}\}$-adapted, $p$-dimensional process $Z$
defined on some filtered probability space $(\Omega,{\cal F},\{{\cal
F}_{t}\},{\bf P})$ such that under ${\bf P}$,
\begin{eqnarray}
X(t)=Z(t)+RY(t)\;\;for\;all\;t\geq 0, \nonumber
\end{eqnarray}
where
\begin{enumerate}
\item $X$ has continuous paths in ${\bf S}$, ${\bf P}$-a.s.,
\item under ${\bf P}$, $Z$ is a $p$-dimensional Brownian motion with
drift vector $\theta$ and covariance matrix $\Gamma$ such that
$\{Z(t)- \theta t,{\cal F}_{t},t \geq 0\}$ is a martingale and
$P{Z}^{-1}(0)=\pi$,
\item $Y$ is a $\{{\cal F}_{t}\}$-adapted, $b$-dimensional process such that
${\bf P}$-a.s., for each $i\in\{1,...,b\}$, the $i$th component
$Y_{i}$ of $Y$ satisfies
\begin{enumerate}
\item $Y_{i}(0)=0$,
\item $Y_{i}$ is continuous and non-decreasing,
\item $Y_{i}$ can increase only when Z is on the face $D_{i}$, i.e.,
as given in \eq{bsdehjb}.
\end{enumerate}
\end{enumerate}
\end{definite}
From the physical viewpoint of queueing system (see, e.g.,
Dai~\cite{dai:broapp,dai:optrat}) and the discussion in Reiman and
Williams~\cite{reiwil:boupro}, the pushing process $Y$ in
Definition~\ref{srbm} can be assumed to a.s. satisfy
\begin{eqnarray}
Y_{i}(t)=\int_{0}^{t}I_{D_{i}}(X(s))ds. \elabel{equivy}
\end{eqnarray}

Now, assume that $H(x)$ is the stationary distribution that we
expect for the RBM $X$. For example, in reality, it is the given
distribution of the long-run average queue lengths among different
users or job classes. Theoretically, it can be computed by a method
(e.g., the finite element method designed and implemented in Dai
{\em et al.}~\cite{dai:broapp,shechedaidai:finele}).
Then, we can use a B-PDE or a B-SPDE (a special form of the system
in \eq{fbspdef}) to get the transition function at each time point
to reach the targeted or limiting stationary distribution $H(x)$ for
the RBM $X$ for a given initial distribution (e.g., $X(0)=0$ a.s. in
many situations). Hence, the corresponding performance measures of
the physical queueing system can be estimated. More precisely, we
have the following theorem and related remark.
\begin{theorem}\label{srbmth}
Suppose that the reflection matrix satisfies the completely-${\cal
S}$ condition. Then, the transition function of the RBM $X$ over
$[0,T]$ is determined by
\begin{eqnarray}
\;\;\;\;\;V(t,x)&=&H(x)+\int_{t}^{T}{\cal L}(s,x,V)ds,
\elabel{bspdefqueue}
\end{eqnarray}
where $V$ is a $1$-dimensional function. Furthermore, ${\cal L}$ is
the following form of partial differential operator
\begin{eqnarray}
{\cal L}(t,x,V,\cdot)&=&({\cal K}(t,x,V,\cdot),{\cal
D}_{1}(t,x,V,\cdot),...,{\cal D}_{b}(t,x,V,\cdot)),
\elabel{concretefI}\\
{\cal
K}(t,x,V,\cdot)&=&\sum_{i,j=1}^{p}\Gamma_{ij}\frac{\partial^{2}V(t,x)}{\partial
x_{i}\partial x_{j}}+\theta\cdot\bigtriangledown
V(t,x)+\sum_{i=1}^{b}{\cal D}_{i}(t,x,V,\cdot),
\elabel{concretefII}\\
{\cal D}_{i}(t,x,V,\cdot)&=&(v_{i}\cdot\bigtriangledown
V(t,x))I_{D_{i}}(x) \;\; \mbox{for}\;\;x\in
D_{i}\;\;\mbox{with}\;\;i\in\{1,...,b\}, \elabel{concretefIII}
\end{eqnarray}
where $\bigtriangledown V$ is the gradient vector of $V$ in $x$ and
$I_{F_{i}}$ is the indicator function over the set $F_{i}$.
\end{theorem}
\begin{proof}
It follows from the completely-${\cal S}$ condition that the RBM $X$
is a strong Markov process (see, e.g., Dai and
Williams~\cite{daiwil:exiuni}).
Then, by applying the It$\hat{o}$'s formula (see, e.g.,
Protter~\cite{pro:stoint}) and Fokker-Planck's formula (or called
Kolmogorov's forward/backward equations, see, e.g.,
$\emptyset$ksendal~\cite{oks:stodif}), we know that the claim stated
in the theorem is true. $\Box$
\end{proof}
\begin{remark}
Owing to the uncertainty error of measurement, $H(x)$ could be
random. Furthermore, the coefficients in \eq{concretefII} may also
be random, e.g., for the case that the limit $\hat{Q}(\cdot)$ is a
RDRS. Thus, a B-SPDE can be introduced. Furthermore, the indicator
function $I_{F_{i}}(x)$ can be approximated by a sufficient smooth
function in order to apply Theorem~\ref{bsdeyI} to the equation in
\eq{bspdefqueue}, which is reasonable from the viewpoint of
numerical computation.
\end{remark}

\subsection{Queueing Based Game Problem}\label{queuegame}

From the information system displayed in
Figures~\ref{gameI}-\ref{queueI} (presenting a parallel-server
queueing system with $q=p$), we can give an explanation about the
decision process for such a game problem. In this game, each player
(or called user in Dai~\cite{dai:optrat}) relates to a control
process $u_{l}(\cdot)$ for $l\in\{1,...,q\}$ over certain resource
pool (e.g., called the transmission rate allocation process over a
randomly evolving capacity region in Dai~\cite{dai:optrat}). In the
meanwhile, each player $l$ is assigned a surrogate utility function
$c_{l}$ of his submitted bid (called queue length in
Dai~\cite{dai:optrat}, or the approximated queue length RBM $Z$ in
Definition~\ref{srbm}) to the network, the price from the network to
him, and the control policy at each time point by the central
information administrative. Then, an optimal and/or fair control
process can be determined by the utility functions of all players,
queueing process, and the available resource constraint in a
cooperative way (see, e.g., Jones~\cite{jon:gamthe}).

\section{Proofs of Theorem~\ref{bsdeyI} and Theorem~\ref{infdn}}
\label{uniextproof}

We justify the two theorems by first proving three lemmas in the
following subsection.

\subsection{The Lemmas}

\begin{lemma}\label{martindecom}
Assume that the conditions in Theorem~\ref{bsdeyI} hold and take a
quadruplet for each fixed $x\in D$ and $z\in{\cal Z}^{h}$,
\begin{eqnarray}
&(U^{1}(\cdot,x),V^{1}(\cdot,x),\bar{V}^{1}(\cdot,x),\tilde{V}^{1}(\cdot,x,z))\in
{\cal Q}^{2}_{{\cal F}}([0,T]\times D).\elabel{threeu}
\end{eqnarray}
Then, there exists another quadruplet
$(U^{2}(\cdot,x),V^{2}(\cdot,x),\bar{V}^{2}(\cdot,x),\tilde{V}^{2}(\cdot,x,z))$
such that
\begin{eqnarray}
\;\;\;\;\;\;\;\;\left\{\begin{array}{ll}
U^{2}(t,x)&=G(x)+\int_{0}^{t}{\cal L}(s^{-},x,U^{1},V^{1},\bar{V}^{1},\tilde{V}^{1})ds\\
&\;\;\;\;\;+\int_{0}^{t}{\cal J}(s^{-},x,U^{1},V^{1},\bar{V}^{1},\tilde{V}^{1})dW(s)\\
&\;\;\;\;\;+\int_{0}^{t}\int_{{\cal Z}^{h}}{\cal
I}(s^{-},x,U^{1},V^{1},\bar{V}^{1},\tilde{V}^{1},z)\tilde{N}(\lambda ds,dz),\\
V^{2}(t,x)&=H(x)+\int_{t}^{T}\bar{{\cal L}}(s^{-},x,U^{1},V^{1},\bar{V}^{1},\tilde{V}^{1})ds\\
&\;\;\;\;\;+\int_{t}^{T}\left(\bar{{\cal
J}}(s^{-},x,U^{1},V^{1},\bar{V}^{1},\tilde{V}^{1})\right.\\
&\;\;\;\;\;\;\;\;\;\;\;\;\;\;\;
\left.+\bar{V}^{1}(s^{-},x)-\bar{V}^{2}(s^{-},x)\right)dW(s)\\
&\;\;\;\;\;+\int_{t}^{T}\int_{{\cal Z}^{h}}\left(\bar{{\cal
I}}(s^{-},x,U^{1},V^{1},\bar{V}^{1},\tilde{V}^{1},z)\right.\\
&\;\;\;\;\;\;\;\;\;\;\;\;\;\;\;
\left.+\tilde{V}^{1}(s^{-},x,z)-\tilde{V}^{2}(s^{-},x,z)\right)\tilde{N}(\lambda
ds,dz),
\end{array}
\right.
\elabel{sigmanV}
\end{eqnarray}
where $(U^{2},V^{2})$ is a $\{{\cal F}_{t}\}$-adapted c\`adl\`ag
process and $(\bar{V}^{2},\tilde{V}^{2})$ is the corresponding
predictable process. Furthermore, for each $x\in D$,
\begin{eqnarray}
&&E\left[\int_{0}^{T}\|U^{2}(t,x)\|^{2}dt\right]<\infty,
\elabel{maxnormo}\\
&&E\left[\int_{0}^{T}\|V^{2}(t,x)\|^{2}dt\right]<\infty,
\elabel{maxnormI}\\
&&E\left[\int_{0}^{T}\|\bar{V}^{2}(t,x)\|^{2}dt\right]<\infty,
\elabel{maxnormII}\\
&&E\left[\sum_{i=1}^{h}\int_{0}^{T}\int_{{\cal Z}}
\left\|\tilde{V}^{2}_{i}(t,x,z_{i})\right\|^{2}
\nu_{i}(dz_{i})dt\right]<\infty. \elabel{maxnormIII}
\end{eqnarray}
\end{lemma}
\begin{proof}
For each fixed $x\in D$ and a quadruplet as stated in \eq{threeu},
it follows from conditions \eq{blipschitz}-\eq{blipicADII} that
\begin{eqnarray}
&&{\cal L}(\cdot,x,U^{1},V^{1},\bar{V}^{1},\tilde{V}^{1})\in L^{2}_{{\cal
F}}([0,T],C^{\infty}(D,R^{r})),
\elabel{ladaptedI}\\
&&{\cal J}(\cdot,x,U^{1},V^{1},\bar{V}^{1},\tilde{V}^{1})\in L^{2}_{{\cal
F}}([0,T],C^{\infty}(D,R^{r\times d})),
\elabel{ladaptedIai}\\
&&{\cal I}(\cdot,x,U^{1},V^{1},\bar{V}^{1},\tilde{V}^{1})\in
L^{2}_{{\cal F}}([0,T]\times{\cal Z}^{h},C^{\infty}(D,R^{r\times
h})),
\elabel{ladaptedIaii}\\
&&\bar{{\cal L}}(\cdot,x,U^{1},V^{1},\bar{V}^{1},\tilde{V}^{1})\in L^{2}_{{\cal
F}}([0,T],C^{\infty}(D,R^{q})),
\elabel{fbladaptedI}\\
&&\bar{{\cal J}}(\cdot,x,U^{1},V^{1},\bar{V}^{1},\tilde{V}^{1})\in L^{2}_{{\cal
F}}([0,T],C^{\infty}(D,R^{q\times d})),
\elabel{fbladaptedIai}\\
&&\bar{{\cal I}}(\cdot,x,U^{1},V^{1},\bar{V}^{1},\tilde{V}^{1})\in
L^{2}_{{\cal F}}([0,T]\times{\cal Z}^{h},C^{\infty}(D,R^{q\times
h})). \elabel{fbladaptedIaii}
\end{eqnarray}
By considering ${\cal L}$, ${\cal J}$, and ${\cal I}$ in
\eq{ladaptedI}-\eq{ladaptedIaii} as new starting ${\cal
L}(\cdot,x,0,0,0,0)$, ${\cal J}(\cdot,x,0,0,0,0)$, and ${\cal
I}(\cdot,x,0,0,0,0)$, we can define $U^{2}$ by the forward iteration
in \eq{sigmanV}. Furthermore, $U^{2}$ is a $\{{\cal
F}_{t}\}$-adapted c\`adl\`ag process that is square-integrable for
each $x\in D$ in the sense of \eq{maxnormo}.

Now, consider $\bar{{\cal L}}$, $\bar{{\cal J}}$, and $\bar{{\cal
I}}$ in \eq{fbladaptedI}-\eq{fbladaptedIaii} as new starting
$\bar{{\cal L}}(\cdot,x,0,0,0,0)$, $\bar{{\cal
J}}(\cdot,x,0,0,0,0)$, and $\bar{{\cal I}}(\cdot,x,0,0,0,0)$. Then,
it follows from the Martingale representation theorem (see, e.g.,
Theorem 5.3.5 in page 266 of Applebaum~\cite{app:levpro}) that there
are unique predictable processes $\bar{V}^{2}(\cdot,x)$ and
$\tilde{V}^{2}(\cdot,x,z)$ such that
\begin{eqnarray}
&&\hat{V}^{2}(t,x)
\elabel{sigmanI}\\
&\equiv& E\left[H(x) +\int_{0}^{T}\bar{{\cal
L}}(s^{-},x,U^{1},V^{1},\bar{V}^{1},\tilde{V}^{1})ds\right.
\nonumber\\
&&+\int_{0}^{T}\left(\bar{{\cal
J}}(s^{-},x,U^{1},V^{1},\bar{V}^{1},\tilde{V}^{1})+\bar{V}^{1}(s^{-},x)\right)dW(s)
\nonumber\\
&& \left.\left.+\int_{0}^{T}\int_{{\cal Z}}\left(\bar{{\cal
I}}(s^{-},x,U^{1},V^{1},\bar{V}^{1},\tilde{V}^{1},z)+\tilde{V}^{1}(s^{-},x,z)\right)
\tilde{N}(\lambda ds,dz)\right|{\cal F}_{t}\right]
\nonumber\\
&=&\hat{V}^{2}(0,x)+\int_{0}^{t}\bar{V}^{2}(s^{-},x)dW(s)
+\int_{0}^{t}\int_{{\cal
Z}}\tilde{V}^{2}(s^{-},x,z)\tilde{N}(\lambda ds,dz). \nonumber
\end{eqnarray}
Furthermore, $\bar{V}^{2}$ and $\tilde{V}^{2}$ are square-integrable
for each $x\in D$ in the sense of \eq{maxnormII}-\eq{maxnormIII},
and
\begin{eqnarray}
&&\hat{V}^{2}(0,x)
\elabel{sigmanII}\\
&=&\hat{V}^{2}(T,x)-\int_{0}^{T}\bar{V}^{2}(s^{-},x)dW(s)
-\int_{0}^{T} \int_{{\cal
Z}}\tilde{V}^{2}(s^{-},x,z)\tilde{N}(\lambda
ds,dz)\nonumber\\
&=&H(x)+\int_{0}^{T}\bar{{\cal
L}}(s^{-},x,U^{1},V^{1},\bar{V}^{1},\tilde{V}^{1})ds
\nonumber\\
&&+\int_{0}^{T}\left(\bar{{\cal
J}}(s^{-},x,U^{1},V^{1},\bar{V}^{1},\tilde{V}^{1})+\bar{V}^{1}(s^{-},x)
-\bar{V}^{2}(s^{-},x)\right)dW(s)
\nonumber\\
&&+\int_{0}^{T}\int_{{\cal Z}}\left(\bar{{\cal
I}}(s^{-},x,U^{1},V^{1},\bar{V}^{1},\tilde{V}^{1},z)+\tilde{V}^{1}(s^{-},x,z)
-\tilde{V}^{2}(s^{-},x,z)\right)\tilde{N}(\lambda ds,dz). \nonumber
\end{eqnarray}
Owing to the corollary in page 8 of Protter~\cite{pro:stoint},
$\hat{V}^{2}(\cdot,x)$ can be taken as a c\`adl\`ag process. Now,
define a process $V^{2}$ given by
\begin{eqnarray}
V^{2}(t,x)&=& E\left[H(x) +\int_{t}^{T}\bar{{\cal
L}}(s^{-},x,U^{1},V^{1},\bar{V}^{1},\tilde{V}^{1})ds\right.
\elabel{sigmanIII}\\
&&+\int_{t}^{T}\left(\bar{{\cal
J}}(s^{-},x,U^{1},V^{1},\bar{V}^{1},\tilde{V}^{1})+\bar{V}^{1}(s^{-},x)\right)dW(s)
\nonumber\\
&&\left.\left.+\int_{t}^{T}\int_{{\cal Z}}\left(\bar{{\cal
I}}(s^{-},x,U^{1},V^{1},z)+\tilde{V}^{1}(s^{-},x,z)\right)\tilde{N}(\lambda
ds,dz)\right|{\cal F}_{t}\right]. \nonumber
\end{eqnarray}
Thus, it follows from \eq{fblipschitzoI}-\eq{nlipoI} and simple
calculation that $V^{2}(\cdot,x)$ is square-integrable in the sense
of \eq{maxnormI}. In addition, by \eq{sigmanI}-\eq{sigmanIII}, we
know that
\begin{eqnarray}
V^{2}(t,x)&=&\hat{V}^{2}(t,x)-\int_{0}^{t}\bar{{\cal
L}}(s^{-},x,U^{1},V^{1},\bar{V}^{1},\tilde{V}^{1})ds
\elabel{sigmanIV}\\
&&-\int_{0}^{t}\left(\bar{{\cal
J}}(s^{-},x,U^{1},V^{1},\bar{V}^{1},\tilde{V}^{1})+\bar{V}^{1}(s^{-},x)\right)dW(s)
\nonumber\\
&&-\int_{0}^{t}\int_{{\cal Z}}\left(\bar{{\cal
I}}(s^{-},x,U^{1},V^{1},z)+\tilde{V}^{1}(s^{-},x,z)\right)\tilde{N}(\lambda
ds,dz), \nonumber
\end{eqnarray}
which implies that $V^{2}(\cdot,x)$ is a c\`adl\`ag process.

Hence, for a given quadruplet in \eq{threeu}, it follows from
\eq{sigmanI}-\eq{sigmanII} and \eq{sigmanIV} that the associated
quadruplet ($U^{2}(\cdot,x)$, $V^{2}(\cdot,x),$
$\bar{V}^{2}(\cdot,x)$, $\tilde{V}^{2}(\cdot,x,z)$) satisfies the
equation \eq{sigmanV} as stated in the lemma. Furthermore, we know
that
\begin{eqnarray}
&&V^{2}(t,x)\elabel{sigmanVI}\\
&\equiv&V^{2}(0,x)-\int_{0}^{t}\bar{{\cal
L}}(s^{-},x,U^{1},V^{1},\bar{V}^{1},\tilde{V}^{2})ds
\nonumber\\
&&-\int_{0}^{t}\left(\bar{{\cal
J}}(s^{-},x,U^{1},V^{1},\bar{V}^{1},\tilde{V}^{1})+\bar{V}^{1}(s^{-},x)
-\bar{V}^{2}(s^{-},x)\right)dW(s)
\nonumber\\
&&-\int_{0}^{t}\int_{{\cal Z}}\left(\bar{{\cal
I}}(s^{-},x,U^{1},V^{1},\bar{V}^{-},\tilde{V}^{1},z)+\tilde{V}^{1}(s^{-},x,z)
-\tilde{V}^{2}(s^{-},x,z)\right)\tilde{N}(\lambda ds,dz). \nonumber
\end{eqnarray}
Thus, we complete the proof of Lemma~\ref{martindecom}. $\Box$
\end{proof}
\begin{lemma}\label{differentiableV}
Under the conditions of Theorem~\ref{bsdeyI}, consider a quadruplet
as in \eq{threeu} for each fixed $x\in D$ and $z\in{\cal Z}^{h}$.
Define $(U(t,x), V(t,x),\bar{V}(t,x),\tilde{V}(t,x,z))$ by
\eq{sigmanV}. Then, $(U^{(c)}(\cdot,x)$, $V^{(c)}(\cdot,x)$,
$\bar{V}^{(c)}(\cdot,x)$, $\tilde{V}^{(c)}(\cdot,x,z))$ for each
$c\in\{0,1,...,\}$ exists a.s. and satisfies
\begin{eqnarray}
&&\;\;\;\;\;\;\left\{\begin{array}{ll}
U^{(c)}_{i_{1}...i_{p}}(t,x)&=G^{(c)}_{i_{1}...i_{p}}(x)
+\int_{0}^{t}{\cal L}^{(c)}_{i_{1}...i_{p}}(s^{-},x,U^{1},V^{1},\bar{V}^{1},\tilde{V}^{1})ds\\
&\;\;\;\;+\int_{0}^{t}{\cal J}^{(c)}_{i_{1}...i_{p}}(s^{-},x,U^{1},V^{1},\bar{V}^{1},\tilde{V}^{1})dW(s)\\
&\;\;\;\;+\int_{0}^{t}\int_{{\cal Z}}{\cal
I}^{(c)}_{i_{1}...i_{p}}(s^{-},x,U^{1},V^{1},\bar{V}^{1},\tilde{V}^{1},z)\tilde{N}(\lambda ds,dz),\\
V^{(c)}_{i_{1}...i_{p}}(t,x)&=H^{(c)}_{i_{1}...i_{p}}(x)
+\int_{t}^{T}\bar{{\cal L}}^{(c)}_{i_{1}...i_{p}}(s^{-},x,U^{1},V^{1},\bar{V}^{1},\tilde{V}^{1})ds\\
&\;\;\;\;+\int_{t}^{T}\left(\bar{{\cal
J}}^{(c)}_{i_{1}...i_{p}}(s^{-},x,U^{1},V^{1},\bar{V}^{1},\tilde{V}^{1})\right.\\
&\;\;\;\;\;\;\;\;\left.+\bar{V}^{1,(c)}_{i_{1}...i_{p}}(s^{-},x)-\bar{V}^{(c)}_{i_{1}...i_{p}}(s^{-},x)\right)dW(s)\\
&\;\;\;\;+\int_{t}^{T}\int_{{\cal Z}}\left(\bar{{\cal
I}}^{(c)}_{i_{1}...i_{p}}(s^{-},x,U^{1},V^{1},\bar{V}^{1},\tilde{V}^{1},z)\right.\\
&\;\;\;\;\;\;\;\;\left.+\tilde{V}^{1,(c)}_{i_{1}...i_{p}}
(s^{-},x,z)-\tilde{V}^{(c)}_{i_{1}...i_{p}}
(s^{-},x,z)\right)\tilde{N}(\lambda ds,dz),
\end{array}
\right.
\elabel{pdsigmanV}
\end{eqnarray}
where $i_{1}+...+i_{p}=c$ and $i_{l}\in\{0,1,...,c\}$ with
$l\in\{1,...,p\}$. Furthermore, $(U^{(c)}_{i_{1}...i_{p}}$,
$V^{(c)}_{i_{1}...i_{p}})$ for each $c\in\{0,1,...\}$ is a $\{{\cal
F}_{t}\}$-adapted c\`adl\`ag process and
$(\bar{V}^{(c)}_{i_{1}...i_{p}}$, $\tilde{V}^{(c)}_{i_{1}...i_{p}})$
is the associated predictable processes. All of them are
squarely-integrable in the senses of \eq{maxnormI}-\eq{maxnormIII}.
\end{lemma}
\begin{proof}
Without loss of generality, we only consider the point $x\in D$,
which is an interior one of $D$. Otherwise, we can use the
corresponding derivative in a one-side manner to replace the one in
the following proof.

First, we show that the claim in the lemma is true for $c=1$. To do
so, for each given $t\in[0,T],x\in D,z\in{\cal Z}^{h}$, and
$(U^{1}(t,x),V^{1}(t,x),\bar{V}^{1}(t,x),\tilde{V}^{1}(t,x,z))$ as
in the lemma, let
\begin{eqnarray}
&(U^{(1)}_{i_{l}}(t,x),V^{(1)}_{i_{l}}(t,x),\bar{V}^{(1)}_{i_{l}}(t,x),
\tilde{V}^{(1)}_{i_{l}}(t,x,z)) \elabel{tripletdv}
\end{eqnarray}
be defined by \eq{sigmanV} but each ${\cal A}\in\{{\cal L},{\cal
J},{\cal I},\bar{{\cal L}},\bar{{\cal J}},\bar{{\cal I}}\}$ is
replaced by its first-order partial derivative
\begin{eqnarray}
&{\cal A}^{(1)}_{i_{l}}\in\left\{{\cal L}^{(1)}_{i_{l}}, {\cal
J}^{(1)}_{i_{l}},{\cal I}^{(1)}_{i_{l}},\bar{{\cal
L}}^{(1)}_{i_{l}},\bar{{\cal J}}^{(1)}_{i_{l}},\bar{{\cal
I}}^{(1)}_{i_{l}}\right\} \nonumber
\end{eqnarray}
with respect to $x_{l}$ for $l\in\{1,...,p\}$ if $i_{l}=1$. Then, we
can show that the quadruplet defined in \eq{tripletdv} for each $l$
is the required first-order partial derivative of
$(U,V,\bar{V},\tilde{V})$ in \eq{sigmanV} for the given
$(U^{1},V^{1},\bar{V}^{1},\tilde{V}^{1})$.

In fact, considering an interior point $x$ of $D$, we can take
sufficiently small constant $\delta$ such that $x+\delta e_{l}\in
D$, where $e_{l}$ is the unit vector whose $l$th component is one
and others are zero. Without loss of generality, we assume that
$\delta>0$. Then, for each
$f\in\{U,V,\bar{V},\tilde{V},U^{1},V^{1},\bar{V}^{1},\tilde{V}^{1}\}$
and $i_{l}=1$ with $l\in\{1,...,p\}$, we define
\begin{eqnarray}
&&f_{i_{l},\delta}(t,x)\equiv f(t,x+\delta e_{l}).
\elabel{fdeltan}
\end{eqnarray}
Furthermore, let
\begin{eqnarray}
&\Delta f^{(1)}_{i_{l},\delta}(t,x)
=\frac{f_{i_{l},\delta}(t,x)-f(t,x)}{\delta}-f^{(1)}_{i_{l}}(t,x),
\elabel{partialII}
\end{eqnarray}
and let
\begin{eqnarray}
&&\Delta{\cal
A}^{(1)}_{i_{l},\delta}(t,x,U^{1},V^{1},\bar{V}^{1},\tilde{V}^{1})
\elabel{deltasqr}\\
&=&\frac{1}{\delta}\left({\cal A}(t,x+\delta e_{l},U^{1}(t,x+\delta
e_{l}),V^{1}(t,x+\delta e_{l}),\bar{V}^{1}(t,x+\delta
e_{l}),\tilde{V}^{1}(t,x+\delta e_{l},z))\right.
\nonumber\\
&&\;\;\;\;\;\;\left.-{\cal
A}(t,x,U^{1}(s,x),V^{1}(t,x),\bar{V}^{1}(t,x),\tilde{V}^{1}(t,x,z))\right)
\nonumber\\
&&-{\cal
A}^{(1)}_{i_{l}}(t,x,U^{1}(s,x),V^{1}(t,x),\bar{V}^{1}(t,x),\tilde{V}^{1}(t,x,z))
\nonumber
\end{eqnarray}
for each ${\cal A}\in\{{\cal L},{\cal J},{\cal I},\bar{{\cal
L}},\bar{{\cal J}},\bar{{\cal I}}\}$.

Now, let Tr$(A)$ denote the trace of the matrix $A'A$ for a given
matrix $A$ and let $(\mbox{Tr}(A))_{j}$ be the $j$th term in the
summation of the trace. Furthermore, for each fixed $t\in[0,T]$,
$\delta>0$, and $\gamma>0$, define
\begin{eqnarray}
Z_{\delta}(t,x)&\equiv&\zeta(\Delta
U^{(1)}_{i_{l},\delta}(t,x)+\Delta V^{(1)}_{i_{l},\delta}(t,x))
\elabel{itofunction}\\
&=&\left(\mbox{Tr}\left(\Delta
U_{i_{l},\delta}^{(1)}(t,x)\right)+\mbox{Tr}\left(\Delta
V_{i_{l},\delta}^{(1)}(t,x)\right)\right)e^{2\gamma t}. \nonumber
\end{eqnarray}
Then, it follows from \eq{sigmanVI} and the It$\hat{o}$'s formula
(see, e.g., Theorem 1.14 and Theorem 1.16 in pages 6-9 of
$\emptyset$ksendal and Sulem~\cite{okssul:appsto}) that
\begin{eqnarray}
&&Z_{\delta}(t,x)+\int_{t}^{T} \mbox{Tr}\left(\Delta\bar{{\cal
J}}^{(1)}_{i_{l},\delta}(s^{-},x,U^{1},V^{1},\bar{V}^{1},\tilde{V}^{1})\right.
\elabel{pvitod}\\
&&\;\;\;\;\;\;\;\;\;\;\;\;\;\;\;\;\;\;\;\;\;\;\;\;\;\;
\left.+\Delta\bar{V}^{1,(1)}_{i_{l},\delta}(s^{-},x)
-\Delta\bar{V}^{(1)}_{i_{l},\delta}(s,x) \right)e^{2\gamma s}ds
\nonumber\\
&&+\sum_{j=1}^{h}\int_{t}^{T} \int_{{\cal
Z}}\left(\mbox{Tr}\left(\Delta\bar{{\cal
I}}^{(1)}_{i_{l},\delta}(s^{-},x,U^{1},V^{1},\bar{V}^{1},\tilde{V}^{1},z)\right.\right.
\nonumber\\
&&\;\;\;\;\;\;\;\;\;\;\;\;\;\;\;\;\;\;\;\;\;\;\;\;\;\;
\left.\left.+\Delta\tilde{V}^{1,(1)}_{i_{l},\delta}(s^{-},x,z_{j})
-\Delta\tilde{V}^{(1)}_{i_{l},\delta}(s^{-},x,z)\right)\right)_{j}
e^{2\gamma s}N_{j}(\lambda_{j}ds,dz_{j})
\nonumber\\
&=&2\int_{0}^{t}\left(-\gamma\mbox{Tr}\left(\Delta
U_{i_{l},\delta}^{(1)}(s,x)\right)+\left(\Delta
U_{i_{l},\delta}^{(1)}(s,x)\right)' \left(\Delta{\cal
L}^{(1)}_{i_{l},\delta}(s,x,U^{1},V^{1},\bar{V}^{1},\tilde{V}^{1})\right)\right)e^{2\gamma
s}ds
\nonumber\\
&&+2\int_{t}^{T}\left(-\gamma\mbox{Tr}\left(\Delta
V_{i_{l},\delta}^{(1)}(s,x)\right)+\left(\Delta
V_{i_{l},\delta}^{(1)}(s,x)\right)' \left(\Delta\bar{{\cal
L}}^{(1)}_{i_{l},\delta}(s,x,U^{1},V^{1},\bar{V}^{1},\tilde{V}^{1})\right)\right)e^{2\gamma
s}ds
\nonumber\\
&&-M_{\delta}(t,x)
\nonumber\\
&\leq&\left(-2\gamma+\frac{1}{\hat{\gamma}}\right)
\left(\int_{0}^{t}\mbox{Tr}\left(\Delta
U_{i_{l},\delta}^{(1)}(s,x)\right)e^{2\gamma
s}ds+\int_{t}^{T}\mbox{Tr}\left(\Delta
V_{i_{l},\delta}^{(1)}(s,x)\right)e^{2\gamma s}ds\right)
\nonumber\\
&&+\hat{\gamma}\int_{0}^{t}\left\|\Delta{\cal
L}^{(1)}_{i_{l},\delta}(s,x,U^{1},V^{1},\bar{V}^{1},\tilde{V}^{1})\right\|^{2}e^{2\gamma
s}ds
\nonumber\\
&&+\hat{\gamma}\int_{t}^{T}\left\|\Delta\bar{{\cal
L}}^{(1)}_{i_{l},\delta}(s,x,U^{1},V^{1},\bar{V}^{1},\tilde{V}^{1})\right\|^{2}e^{2\gamma
s}ds\nonumber\\
&&-M_{\delta}(t,x)
\nonumber\\
&=&\hat{\gamma}\int_{0}^{t}\left\|\Delta{\cal
L}^{(1)}_{i_{l},\delta}(s,x,U^{1},V^{1},\bar{V}^{1},\tilde{V}^{1})\right\|^{2}e^{2\gamma
s}ds
\nonumber\\
&&+\hat{\gamma}\int_{t}^{T}\left\|\Delta\bar{{\cal
L}}^{(1)}_{i_{l},\delta}(s,x,U^{1},V^{1},\bar{V}^{1},\tilde{V}^{1})\right\|^{2}e^{2\gamma
s}ds\nonumber\\
&&-M_{\delta}(t,x) \nonumber
\end{eqnarray}
if, in the last equality, we take
\begin{eqnarray}
\hat{\gamma}=\frac{1}{2\gamma}>0.
\elabel{thgamma}
\end{eqnarray}
Note that $M_{\delta}(t,x)$ in \eq{pvitod} is a martingale of the
form,
\begin{eqnarray}
&&M_{\delta}(t,x)\elabel{tmsigman}\\
&=&-2\sum_{j=1}^{d}\int_{0}^{t}\left(\Delta
U^{(1)}_{i_{l},\delta}(s^{-},x)\right)'\Delta({\cal
J}_{j})_{i_{l},\delta}^{(1)}(s^{-},x,U^{1},V^{1},\bar{V}^{1},\tilde{V}^{1})e^{2\gamma
s}dW_{j}(s)
\nonumber\\
&&-2\sum_{j=1}^{h}\int_{0}^{t} \int_{{\cal Z}}\left(\Delta
U^{(1)}_{i_{l},\delta}(s^{-},x)\right)' \Delta({\cal
I}_{j})^{(1)}_{i_{l},\delta}(s^{-},x,U^{1},V^{1},\bar{V}^{1},\tilde{V}^{1},z_{j})e^{2\gamma
s}\tilde{N}_{j}(\lambda_{j}ds,dz_{j})
\nonumber\\
&&+2\sum_{j=1}^{d}\int_{t}^{T}\left(\Delta
V^{(1)}_{i_{l},\delta}(s^{-},x)\right)'\left(\Delta(\bar{{\cal
J}}_{j})_{i_{l},\delta}^{(1)}(s^{-},x,U^{1},V^{1},\bar{V}^{1},\tilde{V}^{1})\right.
\nonumber\\
&&\;\;\;\;\;\;\;\;\;\;\;\;\;\;\;\;\;\;\;
\left.+\Delta(\bar{V}^{1}_{j})_{i_{l},\delta}^{(1)}(s^{-},x)
-\Delta(\bar{V}_{j})_{i_{l},\delta}^{(1)}(s^{-},x)\right)e^{2\gamma
s}dW_{j}(s)
\nonumber\\
&&+2\sum_{j=1}^{h}\int_{t}^{T} \int_{{\cal Z}}\left(\Delta
V^{(1)}_{i_{l},\delta}(s^{-},x)\right)'\left(\Delta(\bar{{\cal
I}}_{j})^{(1)}_{i_{l},\delta}(s^{-},x,U^{1},V^{1},\bar{V}^{1},\tilde{V}^{1},z_{j})\right.
\nonumber\\
&&\;\;\;\;\;\;\;\;\;\;\;\;\;\;\;\;\;\;\;
\left.+\Delta(\tilde{V}^{1}_{j})^{(1)}_{i_{l},\delta}(s^{-},x,z_{j})
-\Delta(\tilde{V}_{j})^{(1)}_{i_{l},\delta}(s^{-},x,z_{j})\right)
e^{2\gamma s}\tilde{N}_{j}(\lambda_{j}ds,dz_{j}). \nonumber
\end{eqnarray}
Thus, by the martingale property and \eq{pvitod}, we know that
\begin{eqnarray}
&&E\left[Z_{\delta}(t,x)+\int_{t}^{T}
\mbox{Tr}\left(\Delta\bar{{\cal
J}}^{(1)}_{i_{l},\delta}(s^{-},x,U^{1},V^{1},\bar{V}^{1},\tilde{V}^{1})\right.\right.
\elabel{epvitod}\\
&&\;\;\;\;\;\;\;\;\;\;\;\;\;\;\;\;\;\;\;\;\;\;\;\;\;\;
\left.+\Delta\bar{V}^{1,(1)}_{i_{l},\delta}(s^{-},x)
-\Delta\bar{V}^{(1)}_{i_{l},\delta}(s,x) \right)e^{2\gamma s}ds
\nonumber\\
&&+\sum_{j=1}^{h}\int_{t}^{T} \int_{{\cal
Z}}\left(\mbox{Tr}\left(\Delta\bar{{\cal
I}}^{(1)}_{i_{l},\delta}(s^{-},x,U^{1},V^{1},\bar{V}^{1},\tilde{V}^{1},z)\right.\right.
\nonumber\\
&&\;\;\;\;\;\;\;\;\;\;\;\;\;\;\;\;\;\;\;\;\;\;\;\;\;\;
\left.\left.\left.+\Delta\tilde{V}^{1,(1)}_{i_{l},\delta}(s^{-},x,z)
-\Delta\tilde{V}^{(1)}_{i_{l},\delta}(s^{-},x,z)\right)\right)_{j}
e^{2\gamma s}N_{j}(\lambda_{j}ds,dz_{j})\right]
\nonumber\\
&\leq&\hat{\gamma}E\left[\int_{0}^{t}\left\|\Delta{\cal
L}^{(1)}_{i_{l},\delta}(s,x,U^{1},V^{1},\bar{V}^{1},\tilde{V}^{1})\right\|^{2}e^{2\gamma
s}ds\right]
\nonumber\\
&&+\hat{\gamma}\left[\int_{t}^{T}\left\|\Delta\bar{{\cal
L}}^{(1)}_{i_{l},\delta}(s,x,U^{1},V^{1},\bar{V}^{1},\tilde{V}^{1})\right\|^{2}e^{2\gamma
s}ds\right]. \nonumber
\end{eqnarray}
Furthermore, by \eq{pvitod}-\eq{epvitod} and the
Burkholder-Davis-Gundy's inequality (see, e.g., Theorem 48 in page
193 of Protter~\cite{pro:stoint}), we have the following
observation,
\begin{eqnarray}
&&E\left[\sup_{0\leq t\leq T}\left|M_{\delta}(t,x)\right|\right]
\elabel{emdelta}\\
&&\leq \hat{\gamma}K_{1}E\left[\int_{0}^{t}\left\|\Delta{\cal
L}^{(1)}_{i_{l},\delta}(s,x,U^{1},V^{1},\bar{V}^{1},\tilde{V}^{1})\right\|^{2}e^{2\gamma
s}ds\right]
\nonumber\\
&&\;\;\;+\hat{\gamma}K_{1}\left[\int_{t}^{T}\left\|\Delta\bar{{\cal
L}}^{(1)}_{i_{l},\delta}(s,x,U^{1},V^{1},\bar{V}^{1},\tilde{V}^{1})\right\|^{2}e^{2\gamma
s}ds\right], \nonumber
\end{eqnarray}
where $K_{1}$ is some nonnegative constant depending only on
$K_{D,0}$, $K_{D,1}$, $T$, and $d$. Note that, the detailed
estimation procedure for the quantity on the right-hand side of
\eq{emdelta} is postponed to the same argument used for \eq{vitodII}
in the proof of Lemma~\ref{lemmathree} since more exact calculations
are required there.

Next, for each fixed $t\in[0,T]$, $x\in D$, and $\sigma>0$, consider
the random variable set $\{Z_{\delta}(t,x)$,
$\delta\in[0,\sigma]\}$. It follows from Lemma 1.3 in pages 6-7 of
Peskir and Shiryaev~\cite{pesshi:optsto} that there is a countable
subset ${\cal C}=\{\delta_{1},\delta_{2},...\}\subset[0,\sigma]$
such that
\begin{eqnarray}
&&\mbox{esssup}_{\delta\in[0,\sigma]}Z_{\delta}(t,x)=\sup_{\delta\in{\cal
C}}Z_{\delta}(t,x),\;\;\mbox{a.s.}, \elabel{esssupeq}
\end{eqnarray}
where ``esssup'' denotes the essential supremum. Furthermore, take
\begin{eqnarray}
\left\{\begin{array}{ll}
\bar{Z}_{\delta_{1}}(t,x)=Z_{\delta_{1}}(t,x),\;\;\\
\bar{Z}_{\delta_{n+1}}(t,x)=\bar{Z}_{\delta_{n}}(t,x)\vee
Z_{\delta_{n+1}}(t,x)\;\;\mbox{for}\;\; n\in\{1,2,...\},
\end{array}
\right. \elabel{maxbarzz}
\end{eqnarray}
where $\alpha\vee\beta=\max\{\alpha,\beta\}$ for any two real
numbers $\alpha$ and $\beta$. Obviously,
\begin{eqnarray}
&&\left\{\begin{array}{ll} Z_{\delta}(t,x)\leq\bar{Z}_{\delta}(t,x)
&\mbox{for each}\;\;\delta\in{\cal C}\\
\bar{Z}_{\delta_{1}}(t,x)\leq\bar{Z}_{\delta_{2}}(t,x)&\mbox{for
any}\;\; \delta_{1},\delta_{2}\in{\cal C}\;\;\mbox{satisfying}\;\;
\delta_{1}\leq\delta_{2}.
\end{array}
\right. \elabel{maxbarzzI}
\end{eqnarray}
The second inequality in \eq{maxbarzzI} implies that the set
$\left\{\bar{Z}_{\delta}(t,x),\delta\in{\cal C}\right\}$ is upwards
directed. Hence, for each $t\in[0,T]$, $x\in D$, $\sigma>0$, and the
associated sequence of $\{\delta_{n},n=1,2,...\}$, it follows from
\eq{esssupeq} that
\begin{eqnarray}
&&E\left[\mbox{esssup}_{0\leq\delta\leq\sigma}Z_{\delta}(t,x)\right]
\elabel{difsuco}\\
&\leq&E\left[\mbox{esssup}_{\delta\in{\cal C}}
\bar{Z}_{\delta}(t,x)\right]
\nonumber\\
&=&\lim_{n\rightarrow\infty}E\left[\bar{Z}_{\delta_{n}}(t,x)\right]
\nonumber\\
&=&\lim_{n\rightarrow\infty}
E\left[\max_{\delta\in\{\delta_{1},...,\delta_{n}\}}Z_{\delta}(t,x)\right].
\nonumber
\end{eqnarray}
In addition, for each fixed $n\in\{2,3,...\}$, let
\begin{eqnarray}
&&\bar{M}_{\delta_{n}}(t,x)
=M_{\delta_{n}}(t,x)I_{\{Z_{\delta_{n}}\geq\bar{Z}_{\delta_{n-1}}\}}
+M_{\delta_{n-1}}(t,x)I_{\{Z_{\delta_{n}}<\bar{Z}_{\delta_{n-1}}\}}.
\end{eqnarray}
Thus, by the induction method in terms of $n\in\{1,2,...\}$ and
\eq{pvitod}, we know that
\begin{eqnarray}
&&E\left[\max_{\delta\in\{\delta_{1},...,\delta_{n}\}}Z_{\delta}(t,x)\right]
\elabel{difsucoI}\\
&\leq&\hat{\gamma}\lim_{n\rightarrow\infty}
E\left[\int_{0}^{t}\max_{\delta\in\{\delta_{1},...,\delta_{n}\}}
\left\|\Delta{\cal
L}^{(1)}_{i_{l},\delta}(s,x,U^{1},V^{1},\bar{V}^{1},\tilde{V}^{1})\right\|^{2}e^{2\gamma
s}ds\right.
\nonumber\\
&&\left.+\int_{t}^{T}\max_{\delta\in\{\delta_{1},...,\delta_{n}\}}
\left\|\Delta\bar{{\cal
L}}^{(1)}_{i_{l},\delta}(s,x,U^{1},V^{1},\bar{V}^{1},\tilde{V}^{1})\right\|^{2}e^{2\gamma
s}ds\right]
\nonumber\\
&&-\lim_{n\rightarrow\infty}E\left[\bar{M}_{\delta_{n}}(t,x)\right]
\nonumber\\
&\leq& KE\left[\int_{0}^{t}\mbox{esssup}_{0\leq\delta\leq\sigma}
\left\|\Delta{\cal
L}^{(1)}_{i_{l},\delta}(s,x,U^{1},V^{1},\bar{V}^{1},\tilde{V}^{1})\right\|^{2}e^{2\gamma
s}ds\right.\nonumber\\
&&\;\;\;\;\;\left.+\int_{t}^{T}\mbox{esssup}_{0\leq\delta\leq\sigma}
\left\|\Delta\bar{{\cal
L}}^{(1)}_{i_{l},\delta}(s,x,U^{1},V^{1},\bar{V}^{1},\tilde{V}^{1})\right\|^{2}e^{2\gamma
s}ds\right] \nonumber\\
&&\;\;\;\;\;+\int_{0}^{T}\mbox{esssup}_{0\leq\delta\leq\sigma}\left\|\Delta{\cal
J}^{(1)}_{i_{l},\delta}(s^{-},x,U^{1},V^{1},\bar{V}^{1},\tilde{V}^{1})\right\|^{2}
e^{2\gamma s}ds
\nonumber\\
&&\;\;\;\;\;+\left.\int_{0}^{T}\sum_{i=1}^{h}\int_{{\cal Z}}
\mbox{esssup}_{0\leq\delta\leq\sigma}\left\|\Delta{\cal
I}^{(1)}_{i,i_{l},\delta}(s^{-},x,U^{1},V^{1},\bar{V}^{1},\tilde{V}^{1},z_{i})\right\|^{2}e^{2\gamma
s}\lambda_{i}\nu_{i}(dz_{i})ds\right], \nonumber
\end{eqnarray}
where $K$ is a nonnegative constant depending only on $K_{D,0}$,
$d$, $T$, and $\gamma$. Note that, in the second inequality, we have
used the fact in \eq{epvitod} and the following observation
\begin{eqnarray}
&\left|E\left[\bar{M}_{\delta_{n}}(t,x)\right]\right| \leq
E\left[\sup_{t\in[0,T]}\left\|M_{\delta_{n}}(t,x)\right\|\right]
+E\left[\sup_{t\in[0,T]}\left\|M_{\delta_{n-1}}(t,x)\right\|\right].
\elabel{expectc}
\end{eqnarray}

Now, recall the condition that
\begin{eqnarray}
&(U^{1}(\cdot,x),V^{1}(\cdot,x),\bar{V}^{1}(\cdot,x),\tilde{V}^{1}(\cdot,x,z))\in
{\cal Q}^{2}_{{\cal F}}([0,T]\times D). \nonumber
\end{eqnarray}
Then, for each $x\in D$, $z\in{\cal Z}^{h}$, any $c\in\{0,1,...\}$,
and any small number $\xi$ such that $x+\xi e_{l}\in D$, we have
that
\begin{eqnarray}
&&\left\|(U^{1,(c)}(t,x+\xi e_{l}),V^{1,(c)}(t,x+\xi
e_{l}),\bar{V}^{1,(c)}(t,x+\xi e_{l}),\tilde{V}^{1,(c)}(t,x+\xi
e_{l},z))\right\|
\elabel{dominatedb}\\
&\leq&\left\|\left(\max_{x\in
D}\left\|U^{1,(c)}(t,x)\right\|,\max_{x\in
D}\left\|V^{1,(c)}(t,x)\right\|,\max_{x\in
D}\left\|\bar{V}^{1,(c)}(t,x)\right\|,\max_{x\in
D}\left\|\tilde{V}^{1,(c)}(t,x,z)\right\|\right)\right\|. \nonumber
\end{eqnarray}
Note that the related quantities on the right-hand side of
\eq{dominatedb} are squarely integrable a.s. in term of the Lebesgue
measure and/or the L\'evy measure. Therefore,
$\tilde{V}^{1}(t,x,\cdot)$ (the integration of
$\tilde{V}^{1}(t,x,z)$ with respect to the L\'evy measure) is also
infinitely smooth in each $x\in D$ due to the Lebesgue's dominated
convergence theorem. Thus, by the mean-value theorem, there exist
some constants $\xi_{1}\in(0,\delta)$ and $\xi\in(0,\xi_{1})$, which
depend on $\delta$, such that
\begin{eqnarray}
&&\Delta{\cal
A}^{(1)}_{i_{l},\delta}(t,x,U^{1},V^{1},\bar{V}^{1},\tilde{V}^{1})
\elabel{mvthe}\\
&=&\xi_{1}{\cal A}^{(2)}_{i_{l}}(t,x+\xi e_{l},U^{1}(t,x+\xi
e_{l}),V^{1}(t,x+\xi e_{l}),\bar{V}^{1}(t,x+\xi
e_{l}),\tilde{V}^{1}(t,x+\xi e_{l},\cdot)) \nonumber
\end{eqnarray}
a.s. for each ${\cal A}\in\{{\cal L},{\cal J},\bar{{\cal L}}\}$. Due
to \eq{mvthe}, \eq{blipschitz}, and \eq{dominatedb}, the quantity on
the left-hand side of \eq{mvthe} for all $\delta$ is dominated by a
squarely-integrable random variable in terms of the product measure
$dt\times dP$.
Similarly, for ${\cal A}=\bar{{\cal J}}$ and each $z\in{\cal
Z}^{h}$, we a.s. have that
\begin{eqnarray}
&&\Delta{\cal
A}^{(1)}_{i_{l},\delta}(t,x,U^{1},V^{1},\bar{V}^{1},\tilde{V}^{1},z)
\elabel{amvthe}\\
&=&\xi_{1}{\cal A}^{(2)}_{i_{l}}(t,x+\xi e_{l},U^{1}(t,x+\xi
e_{l}),V^{1}(t,x+\xi e_{l}),\bar{V}^{1}(t,x+\xi
e_{l}),\tilde{V}^{1}(t,x+\xi e_{l},z),z). \nonumber
\end{eqnarray}
Owing to \eq{mvthe}, \eq{nlipo}, and \eq{dominatedb}, the quantity
on the left-hand side of \eq{amvthe} for all $\delta$ is dominated
by a squarely-integrable random variable in terms of the product
measure $dt\times\nu(dz)\times dP$.
Therefore, it follows from \eq{difsuco}-\eq{difsucoI} and the
Lebesgue's dominated convergence theorem that
\begin{eqnarray}
&&\lim_{\sigma\rightarrow
0}E\left[\mbox{esssup}_{0\leq\delta\leq\sigma}Z_{\delta}(t,x)\right]
\elabel{difsucI}\\
&\leq& KE\left[\int_{0}^{t}\lim_{\sigma\rightarrow
0}\mbox{esssup}_{0\leq\delta\leq\sigma} \left\|\Delta{\cal
L}^{(1)}_{i_{l},\delta}(s,x,U^{1},V^{1},\bar{V}^{1},\tilde{V}^{1})\right\|^{2}e^{2\gamma
s}ds\right.
\nonumber\\
&&\;\;\;\;\;\;\;+\int_{t}^{T}\lim_{\sigma\rightarrow
0}\mbox{esssup}_{0\leq\delta\leq\sigma}\left\|\Delta\bar{{\cal
L}}^{(1)}_{i_{l},\delta}(s,x,U^{1},V^{1},\bar{V}^{1},\tilde{V}^{1})\right\|^{2}e^{2\gamma
s}ds
\nonumber\\
&&\;\;\;\;\;\;\;+\int_{0}^{T}\lim_{\sigma\rightarrow
0}\mbox{esssup}_{0\leq\delta\leq\sigma}\left\|\Delta{\cal
J}^{(1)}_{i_{l},\delta}(s^{-},x,U^{1},V^{1},\bar{V}^{1},\tilde{V}^{1})\right\|
e^{2\gamma s}ds
\nonumber\\
&&\;\;\;\;\;\;\;+\left.\int_{0}^{T}\sum_{i=1}^{h}\int_{{\cal Z}}
\lim_{\sigma\rightarrow
0}\mbox{esssup}_{0\leq\delta\leq\sigma}\left\|\Delta{\cal
I}^{(1)}_{i_{l},\delta}(s^{-},x,U^{1},V^{1},\bar{V}^{1},\tilde{V}^{1},z_{i})\right\|e^{2\gamma
s}\lambda_{i}\nu_{i}(dz_{i})ds\right]. \nonumber
\end{eqnarray}
Hence, by \eq{difsucI} and the Fatou's lemma, we know that, for any
sequence $\sigma_{n}$ satisfying $\sigma_{n}\rightarrow 0$ along
$n\in{\cal N}$, there is a subsequence ${\cal N}'\subset{\cal N}$
such that
\begin{eqnarray}
&&\mbox{esssup}_{0\leq\delta\leq\sigma_{n}}Z_{\delta}(t,x))\rightarrow
0\;\;\mbox{along}\;\;n\in{\cal N}'\;\;\mbox{a.s.} \elabel{zetazero}
\end{eqnarray}
The convergence in \eq{zetazero} implies that the first-order
derivatives of $U$ and $V$ in terms of $x_{l}$ for each
$l\in\{1,...,p\}$ exists. More exactly, they equal
$U_{i_{l}}^{(1)}(t,x)$ and $V_{i_{l}}^{(1)}(t,x)$ a.s. respectively
for each $t\in[0,T]$ and $x\in D$. Furthermore, they are $\{{\cal
F}_{t}\}$-adapted.

Now, we prove the claim for $\bar{V}$. In fact, it follows from the
proof as in \eq{difsuco}-\eq{difsucoI} that
\begin{eqnarray}
&&\lim_{\sigma\rightarrow
0}E\left[\int_{t}^{T}\mbox{esssup}_{0\leq\delta\leq\sigma}
\mbox{Tr}\left(\Delta\bar{{\cal
J}}^{(1)}_{i_{l},\delta}(s,x,U^{1},V^{1},\bar{V}^{1},\tilde{V}^{1})\right.\right.
\elabel{difsucII}\\
&&\;\;\;\;\;\;\;\;\;\;\left.\left.+\Delta(\bar{V}^{1})^{(1)}_{i_{l},\delta}(s,x)
-\Delta\bar{V}^{(1)}_{i_{l},\delta}(s,x) \right)e^{2\gamma
s}ds\right]
\nonumber
\end{eqnarray}
is also bounded by the quantity on the right-hand side of
\eq{difsucI}. Thus, by \eq{zetazero} and \eq{difsucII}, we know that
\begin{eqnarray}
&&\lim_{\delta\rightarrow
0}\Delta\bar{V}^{(1)}_{i_{l},\delta}(t,x)\nonumber\\
&=&\lim_{\delta\rightarrow 0}\left(\Delta\bar{{\cal
J}}^{(1)}_{i_{l},\delta}(t,x,U^{1},V^{1},\bar{V}^{1},\tilde{V}^{1})
+\Delta(\bar{V}^{1})^{(1)}_{i_{l},\delta}(t,x)\right)
\nonumber\\
&=&0,\;\;\;\mbox{a.s.} \nonumber
\end{eqnarray}
Hence, the first-order derivative of $\bar{V}$ in $x_{l}$ for each
$l\in\{1,...,p\}$ exists and equals $\bar{V}^{(1)}_{i_{l}}(t,x)$
a.s. for every $t\in[0,T]$ and $x\in D$. Furthermore, it is a
$\{{\cal F}_{t}\}$-predictable process. Similarly, we can get the
conclusion for $\tilde{V}_{i_{l}}^{(1)}(t,x,z)$ associated with each
$l$, $t$, $x$, and $z$.

Second, we suppose that $(U^{(c-1)}(t,x),V^{(c-1)}(t,x),$
$\bar{V}^{(c-1)}(t,x),$ $\tilde{V}^{(c-1)}(t,x,z))$ corresponding to
a given $(U^{1}(t,x),$ $V^{1}(t,x)$, $\bar{V}^{1}(t,x),$
$\tilde{V}^{1}(t,x,z))$ $\in{\cal Q}^{2}_{{\cal F}}([0,T]\times D)$
exists for any given $c\in\{1,2,...\}$. Then, we can show that
\begin{eqnarray}
&&\left(U^{(c)}(t,x),V^{(c)}(t,x), \bar{V}^{(c)}(t,x),
\tilde{V}^{(c)}(t,x,z)\right)
\elabel{cthderivative}
\end{eqnarray}
exists for the given $c\in\{1,2,...\}$.

In fact, consider any fixed nonnegative integer numbers
$i_{1},...,i_{p}$ satisfying $i_{1}+...+i_{p}=c-1$ for the given
$c\in\{1,2,...\}$. Take $f\in\{U,V,\bar{V}, \tilde{V}\}$,
$l\in\{1,...,p\}$, and sufficiently small $\delta>0$. Then,  let
\begin{eqnarray}
&&f_{i_{1}...(i_{l}+1)...i_{p},\delta}^{(c-1)}(t,x)\equiv
f_{i_{1}...i_{p}}^{(c-1)}(t,x+\delta e_{l})\elabel{fdeltan}
\end{eqnarray}
correspond to the $(c-1)$th-order partial derivative ${\cal
A}^{(c-1)}_{i_{1}...i_{p}}(s,x+\delta e_{l},U^{1}(s,x+\delta
e_{l}),V^{1}(s,x+\delta e_{l}))$ of ${\cal A}\in\{{\cal L},{\cal
J},{\cal I},\bar{{\cal L}},\bar{{\cal J}},\bar{{\cal I}}\}$ via
\eq{sigmanV}. Similarly, let
\begin{eqnarray}
(U^{(c)}_{i_{1}...(i_{l}+1)...i_{p}}(t,x),
V^{(c)}_{i_{1}...(i_{l}+1)...i_{p}}(t,x),
\bar{V}^{(c)}_{i_{1}...(i_{l}+1)...i_{p}}(t,x),
\tilde{V}^{(c)}_{i_{1}...(i_{l}+1)...i_{p}}(t,x,z))
\nonumber
\end{eqnarray}
be defined by \eq{sigmanV}, where ${\cal A}\in\{{\cal L},{\cal
J},{\cal I},\bar{{\cal L}},\bar{{\cal J}},\bar{{\cal I}}\}$ are
replaced by their $c$th-order partial derivatives ${\cal
A}^{(c)}_{i_{1}...(i_{l}+1)...i_{p}}$ corresponding to a given
$t,x$, $U^{1}(t,x)$, $V^{1}(t,x)$, $\bar{V}^{1}(t,x)$,
$\tilde{V}^{1}(t,x,z)$. Furthermore, let
\begin{eqnarray}
&&\Delta f^{(c)}_{i_{1}...(i_{l}+1)...i_{p},\delta}(t,x)
=\frac{f^{(c-1)}_{i_{1}...(i_{l}+1)...i_{p},\delta}(t,x)
-f^{(c-1)}_{i_{1}...i_{p}}(t,x)}{\delta}
-f^{(c)}_{i_{1}...(i_{l}+1)...i_{p}}(t,x) \elabel{ipartialII}
\end{eqnarray}
for each
$f\in\{U,V,\bar{V},\tilde{V},U^{1},V^{1},\bar{V}^{1},\tilde{V}^{1}\}$.
Then, define
\begin{eqnarray}
&&\Delta{\cal
A}^{(c)}_{i_{1}...(i_{l}+1)...i_{p},\delta}(t,x,U^{1},V^{1})
\elabel{ideltasqr}\\
&\equiv&\frac{1}{\delta}\left({\cal
A}^{(c-1)}_{i_{1}...i_{p}}(t,x+\delta e_{l},U^{1}(t,x+\delta
e_{l})-V^{1}(t,x+\delta e_{l}),\cdot)\right.
\nonumber\\
&&\left.\;\;\;\;\;-{\cal
A}^{(c-1)}_{i_{1}...i_{p}}(s,x,U^{1}(s,x),V^{1}(s,x)\cdot)\right)
\nonumber\\
&&-{\cal
A}^{(c)}_{i_{1}...(i_{l}+1)...i_{p}}(s,x,U^{1}(s,x),V^{1}(s,x)\cdot)
\nonumber
\end{eqnarray}
for each ${\cal A}\in\{{\cal L},{\cal J},{\cal I},\bar{{\cal
L}},\bar{{\cal J}},\bar{{\cal I}}\}$. Thus, by the It$\hat{o}$'s
formula and repeating the procedure as used in the first step, we
know that
\begin{eqnarray}
(U^{(c)}_{i_{1}...(i_{l}+1)...i_{p}}(t,x),
V^{(c)}_{i_{1}...(i_{l}+1)...i_{p}}(t,x),
\bar{V}^{(c)}_{i_{1}...(i_{l}+1)...i_{p}}(t,x),
\tilde{V}^{(c)}_{i_{1}...(i_{l}+1)...i_{p}}(t,x,z))
\nonumber
\end{eqnarray}
exist for the given $c\in\{1,2,...\}$ and all $l\in\{1,...,p\}$.
Therefore, the claim in \eq{cthderivative} is true.

Third, by the induction method with respect to $c\in\{1,2,...\}$ and
the continuity of all partial derivatives in terms of $x\in D$, we
know that the claims in the lemma are true. Hence, we finish the
proof of Lemma~\ref{differentiableV}. $\Box$
\end{proof}

\vskip0.4cm To state and prove the next lemma, let $D_{{\cal
F}}^{2}([0,T],C^{\infty}(D,R^{l}))$ with $l\in\{r,q\}$ be the set of
$R^{l}$-valued $\{{\cal F}_{t}\}$-adapted and squarely integrable
c\`adl\`ag processes as in \eq{adaptednormI}. Furthermore, for any
given number sequence $\gamma=\{\gamma_{c},c=0,1,2,...\}$ with
$\gamma_{c}\in R$, define ${\cal M}^{D}_{\gamma}[0,T]$ to be the
following Banach space (see, e.g., the related explanation in Yong
and Zhou~\cite{yonzho:stocon}, and Situ~\cite{sit:solbac})
\begin{eqnarray}
{\cal M}^{D}_{\gamma}[0,T]&\equiv&D^{2}_{{\cal
F}}([0,T],C^{\infty}(D,R^{r}))
\elabel{combanach}\\
&&\times D^{2}_{{\cal
F}}([0,T],C^{\infty}(D,R^{q})) \nonumber\\
&&\times L^{2}_{{\cal
F},p}([0,T],C^{\infty}(D,R^{q\times
d}))\nonumber\\
&&\times L^{2}_{p}([0,T]\times R_{+}^{h},C^{\infty}(D,R^{q\times
h})), \nonumber
\end{eqnarray}
which is endowed with the norm
\begin{eqnarray}
\;\;\;\left\|(U,V,\bar{V},\tilde{V})\right\|^{2}_{{\cal
M}^{D}_{\gamma}}&\equiv&
\sum_{c=0}^{\infty}\xi(c)\left\|(U,V,\bar{V},\tilde{V})
\right\|^{2}_{{\cal M}^{D}_{\gamma_{c},c}} \elabel{comnorm}
\end{eqnarray}
for any given $(U,V,\bar{V},\tilde{V})\in{\cal
M}^{D}_{\gamma}[0,T]$, and
\begin{eqnarray}
\;\;\;\;\;\;\;\;\left\|(U,V,\bar{V},\tilde{V})\right\|^{2}_{{\cal
M}^{D}_{\gamma_{c},c}}&=&E\left[\sup_{0\leq t\leq
T}\left\|U(t)\right\|^{2}_{C^{c}(D,q)}e^{2\gamma_{c}t}\right]
\elabel{kcomnorm}\\
&&+E\left[\sup_{0\leq t\leq
T}\left\|V(t)\right\|^{2}_{C^{c}(D,q)}e^{2\gamma_{c}t}\right]
\nonumber\\
&&+E\left[\int_{0}^{T}
\left\|\bar{V}(t)\right\|^{2}_{C^{c}(D,qd)}e^{2\gamma_{c}t}dt\right]
\nonumber\\
&&+E\left[\int_{0}^{T}\left\|\tilde{V}(t)
\right\|^{2}_{\nu,c}e^{2\gamma_{c}t}dt\right]. \nonumber
\end{eqnarray}
Then, we have the following lemma.
\begin{lemma}\label{lemmathree}
Under the conditions of Theorem~\ref{bsdeyI}, all the claims in the
theorem are true.
\end{lemma}
\begin{proof}
By \eq{sigmanV}, we can define the following map
\begin{eqnarray}
\;\;\Xi:\;\;(U^{1}(\cdot,x),V^{1}(\cdot,x),\bar{V}^{1}(\cdot,x),
\tilde{V}^{1}(\cdot,x,z))\rightarrow(U(\cdot,x),V(\cdot,x),
\bar{V}(\cdot,x),\tilde{V}(\cdot,x,z)). \nonumber
\end{eqnarray}
Then, we show that $\Xi$ forms a contraction mapping in ${\cal
M}^{D}_{\gamma}[0,T]$. In fact, consider
\begin{eqnarray}
\;\;(U^{i}(\cdot,x),V^{i}(\cdot,x),\bar{V}^{i}(\cdot,x),
\tilde{V}^{i}(\cdot,x,z))\in{\cal M}^{D}_{\gamma}[0,T] \nonumber
\end{eqnarray}
for each $i\in\{1,2,...\}$, satisfying
\begin{eqnarray}
&&(U^{i+1}(\cdot,x),V^{i+1}(\cdot,x),\bar{V}^{i+1}(\cdot,x),
\tilde{V}^{i+1}(\cdot,x,z))\nonumber\\
&=&\Xi(U^{i}(\cdot,x),V^{i}(\cdot,x),
\bar{V}^{i}(\cdot,x),\tilde{V}^{i}(\cdot,x,z)). \nonumber
\end{eqnarray}
Furthermore, define
\begin{eqnarray}
&&\Delta f^{i}=f^{i+1}-f^{i}\;\;\;\mbox{with}\;\;\;
f\in\{U,V,\bar{V},\tilde{V}\}
\nonumber
\end{eqnarray}
and take
\begin{eqnarray}
&&\zeta(\Delta U^{i}(t,x)+\Delta
V^{i}(t,x))=\left(\mbox{Tr}\left(\Delta
U^{i}(t,x)\right)+\mbox{Tr}\left(\Delta
V^{i}(t,x)\right)\right)e^{2\gamma_{0}t}. \elabel{firstzeta}
\end{eqnarray}
Thus, it follows from \eq{blipschitz} and the similar argument as
used in proving \eq{pvitod} that, for a $\gamma_{0}>0$ and each
$i\in\{2,3,...\}$,
\begin{eqnarray}
&&\zeta(\Delta U^{i}(t,x)+\Delta
V^{i}(t,x))\elabel{eupvitod}\\
&&+\int_{t}^{T}\mbox{Tr}\left(\Delta\bar{{\cal
J}}(s,x,U^{i},V^{i},\bar{V}^{i},\tilde{V}^{i},U^{i-1},V^{i-1},\bar{V}^{i-1},\tilde{V}^{i-1})\right.
\nonumber\\
&&\;\;\;\;\;\;\;\;\;\;\;\;\;\;\;\;\;\left.+\Delta\bar{V}^{i-1}(s,x)-\Delta\bar{V}^{i}(s,x)
\right)e^{2\gamma_{0}s}ds \nonumber\\
&&+\sum_{j=1}^{h}\int_{t}^{T} \int_{{\cal
Z}}\left(\mbox{Tr}\left(\Delta\bar{{\cal
I}}(s^{-},x,U^{i},V^{i},\bar{V}^{i},\tilde{V}^{i},U^{i-1},V^{i-1},\bar{V}^{i-1},\tilde{V}^{i-1},z)
\right.\right.
\nonumber\\
&&\;\;\;\;\;\;\;\;\;\;\;\;\;\;\;\;\;
\left.\left.+\Delta\tilde{V}^{i-1}(s^{-},x,z)-\Delta\tilde{V}^{i}(s^{-},x,z)\right)\right)_{j}
e^{2\gamma_{0}s}N_{j}(\lambda_{j}ds,dz_{j})
\nonumber\\
&\leq&\hat{\gamma}_{0}\left(\int_{0}^{t}\left\|\Delta{\cal
L}(s,x,U^{i},V^{i},\bar{V}^{i},\tilde{V}^{i},U^{i-1},V^{i-1},\bar{V}^{i-1},\tilde{V}^{i-1})
\right\|^{2}e^{2\gamma_{0}s}ds\right.
\nonumber\\
&&\left.+\int_{t}^{T}\left\|\Delta\bar{{\cal
L}}(s,x,U^{i},V^{i},\bar{V}^{i},\tilde{V}^{i},U^{i-1},V^{i-1},\bar{V}^{i-1},\tilde{V}^{i-1})
\right\|^{2}e^{2\gamma_{0}s}ds\right)
\nonumber\\
&&-M^{i}(t,x)
\nonumber\\
&\leq&\hat{\gamma}_{0}K_{a,0}N^{i-1}(t)-M^{i}(t,x),
\nonumber
\end{eqnarray}
where $K_{a,0}$ is some nonnegative constant depending only on
$K_{D,0}$. For the last inequality in \eq{eupvitod}, we have taken
\begin{eqnarray}
&&\hat{\gamma}_{0}=\frac{1}{2\gamma_{0}}>0. \elabel{euthgamma}
\end{eqnarray}
Furthermore, $N^{i-1}(t)$ appeared in \eq{eupvitod} is given by
\begin{eqnarray}
&&N^{i-1}(t)
\elabel{euntt}\\
&=&\int_{0}^{t}\left\|\Delta
U^{i-1}(s)\right\|^{2}_{C^{k}(D,r)}e^{2\gamma_{0}s}ds
\nonumber\\
&&+\int_{t}^{T}\left(\left\|\Delta
V^{i-1}(s)\right\|^{2}_{C^{k}(D,q)}+\left\|\Delta\bar{V}^{i-1}(s)
\right\|^{2}_{C^{k}(D,qd)}+\left\|\Delta\tilde{V}^{i-1}(s)
\right\|^{2}_{\nu,k}\right)e^{2\gamma_{0}s}ds.
\nonumber
\end{eqnarray}
In addition, $M^{i}(t,x)$ in \eq{eupvitod} is a martingale of the
form,
\begin{eqnarray}
&&\;\;\;M^{i}(t,x)=\elabel{eumsigman}\\
&&-2\sum_{j=1}^{d}\int_{0}^{t}\left(\Delta
U^{i}(s^{-},x)\right)'\nonumber\\
&&\;\;\;\;\;\;\;\;\;\;\;\;\;\;\;\;\Delta{\cal
J}_{j}(s^{-},x,U^{i},V^{i},\bar{V}^{i},\tilde{V}^{i},U^{i-1},V^{i-1},\bar{V}^{i-1},\tilde{V}^{i-1})
e^{2\gamma s}dW_{j}(s)
\nonumber\\
&&-2\sum_{j=1}^{h}\int_{0}^{t} \int_{{\cal Z}}\left(\Delta
U^{i}(s^{-},x)\right)'\nonumber\\
&&\;\;\;\;\;\;\;\;\;\;\;\;\;\;\;\;\Delta{\cal
I}_{j}(s^{-},x,U^{i},V^{i},\bar{V}^{i},\tilde{V}^{i},U^{i-1},V^{i-1},\bar{V}^{i-1},\tilde{V}^{i-1},z_{j})
e^{2\gamma s}\tilde{N}_{j}(\lambda_{j}ds,dz_{j})
\nonumber\\
&&+2\sum_{j=1}^{d}\int_{t}^{T}\left(\Delta
V^{i}(s^{-},x)\right)'\left(\Delta\bar{{\cal
J}}_{j}(s^{-},x,U^{i},V^{i},\bar{V}^{i},\tilde{V}^{i},U^{i-1},V^{i-1},\bar{V}^{i-1},\tilde{V}^{i-1})\right.
\nonumber\\
&&\;\;\;\;\;\;\;\;\;\;\;\;\;\;\;
\left.+(\Delta\bar{V}^{i-1})_{j}(s^{-},x)-(\Delta\bar{V}^{i})_{j}(s^{-},x)\right)
e^{2\gamma_{0}s}dW_{j}(s)
\nonumber\\
&&+2\sum_{j=1}^{h}\int_{t}^{T} \int_{{\cal Z}}\left((\Delta
V^{i})_{j}(s^{-},x)\right)'\left(\Delta\bar{{\cal
I}}_{j}(s^{-},x,U^{i},V^{i},\bar{V}^{i},\tilde{V}^{i},U^{i-1},V^{i-1},\bar{V}^{i-1},\tilde{V}^{i-1},z_{j})\right.
\nonumber\\
&&\;\;\;\;\;\;\;\;\;\;\;\;\;\;\;\left.+(\Delta\tilde{V}^{i-1})_{j}(s^{-},x,z_{j})
-(\Delta\tilde{V}^{i})_{j}(s^{-},x,z_{j})\right)
e^{2\gamma_{0}s}\tilde{N}_{j}(\lambda_{j}ds,dz_{j}). \nonumber
\end{eqnarray}
Then, it follows from \eq{eupvitod}-\eq{eumsigman} and the
martingale properties related to the It$\hat{o}$'s stochastic
integral that
\begin{eqnarray}
&&E\left[\left(\zeta(\Delta U^{i}(t,x)+\Delta
V^{i}(t,x))\right)e^{2\gamma_{0} t}\right.
\elabel{vitodI}\\
&&+\int_{t}^{T}\mbox{Tr}\left(\Delta\bar{{\cal
J}}(s,x,U^{i},V^{i},\bar{V}^{i},\tilde{V}^{i},U^{i-1},V^{i-1},\bar{V}^{i-1},\tilde{V}^{i-1})\right.
\nonumber\\
&&\;\;\;\;\;\;\;\;\;\;\left.+\Delta\bar{V}^{i-1}(s,x)-\Delta\bar{V}^{i}(s,x)
\right)e^{2\gamma_{0} s}ds
\nonumber\\
&&+\sum_{j=1}^{h}\int_{t}^{T}\int_{{\cal Z}}
\left(\mbox{Tr}\left(\Delta{\cal
I}(s,x,U^{i},V^{i},\bar{V}^{i},\tilde{V}^{i},U^{i-1},V^{i-1},\bar{V}^{i-1},\tilde{V}^{i-1},z)\right.\right.
\nonumber\\
&&\;\;\;\;\;\;\;\;\;\;
\left.\left.\left.+\Delta\tilde{V}^{i-1}(s^{-},x,z)-\Delta\tilde{V}^{i}(s^{-},x,z)
\right)\right)_{j}e^{2\gamma_{0}s}\lambda_{j}ds\nu_{j}(dz_{j})\right]
\nonumber\\
&\leq&\hat{\gamma}_{0}(T+1)K_{a,0}\left\|(\Delta
U^{i-1},V^{i-1},\Delta\bar{V}^{i-1},\Delta\tilde{V}^{i-1})
\right\|^{2}_{{\cal M}^{D}_{\gamma_{0},k}}. \nonumber
\end{eqnarray}
Next, it follows from \eq{eumsigman} that
\begin{eqnarray}
&&\;\;\;E\left[\sup_{0\leq t\leq T}\left|M^{i}(t,x)\right|\right]
\elabel{vitodII}\\
&\leq&2\sum_{j=1}^{d}E\left[\sup_{0\leq t\leq
T}\left|\int_{0}^{t}\left(\Delta
U^{i}(s^{-},x)\right)'\right.\right.\nonumber\\
&&\;\;\;\;\;\;\;\;\;\;\;\;\; \left.\left.\Delta{\cal
J}_{j}(s^{-},x,U^{i},V^{i},\bar{V}^{i},\tilde{V}^{i},U^{i-1},V^{i-1},\bar{V}^{i-1},\tilde{V}^{i-1})
e^{2\gamma_{0} s}dW_{j}(s)\right|\right]
\nonumber\\
&&+2\sum_{j=1}^{h}E\left[\sup_{0\leq t\leq T}\left|\int_{0}^{t}
\int_{{\cal Z}}\left(\Delta
U^{i}(s^{-},x)\right)'\right.\right.\nonumber\\
&&\;\;\;\;\;\;\;\;\;\;\;\;\;\;\;\left.\left.\Delta{\cal
I}_{j}(s^{-},x,U^{i},V^{i},\bar{V}^{i},\tilde{V}^{i},U^{i-1},V^{i-1},\bar{V}^{i-1},\tilde{V}^{i-1},z_{j})
e^{2\gamma_{0} s}\tilde{N}(\lambda_{j}ds,dz_{j})\right|\right]
\nonumber\\
&&+4\sum_{j=1}^{d}E\left[\sup_{0\leq t\leq
T}\left|\int_{0}^{t}\left(\Delta
V^{i}(s^{-},x)\right)'\right.\right.
\nonumber\\
&&\;\;\;\;\;\;\;\;\;\;\;\;\;\;\;\left(\Delta\bar{{\cal
J}}_{j}(s^{-},x,U^{i},V^{i},\bar{V}^{i},\tilde{V}^{i},U^{i-1},V^{i-1},\bar{V}^{i-1},\tilde{V}^{i-1})
\right.
\nonumber\\
&&\;\;\;\;\;\;\;\;\;\;\;\;\;\;\;\left.\left.\left.+(\Delta\bar{V}^{i-1})_{j}
(s^{-},x)-(\Delta\bar{V}^{i})_{j} (s^{-},x)\right)e^{2\gamma_{0}
s}dW_{j}(s)\right|\right]
\nonumber\\
&&+4\sum_{j=1}^{h}E\left[\sup_{0\leq t\leq T}\left|\int_{0}^{t}
\int_{{\cal Z}}\left(\Delta V^{i}(s^{-},x)\right)'\right.\right.
\nonumber\\
&&\;\;\;\;\;\;\;\;\;\;\;\;\;\;\;\left(\Delta\bar{{\cal
I}}_{j}(s^{-},x,U^{i},V^{i},\bar{V}^{i},\tilde{V}^{i},U^{i-1},V^{i-1},\bar{V}^{i-1},\tilde{V}^{i-1},z_{j})\right.
\nonumber\\
&&\;\;\;\;\;\;\;\;\;\;\;\;\;\;\;\left.\left.\left.+(\Delta\tilde{V}^{i-1})_{j}(s^{-},x,z_{j})
-(\Delta\tilde{V}^{i})_{j}(s^{-},x,z_{j})\right)e^{2\gamma_{0}
s}\tilde{N}(\lambda_{j}ds,dz_{j})\right|\right]. \nonumber
\end{eqnarray}
By the Burkholder-Davis-Gundy's inequality (see, e.g., Theorem 48 in
page 193 of Protter~\cite{pro:stoint}), the right-hand side of the
inequality in \eq{vitodII} is bounded by
\begin{eqnarray}
&&K_{b,0}\left(\sum_{j=1}^{d}E\left[\left(\int_{0}^{T}\left\|\Delta
U^{i}(s^{-},x)\right\|^{2}\right.\right.\right.\elabel{IvitodII}\\
&&\;\;\;\;\;\left.\left.\left\|(\Delta{\cal
J}^{i})_{j}(s^{-},x,U^{i},V^{i},\bar{V}^{i},\tilde{V}^{i},U^{i-1},V^{i-1},\bar{V}^{i-1},\tilde{V}^{i-1})\right\|^{2}e^{4\gamma_{0}
s}ds\right)^{\frac{1}{2}} \right]
\nonumber\\
&&+\sum_{j=1}^{h}E\left[\left(\int_{0}^{T} \int_{{\cal
Z}}\left\|\Delta U^{i}(s^{-},x)\right\|^{2}\right.\right.
\nonumber\\
&&\;\;\;\;\;\left.\left.\left\|\Delta{\cal
I}_{j}(s^{-},x,U^{i},V^{i},\bar{V}^{i},\tilde{V}^{i},U^{i-1},V^{i-1},\bar{V}^{i-1},\tilde{V}^{i-1},z_{j})\right\|^{2}
e^{4\gamma_{0}
s}\lambda_{j}\nu_{j}(dz_{j})ds\right)^{\frac{1}{2}}\right]
\nonumber\\
&&+\sum_{j=1}^{d}E\left[\left(\int_{0}^{T}\left\|\Delta
V^{i}(s^{-},x)\right\|^{2}\right.\right.\nonumber\\
&&\;\;\;\;\;\left\|(\Delta\bar{{\cal
J}}^{i})_{j}(s^{-},x,U^{i},V^{i},\bar{V}^{i},\tilde{V}^{i},U^{i-1},V^{i-1},\bar{V}^{i-1},\tilde{V}^{i-1})\right.
\nonumber\\
&&\;\;\;\;\;\left.\left.\left.+(\Delta\bar{V}^{i-1})_{j}
(s^{-},x)-(\Delta\bar{V}^{i})_{j}
(s^{-},x)\right\|^{2}e^{4\gamma_{0} s}ds\right)^{\frac{1}{2}}
\right]\nonumber\\
&&+\sum_{j=1}^{h}E\left[\left(\int_{0}^{T} \int_{{\cal
Z}}\left\|\Delta V^{i}(s^{-},x)\right\|^{2}\right.\right.
\nonumber\\
&&\;\;\;\;\;\left\|\Delta\bar{{\cal
I}}_{j}(s^{-},x,U^{i},V^{i},\bar{V}^{i},\tilde{V}^{i},U^{i-1},V^{i-1},\bar{V}^{i-1},\tilde{V}^{i-1},z_{j})\right.
\nonumber\\
&&\;\;\;\;\;\left.\left.\left.\left.+(\Delta\tilde{V}^{i})_{j}(s^{-},x,z_{j})
-(\Delta\tilde{V}^{i})_{j}(s^{-},x,z_{j})\right\|^{2}e^{4\gamma_{0}
s}\lambda_{j}\nu_{j}(dz_{j})ds\right)^{\frac{1}{2}}\right]\right),
\nonumber
\end{eqnarray}
where $K_{b,0}$ is some nonnegative constant depending only on
$K_{D,0}$ and $T$. Furthermore, it follows from the direct
observation that the quantity in \eq{IvitodII} is bounded by
\begin{eqnarray}
&&K_{b,0}\left(E\left[\left(\sup_{0\leq t\leq T}\|\Delta
U^{i}(t,x)\|^{2}e^{2\gamma_{0} t}\right)^{\frac{1}{2}}\right.\right.
\elabel{IIvitodII}\\
&&\left(\sum_{j=1}^{d}\left(\int_{0}^{T} \left\|\Delta{\cal
J}_{j}(s^{-},x,U^{i},V^{i},\bar{V}^{i},\tilde{V}^{i},U^{i-1},V^{i-1},\bar{V}^{i-1},\tilde{V}^{i-1})
\right\|^{2}e^{2\gamma_{0} s}ds\right)^{\frac{1}{2}}\right.
\nonumber\\
&&\;\;+\sum_{j=1}^{h}\left(\int_{0}^{T} \int_{{\cal
Z}}\left\|\Delta{\cal
I}_{j}(s^{-},x,U^{i},V^{i},\bar{V}^{i},\tilde{V}^{i},U^{i-1},V^{i-1},\bar{V}^{i-1},\tilde{V}^{i-1},z_{j})
\right\|^{2}\right.
\nonumber\\
&&\;\;\;\;\;\;\;\;\;\;\;\;\left.\left.\left.e^{2\gamma_{0}
s}\lambda_{j}\nu_{j}(dz_{j})ds\right)^{\frac{1}{2}} \right)\right]
\nonumber\\
&+&\left(E\left[\left(\sup_{0\leq t\leq T}\|\Delta
V^{i}(t,x)\|^{2}e^{2\gamma_{0} t}\right)^{\frac{1}{2}}\right.\right.
\nonumber\\
&&\left(\sum_{j=1}^{d}\left(\int_{0}^{T} \left\|\Delta\bar{{\cal
J}}_{j}(s^{-},x,U^{i},V^{i},\bar{V}^{i},\tilde{V}^{i},U^{i-1},V^{i-1},\bar{V}^{i-1},\tilde{V}^{i-1})\right.\right.\right.
\nonumber\\
&&\;\;\;\;\;\;\;\;\;
\left.\left.+(\Delta\bar{V}^{i-1})_{j}(s^{-},x)-(\Delta\bar{V}^{i})_{j}(s^{-},x)\right\|^{2}e^{2\gamma_{0}
s}ds\right)^{\frac{1}{2}}
\nonumber\\
&&\;\;+\sum_{j=1}^{h}\left(\int_{0}^{T} \int_{{\cal
Z}}\left\|\Delta\bar{{\cal
I}}_{j}(s^{-},x,U^{i},V^{i},\bar{V}^{i},\tilde{V}^{i},U^{i-1},V^{i-1},\bar{V}^{i-1},\tilde{V}^{i-1},z_{j})\right.\right.
\nonumber\\
&&\;\;\;\;\;\;\;\;\;\left.\left.\left.\left.\left.
+(\Delta\tilde{V}^{i-1})_{j}(s^{-},x,z_{j})-(\Delta\tilde{V}^{i})_{j}(s^{-},x,z_{j})\right\|^{2}
e^{2\gamma_{0}s}\lambda_{j}\nu_{j}(dz_{j})ds\right)^{\frac{1}{2}}
\right)\right]\right). \nonumber
\end{eqnarray}
In addition, by the direct computation, we know that the quantity in
\eq{IIvitodII} is dominated by
\begin{eqnarray}
&&\frac{1}{2}E\left[\sup_{0\leq t\leq T}\|\Delta
U^{i}(t,x)\|^{2}e^{2\gamma_{0} t}\right]
\elabel{IIIvitodII}\\
&&+dK_{b,0}^{2}E\left[\int_{0}^{T}\mbox{Tr}\left(\Delta{\cal
J}(s^{-},x,U^{i},V^{i},\bar{V}^{i},\tilde{V}^{i},U^{i-1},V^{i-1},\bar{V}^{i-1},\tilde{V}^{i-1}\right)
e^{2\gamma_{0} s}ds\right]
\nonumber\\
&&+K_{b,0}^{2}E\left[\sum_{j=1}^{h}\int_{0}^{T}\int_{{\cal Z}}
\mbox{Tr}\left(\Delta{\cal
I}(s^{-},x,U^{i},V^{i},\bar{V}^{i},\tilde{V}^{i},U^{i-1},V^{i-1},\bar{V}^{i-1},\tilde{V}^{i-1},z_{j})\right)_{j}
\right.
\nonumber\\
&&\;\;\;\;\;\;\;\;\;\;\;\;\;\left.e^{2\gamma_{0} s}\lambda_{j}\nu_{j}(dz_{j})ds\right]
\nonumber\\
&&+\frac{1}{2}E\left[\sup_{0\leq t\leq T}\|\Delta
V^{i}(t,x)\|^{2}e^{2\gamma_{0} t}\right]\nonumber\\
&&+dK_{b,0}^{2}E\left[\left(\int_{0}^{T}\mbox{Tr}\left(\Delta\bar{{\cal
J}}(s^{-},x,U^{i},V^{i},\bar{V}^{i},\tilde{V}^{i},U^{i-1},V^{i-1},\bar{V}^{i-1},\tilde{V}^{i-1})\right.\right.\right.
\nonumber\\
&&\;\;\;\;\;\;\;\;\;\;\;\;\;\;\;
\left.\left.\left.+\Delta\bar{V}^{i-1}(s^{-},x)-\Delta\bar{V}^{i}(s^{-},x)
\right)e^{2\gamma_{0} s}ds\right)\right]\nonumber\\
&&+K_{b,0}^{2}E\left[\sum_{j=1}^{h}\int_{0}^{T}\int_{{\cal Z}}
\mbox{Tr}\left(\Delta\bar{{\cal
I}}(s^{-},x,U^{i},V^{i},\bar{V}^{i},\tilde{V}^{i},U^{i-1},V^{i-1},\bar{V}^{i-1},\tilde{V}^{i-1},z_{j})\right.\right.
\nonumber\\
&&\;\;\;\;\;\;\;\;\;\;\;\;\;\;\;\left.\left.+\Delta\tilde{V}^{i-1}(s^{-},x,z)-\Delta\tilde{V}^{i}(s^{-},x,z_{j})\right)
e^{2\gamma_{0} s}\lambda_{j}\nu_{j}(dz_{j})ds\right].
\nonumber
\end{eqnarray}
Due to \eq{vitodI}, the quantity in \eq{IIIvitodII} is bounded by
\begin{eqnarray}
&&\frac{1}{2}\left(E\left[\sup_{0\leq t\leq T}\|\Delta
U^{i}(t)\|^{2}_{C^{0}(r)}e^{2\gamma_{0} t}\right]
+E\left[\sup_{0\leq t\leq T}\|\Delta
U^{i}(t)\|^{2}_{C^{0}(q)}e^{2\gamma_{0} t}\right]\right)
\elabel{IVvitodII}\\
&&+\hat{\gamma}_{0}(T+1)dK_{a,0}K_{b,0}^{2}\left\|(\Delta
U^{i-1},\Delta V^{i-1},\Delta\bar{V}^{i-1},\Delta\tilde{V}^{i-1})
\right\|^{2}_{{\cal M}^{D}_{\gamma_{0},k}}, \nonumber
\end{eqnarray}
where $K_{a,0}$ is some nonnegative constant depending only on $T$,
$d$, and $K_{D,0}$. Thus, it follows from \eq{blipschitz} and
\eq{eupvitod}-\eq{IVvitodII} that
\begin{eqnarray}
&&E\left[\sup_{0\leq t\leq T}\left\|\Delta
U^{i}(t)\right\|^{2}_{C^{0}(q)}e^{2\gamma_{0}
t}\right]+E\left[\sup_{0\leq t\leq T}\left\|\Delta
V^{i}(t)\right\|^{2}_{C^{0}(q)}e^{2\gamma_{0} t}\right]
\elabel{firstineu}\\
&&\leq
2\left(1+dK_{b,0}^{2}\right)K_{a,0}\hat{\gamma}_{0}(T+1)\left\|(\Delta
U^{i-1},\Delta V^{i-1},\Delta\bar{V}^{i-1},\Delta\tilde{V}^{i-1})
\right\|^{2}_{{\cal M}^{D}_{\gamma_{0},k}}. \nonumber
\end{eqnarray}
Furthermore, it follows from \eq{eupvitod} and \eq{blipschitz} that,
for $i\in\{3,4,...\}$,
\begin{eqnarray}
&&E\left[\int_{t}^{T}
\mbox{Tr}\left(\Delta\bar{V}^{i}(s,x)\right)e^{2\gamma_{0}
s}ds\right]
\elabel{secondineu}\\
&\leq& 2E\left[\int_{t}^{T}\mbox{Tr}\left(\Delta\bar{{\cal
J}}(s^{-},x,U^{i},V^{i},\bar{V}^{i},\tilde{V}^{i},U^{i-1},V^{i-1},\bar{V}^{i-1},
\tilde{V}^{i-1})\right.\right.
\nonumber\\
&&\;\;\;\;\;\;\;\;\;\;\;\;\;\;\;\;\;\;
\left.\left.+\Delta\bar{V}^{i-1}(s^{-},x)-\Delta\bar{V}^{i}(s,x)
\right)e^{2\gamma_{0} s}ds\right]\nonumber\\
&&+2E\left[\int_{t}^{T}\mbox{Tr}\left(\Delta\bar{{\cal
J}}(s^{-},x,U^{i},V^{i},\bar{V}^{i},\tilde{V}^{i},U^{i-1},V^{i-1},\bar{V}^{i-1},\tilde{V}^{i-1})
\right.\right.\nonumber\\
&&\;\;\;\;\;\;\;\;\;\;\;\;\;\;\;\;\;\;\left.\left.+\Delta\bar{V}^{i-1}(s^{-},x)\right)e^{2\gamma_{0}
s}ds\right]
\nonumber\\
&\leq&2\hat{\gamma}_{0}K_{C,0}\left(\left\|(\Delta U^{i-1},\Delta
V^{i-1},\Delta\bar{V}^{i-1},\Delta\tilde{V}^{i-1})
\right\|^{2}_{{\cal M}^{D}_{\gamma_{0},k}}\right.
\nonumber\\
&&\;\;\;\;\;\;\;\;\;\;\;\;\;\;
\left.+\left\|(\Delta U^{i-2},\Delta
V^{i-2},\Delta\bar{V}^{i-2},\Delta\tilde{V}^{i-2})
\right\|^{2}_{{\cal M}^{D}_{\gamma_{0},k}}\right), \nonumber
\end{eqnarray}
where $K_{C,0}$ is some nonnegative constant depending only on
$K_{D,0}$ and $T$. Similarly, it follows from \eq{nlipo} that
\begin{eqnarray}
&&E\left[\sum_{j=1}^{h}\int_{t}^{T}\int_{{\cal Z}}
\left(\mbox{Tr}\left(\Delta\tilde{V}^{i}(s^{-},x,z)
\right)\right)_{j}e^{2\gamma_{0}s}\lambda_{j}ds\nu_{j}(dz_{j})\right]
\elabel{thirdvitodI}\\
&\leq&2E\left[\sum_{j=1}^{h}\int_{t}^{T}\int_{{\cal Z}}
\left(\mbox{Tr}\left(\Delta\bar{{\cal
I}}(s,x,U^{i},V^{i},\bar{V}^{i},\bar{V}^{i},U^{i-1},V^{i-1},
\bar{V}^{i-1},\tilde{V}^{i-1},z)\right.\right.\right.
\nonumber\\
&&\;\;\;\;\;\;\;\;\;\;\;\;\;\;
\left.\left.\left.+\Delta\tilde{V}^{i-1}(s^{-},x,z)-\Delta\tilde{V}^{i}(s^{-},x,z)
\right)\right)_{j}e^{2\gamma_{0}s}\lambda_{j}ds\nu_{j}(dz_{j})\right]
\nonumber\\
&&+2E\left[\sum_{j=1}^{h}\int_{t}^{T}\int_{{\cal Z}}
\left(\mbox{Tr}\left(\Delta\bar{{\cal
I}}(s,x,U^{i},V^{i},\bar{V}^{i},\tilde{V}^{i},U^{i-1},V^{i-1},\bar{V}^{i-1},
\tilde{V}^{i-1},z)\right.\right.\right.
\nonumber\\
&&\;\;\;\;\;\;\;\;\;\;\;\;\;\;
\left.\left.\left.+\Delta\tilde{V}^{i-1}(s^{-},x,z)
\right)\right)_{j}e^{2\gamma_{0}s}\lambda_{j}ds\nu_{j}(dz_{j})\right]
\nonumber\\
&\leq&2\hat{\gamma}_{0}K_{C,0}\left(\left\|(\Delta U^{i-1},\Delta
V^{i-1},\Delta\bar{V}^{i-1},\Delta\tilde{V}^{i-1})
\right\|^{2}_{{\cal
M}^{D}_{\gamma_{0},k}}\right.
\nonumber\\
&&\;\;\;\;\;\;\;\;\;\;\;\;\;\; \left.+\left\|(\Delta U^{i-2},\Delta
V^{i-2},\Delta\bar{V}^{i-2},\Delta\tilde{V}^{i-2})
\right\|^{2}_{{\cal M}^{D}_{\gamma_{0},k}}\right). \nonumber
\end{eqnarray}
Thus, by \eq{eupvitod}, \eq{firstineu}-\eq{thirdvitodI}, and the
fact that all functions and norms used in this paper are continuous
in terms of $x$, we have
\begin{eqnarray}
&&\left\|(\Delta U^{i},\Delta
V^{i},\Delta\bar{V}^{i},\Delta\tilde{V}^{i}) \right\|^{2}_{{\cal
M}^{D}_{\gamma_{0},0}}
\elabel{lastine}\\
&\leq&\hat{\gamma}_{0}K_{d,0}\left(\left\|(\Delta U^{i-1},\Delta
V^{i-1},\Delta\bar{V}^{i-1},\Delta\tilde{V}^{i-1})
\right\|^{2}_{{\cal M}^{D}_{\gamma_{0},k}}\right.
\nonumber\\
&&\;\;\;\;\;\;\;\;\;\;\;\;\;\;\left.+\left\|(\Delta U^{i-2},\Delta
V^{i-2},\Delta\bar{V}^{i-2},\Delta\tilde{V}^{i-2})
\right\|^{2}_{{\cal M}^{D}_{\gamma_{0},k}}\right), \nonumber
\end{eqnarray}
where $K_{d,0}$ is some nonnegative constant depending only on
$K_{D,0}$ and $T$.

Now, by Lemma~\ref{differentiableV} and the similar construction as
in \eq{firstzeta}, for each $c\in\{1,2,...\}$, we can define
\begin{eqnarray}
&&\zeta(\Delta U^{c,i}(t,x)+\Delta
V^{c,i}(t,x))\equiv\left(\mbox{Tr}\left(\Delta
U^{c,i}(t,x)\right)+\mbox{Tr}\left(\Delta
V^{c,i}(t,x)\right)\right)e^{2\gamma_{c} t},
\elabel{secondzeta}
\end{eqnarray}
where
\begin{eqnarray}
\Delta U^{c,i}(t,x))&=&(\Delta U^{(0),i}(t,x)),\Delta
U^{(1),i}(t,x)),...,\Delta U^{(c),i}(t,x))', \nonumber\\
\Delta V^{c,i}(t,x))&=&(\Delta V^{(0),i}(t,x)),\Delta
V^{(1),i}(t,x)),...,\Delta V^{(c),i}(t,x))'. \nonumber
\end{eqnarray}
Then, it follows from the It$\hat{o}$'s formula and the similar
discussion for \eq{lastine} that
\begin{eqnarray}
&&\left\|(\Delta U^{i},\Delta
V^{i},\Delta\bar{V}^{i},\Delta\tilde{V}^{i}) \right\|^{2}_{{\cal
M}^{D}_{\gamma_{c},c}}
\elabel{clastine}\\
&\leq&\hat{\gamma}_{c}K_{d,c}\left(\left\|(\Delta U^{i-1},\Delta
V^{i-1},\Delta\bar{V}^{i-1},\Delta\tilde{V}^{i-1})
\right\|^{2}_{{\cal
M}^{D}_{\gamma_{c},k+c}}\right.\nonumber\\
&&\;\;\;\;\;\;\;\;\;\left.+\left\|(\Delta U^{i-2},\Delta
V^{i-2},\Delta\bar{V}^{i-2},\Delta\tilde{V}^{i-2})
\right\|^{2}_{{\cal M}^{D}_{\gamma_{c},k+c}}\right)
\nonumber\\
&\leq&\frac{\delta}{((c+1)^{10}(c+2)^{10}...(c+k)^{10})(\eta(c+1)
\eta(c+2)...\eta(c+k))}
\nonumber\\
&&\left(\left\|(\Delta U^{i-1},\Delta
V^{i-1},\Delta\bar{V}^{i-1},\Delta\tilde{V}^{i-1})
\right\|^{2}_{{\cal M}^{D}_{\gamma_{k+c},k+c}}\right.
\nonumber\\
&&\left.+\left\|(\Delta U^{i-2},\Delta
V^{i-2},\Delta\bar{V}^{i-2},\Delta\tilde{V}^{i-2})
\right\|^{2}_{{\cal M}^{D}_{\gamma_{k+c},k+c}}\right),\nonumber
\end{eqnarray}
where, for the last inequality of \eq{clastine}, we have taken the
number sequence $\gamma$ such that $\gamma_{0}<\gamma_{1}<...$ and
\begin{eqnarray}
&\hat{\gamma}_{c}K_{d,c}((c+1)^{10}(c+2)^{10}...(c+k)^{10})(\eta(c+1)
\eta(c+2)...\eta(c+k))\leq\delta\nonumber
\end{eqnarray}
for some $\delta>0$ such that $2\sqrt{e^{k}\delta}$ is sufficiently
small. Hence, we have
\begin{eqnarray}
&&\left\|(\Delta U^{i},\Delta
V^{i},\Delta\bar{V}^{i},\Delta\tilde{V}^{i})
\right\|^{2}_{{\cal M}^{D}_{\gamma}}
\elabel{infinityine}\\
&\leq& e^{k}\delta\left(\left\|(\Delta U^{i-1},\Delta
V^{i-1},\Delta\bar{V}^{i-1},\Delta\tilde{V}^{i-1})
\right\|^{2}_{{\cal
M}^{D}_{\gamma}}\right.
\nonumber\\
&&\;\;\;\;\;\left.+\left\|(\Delta U^{i-2},\Delta
V^{i-2},\Delta\bar{V}^{i-2},\Delta\tilde{V}^{i-2})
\right\|^{2}_{{\cal M}^{D}_{\gamma}}\right). \nonumber
\end{eqnarray}
Since $(a^{2}+b^{2})^{1/2}\leq a+b$ for $a,b\geq 0$, we have
\begin{eqnarray}
&&\left\|(\Delta U^{i},\Delta
V^{i},\Delta\bar{V}^{i},\Delta\tilde{V}^{i})
\right\|_{{\cal M}^{D}_{\gamma}}
\elabel{noninfinityine}\\
&\leq&\sqrt{e^{k}\delta}\left(\left\|(\Delta U^{i-1},\Delta
V^{i-1},\Delta\bar{V}^{i-1},\Delta\tilde{V}^{i-1})\right\|_{{\cal
M}^{D}_{\gamma}}\right.\nonumber\\
&&\;\;\;\;\;\;\;\;\left.+\left\|(\Delta U^{i-2},\Delta
V^{i-2},\Delta\bar{V}^{i-2},\Delta\tilde{V}^{i-2})\right\|_{{\cal
M}^{D}_{\gamma}}\right). \nonumber
\end{eqnarray}
Therefore,
by \eq{noninfinityine}, we know that
\begin{eqnarray}
&&\sum_{i=3}^{\infty}\left\|(\Delta U^{i},\Delta
V^{i},\Delta\bar{V}^{i},\Delta\tilde{V}^{i})\right\|_{{\cal
M}^{D}_{\gamma}} \elabel{seriescon}\\
&\leq&\frac{\sqrt{e^{k}\delta}}{1-2\sqrt{e^{k}\delta}}\left(2\left\|(\Delta
U^{2},\Delta
V^{2},\Delta\bar{V}^{2},\Delta\tilde{V}^{2})\right\|_{{\cal
M}^{D}_{\gamma}}\right.\nonumber\\
&&\;\;\;\;\;\;\;\;\;\;\;\;\;\;\;\;\;\;\;\left.+\left\|(\Delta
U^{1},\Delta
V^{1},\Delta\bar{V}^{1},\Delta\tilde{V}^{1})\right\|_{{\cal
M}^{D}_{\gamma}}\right) \nonumber\\
&<&\infty.\nonumber
\end{eqnarray}
Thus, from \eq{seriescon}, we see that
$(U^{i},V^{i},\bar{V}^{i},\tilde{V}^{i})$ with $i\in\{1,2,...\}$
forms a Cauchy sequence in ${\cal M}^{D}_{\gamma}[0,T]$, which
implies that there is some $(U,V,\bar{V},\tilde{V})$ such that
\begin{eqnarray}
&&(U^{i},V^{i},\bar{V}^{i},\tilde{V}^{i})\rightarrow
(U,V,\bar{V},\tilde{V})\;\;\mbox{as}\;\;i\rightarrow\infty\;\;
\mbox{in}\;\;{\cal M}^{D}_{\gamma}[0,T]. \elabel{finalcon}
\end{eqnarray}
Finally, by \eq{finalcon} and the similar procedure as used for
Theorem 5.2.1 in pages 68-71 of
$\emptyset$ksendal~\cite{oks:stodif}, we can complete the proof of
Lemma~\ref{lemmathree}. $\Box$
\end{proof}

\subsection{Proof of Theorem~\ref{bsdeyI}}

By combining Lemmas~\ref{martindecom}-\ref{lemmathree}, we can reach
a proof for Theorem~\ref{bsdeyI}. $\Box$

\subsection{Proof of Theorem~\ref{infdn}}

First, we consider a real-valued system corresponding to the case
that $\tau=T$, whose proof is along the line of the one for
Lemma~\ref{lemmathree}. More precisely, for any given number
sequence $\gamma=\{\gamma_{D_{c}},c=0,1,2,...\}$ with
$\gamma_{D_{c}}\in R$, replace the norm for the Banach space ${\cal
M}^{D}_{\gamma}[0,T]$ defined in \eq{combanach} by the one
\begin{eqnarray}
\;\;\;\left\|(U,V,\bar{V},\tilde{V})\right\|^{2}_{{\cal
M}^{D}_{\gamma}} &\equiv&
\sum_{c=0}^{\infty}\xi(c)\left\|(U,V,\bar{V},\tilde{V})
\right\|^{2}_{{\cal M}^{D_{c}}_{\gamma_{D_{c}},c}},
\elabel{fspaceex}
\end{eqnarray}
for any given $(U,V,\bar{V},\tilde{V})$ in this space, where
\begin{eqnarray}
\;\;\;\;\left\|(U,V,\bar{V},\tilde{V})\right\|^{2}_{{\cal
M}^{D_{c}}_{\gamma_{D_{c}}}}&=&E\left[\sup_{0\leq t\leq
T}\left\|U(t)\right\|^{2}_{C^{c}(D_{c},r)}e^{2\gamma_{D_{c}}t}\right]
\nonumber\\
&&+E\left[\sup_{0\leq t\leq
T}\left\|V(t)\right\|^{2}_{C^{c}(D_{c},q)}e^{2\gamma_{D_{c}}t}\right]
\nonumber\\
&&+E\left[\int_{0}^{T}
\left\|\bar{V}(t)\right\|^{2}_{C^{c}(D_{c},qd)}e^{2\gamma_{D_{c}}t}dt\right]
\nonumber\\
&&+E\left[\int_{0}^{T}\left\|\tilde{V}(t)
\right\|^{2}_{\nu,c}e^{2\gamma_{D_{c}}t}dt\right]. \nonumber
\end{eqnarray}
Then, it follows from the similar argument used for \eq{infinityine}
in the proof of Lemma~\ref{lemmathree} that
\begin{eqnarray}
&&(U^{1}(\cdot,x),V^{1}(\cdot,x),\bar{V}^{1}(\cdot,x),\tilde{V}^{1}(\cdot,x,z))
\in\bar{{\cal Q}}^{2}_{{\cal F}}([0,T]\times D) \nonumber
\end{eqnarray}
with $(U^{0},V^{0},\bar{V}^{0},\tilde{V}^{0})=(0,0,0,0)$, where
$(U^{1},V^{1},\bar{V}^{1},\tilde{V}^{1})$ is defined through
\eq{sigmanV} in Lemma~\ref{martindecom}. Furthermore, over each
$D_{c}$ with $c\in\{0,1,...\}$, we have that
\begin{eqnarray}
&&\left\|(\Delta U^{i},\Delta
V^{i},\Delta\bar{V}^{i},\Delta\tilde{V}^{i}) \right\|^{2}_{{\cal
M}^{D}_{\gamma}}
\elabel{uinfinityine}\\
&\leq&e^{k}\delta\left(\left\|(\Delta U^{i-1},\Delta
V^{i-1},\Delta\bar{V}^{i-1},\Delta\tilde{V}^{i-1})
\right\|^{2}_{{\cal M}^{D}_{\gamma}}\right.
\nonumber\\
&&+\left.\left\|(\Delta U^{i-2},\Delta
V^{i-2},\Delta\bar{V}^{i-2},\Delta\tilde{V}^{i-2})
\right\|^{2}_{{\cal M}^{D}_{\gamma}}\right), \nonumber
\end{eqnarray}
where $\delta$ is a constant that can be determined by suitably
choosing a number sequence $\gamma$ such that
$\gamma_{D_{0}}<\gamma_{D_{1}}<...$ and
$0<\sqrt{e^{k}\delta}/(1-2\sqrt{e^{k}\delta})<1$ (note that
$\gamma_{D_{c}}$ may depend on both $D_{c}$ and $c$ for each
$c\in\{0,1,...\}$). Thus, it follows from \eq{uinfinityine} that the
remaining justification for Theorem~\ref{infdn} can be conducted
along the line of proof for Theorem~\ref{bsdeyI}.

Second, we consider a real-valued system corresponding to the case
that $\tau$ is a general stopping time. The proof for this case can
be accomplished by extending the proof corresponding to $\tau=T$ via
the techniques developed in Dai~\cite{dai:meavar,dai:meahed} for
both forward and backward SDEs, and the related discussions in Yong
and Zhou~\cite{yonzho:stocon}.

Third, by direct generalizing the discussion concerning the
real-valued system to complex-valued system, we reach a proof for
Theorem~\ref{infdn}. $\Box$

\section{Proofs of Theorem~\ref{fbexist} and
Theorem~\ref{gtheoremo}}\label{gfbproof}

To provide the proofs for Theorem~\ref{fbexist} and
Theorem~\ref{gtheorem}, we first recall the Skorohod problem and
study some related properties.

\subsection{The Skorohod Problem}

Let $D([0,T],R^{b})$ with $b\in\{p,2p\}$ be the space of all
functions $z:[0,T]\rightarrow R^{b}$ that are right-continuous with
left limits and are endowed with Skorohod topology (see, e.g.,
Billingsley~\cite{bil:conpro}, Jacod and
Shiryaev~\cite{jacshi:limthe}). Then, we can introduce the Skorohod
problem as follows.
\begin{definite} (The Skorohod problem).
Given $z\in D([0,T],R^{p})$ with $z(0) \in D$, a $(D,R)$-regulation
of z over [0,T] is a pair $(x,y)\in D([0,T],D)\times
D([0,T],R_{+}^{b})$ such that
\begin{eqnarray}
x(t)=z(t)+Ry(t)\;\;\mbox{for all}\;\;t\in [0,T],\nonumber
\end{eqnarray}
where,  for each $i\in\{1,...,b\}$,
\begin{enumerate}
\item $y_{i}(0)=0$,
\item $y_{i}$ is nondecreasing,
\item $y_{i}$ can increase only at a time $t\in [0,T]$ with
$x(t)\in F_{i}$.
\end{enumerate}
\end{definite}
Furthermore, we define the modulus of continuity with respect to a
function $z(\cdot)\in D([0,T],R^{b})$ and a real number $\delta>0$
by
\begin{eqnarray}
&&w(z,\delta,T)\equiv\inf_{t_{l}}\max_{l}
\mbox{Osc}\left(z,[t_{l-1},t_{l})\right), \elabel{cmodulus}
\end{eqnarray}
where the infimum takes over all the finite sets $\{t_{l}\}$ of
points satisfying $0=t_{0}<t_{1}<...<t_{m}=T$ and
$t_{l}-t_{l-1}>\delta$ for $l=1,...,m$, and
\begin{eqnarray}
&&\mbox{Osc}(z,[t_{l-1},t_{l}])=\sup_{t_{1}\leq s\leq t\leq
t_{2}}\|z(t)-z(s)\|. \elabel{osclationd}
\end{eqnarray}
Then, we have the following lemma.
\begin{lemma}\label{osclemma}
Suppose that the reflection matrix $R$ in Definition satisfies the
completely-${\cal S}$ condition. Then, any $(D,R)$-regulation
$(x,y)$ of $z\in D([0,T],R^{p})$ with $z(0)\in D$ satisfies the
oscillation inequality over $[t_{1},t_{2}]$ with
$t_{1},t_{2}\in[0,T]$
\begin{eqnarray}
\mbox{Osc}(x,[t_{1},t_{2}])&\leq&\kappa\mbox{Osc}(z,[t_{1},t_{2}]),
\elabel{oscinq}\\
\mbox{Osc}(y,[t_{1},t_{2}])&\leq&\kappa\mbox{Osc}(z,[t_{1},t_{2}]),
\elabel{oscinqI}
\end{eqnarray}
where $\kappa$ is some nonnegative constant depending only on the
inward normal vector $N$ and the reflection matrix $R$.
\end{lemma}
\begin{proof}
For each $t\in[t_{1},t_{2}]$, define
\begin{eqnarray}
\Delta z(t)&\equiv&z(t)-z(t^{-}),\elabel{deltaI}\\
\Delta x(t)&\equiv&x(t)-x(t^{-}),\elabel{deltaII}\\
\Delta y(t)&\equiv&y(t)-y(t^{-}).\elabel{deltaIII}
\end{eqnarray}
Since the reflection matrix $R$ satisfies the completely-${\cal S}$
condition, it is easy to check that the linear complementarity
problem (LCP)
\begin{eqnarray}
\Delta x(t)&=&\Delta z(t)+R\Delta y(t),
\nonumber\\
\Delta x(t)&\in& D,\nonumber\\
\Delta y(t)&\geq& 0,\nonumber\\
\Delta x_{i}(t)\Delta y_{i}(t)&=&0\;\;\mbox{for}\;\;i=1,...,p,
\nonumber\\
(b_{i}-\Delta x_{i}(t))\Delta
y_{i}(t)&=&0\;\;\mbox{for}\;\;i=p+1,...,b, \nonumber
\end{eqnarray}
is completely solvable (see also Theorem 2.1 in
Mandelbaum~\cite{man:dyncom} for the related discussion).
Furthermore, we can conclude that
\begin{eqnarray}
\Delta y(t)\leq C\Delta z(t)\elabel{oydelta}
\end{eqnarray}
for some nonnegative constant $C$ depending only on the inward
normal vector $N$ and the reflection matrix $R$. Then, the rest of
the proof is the direct conclusion of the one for Theorem 3.1 in
Dai~\cite{dai:broapp} or the one for Theorem 4.2 in Dai and
Dai~\cite{daidai:heatra}. $\Box$
\end{proof}
\begin{lemma}\label{complimentlemma}
Assume that $(x^{n},y^{n})\rightarrow (x,y)$ along $n\in\{1,2,...\}$
in $D([0,T],R^{p})\times D([0,T],R^{b})$ and $y^{n}(\cdot)$ is of
bounded variation for each $n\in\{1,2,...\}$. Furthermore, suppose
that
\begin{eqnarray}
&&\int_{0}^{t}f(x^{n}(s))dy^{n}(s)=0 \elabel{intconv}
\end{eqnarray}
for all $n\in\{1,2,...\}$ and each $t\in[0,T]$, where $f\in
C^{b}([0,T],R^{b})$ is a $b$-dimensional bounded vector function.
Then, for each $t\in[0,T]$, we have that
\begin{eqnarray}
&&\int_{0}^{t}f(x(s))dy(s)=0. \elabel{intconvI}
\end{eqnarray}
\end{lemma}
\begin{proof}
It follows from the definition in pages 123-124 of
Billingsley~\cite{bil:conpro} or Theorem 1.14 in page 328 of Jacod
and Shiryaev~\cite{jacshi:limthe} that there is a sequence
$\{\gamma_{n},n\in\{1,2,...\}\}$ of continuous and strictly
increasing functions mapping from $[0,T]\rightarrow[0,T]$ with
$\gamma_{n}(0)=0$ and $\gamma_{n}(T)=T$ such that
\begin{eqnarray}
&&\sup_{t\in[0,T]}\left|\gamma_{n}(t)-t\right|\rightarrow 0,
\elabel{changecm}\\
&&\sup_{t\in[0,T]}\left|(x^{n},y^{n})(\gamma_{n}(t))-(x,y)(t)\right|
\rightarrow0. \elabel{changecmI}
\end{eqnarray}
Then, by the uniform convergence in \eq{changecm}-\eq{changecmI} and
the condition in \eq{intconv}, we know that
\begin{eqnarray}
\int_{0}^{t}f(x(s))dy(s)
&=&\lim_{n\rightarrow\infty}\int_{0}^{t}f(x^{n}(\gamma_{n}(s)))dy^{n}(\gamma_{n}(s))
\nonumber\\
&=&\lim_{n\rightarrow\infty}\int_{0}^{\gamma^{-1}_{n}(t)}f(x^{n}(u))dy^{n}(u)
\nonumber\\
&=&0, \nonumber
\end{eqnarray}
where $\gamma_{n}^{-1}(\cdot)$ is the inverse function of
$\gamma_{n}(\cdot)$ for each $n\in\{1,2,...\}$. Hence, we complete
the proof of Lemma~\ref{complimentlemma}. $\Box$
\end{proof}

\subsection{Proof of Theorem~\ref{fbexist}}\label{fbtproof}

We divide the proof of the theorem into four parts: Part A
(Existence, Uniqueness), Part B, Part C, and Part D, which
correspond to different boundary reflection conditions.

\vskip 0.3cm
{\bf Part A (Existence).} We consider the case that
$L(t,\omega)$ appeared in \eq{fbconI}-\eq{fbconII} is a constant and
both of the forward and the backward SDEs have reflection
boundaries. In this case, we need to prove the claim that there is
an adapted weak solution $((X,Y),(V,\bar{V},\tilde{V},F))$ to the
system in \eq{bsdehjb}.

In fact, for a positive integer $b$, let $D^{2}_{{\cal
F}}([0,T],R^{b})$ be the space of $R^{b}$-valued and $\{{\cal
F}_{t}\}$-adapted processes with sample paths in $D([0,T],R^{b})$.
Furthermore, each $Y\in D^{2}_{{\cal F}}([0,T],R^{b})$ is
square-integrable in the sense that
\begin{eqnarray}
&&E\left[\int_{0}^{T}\|Y(t)\|^{2}dt\right]<\infty.
\elabel{adaptednormIy}
\end{eqnarray}
In addition, we use $D^{2}_{{\cal F},p}([0,T],R^{b})$ to denote the
corresponding predictable space. Then, for a given $n\in\{1,2,...\}$
and a 4-tuple
\begin{eqnarray}
(X^{n},V^{n},\bar{V}^{n},\tilde{V}^{n})&\in& D^{2}_{{\cal F}}([0,T],
R^{p})\times D^{2}_{{\cal F}}([0,T], R^{q})\times D^{2}_{{\cal
F},p}([0,T], R^{q\times d}) \elabel{initialp}\\
&&\times D^{2}_{{\cal F},p}([0,T]\times R^{h}_{+}, R^{q\times h})
\nonumber
\end{eqnarray}
with $X^{n}(0)\in D$ and $V^{n}(T)\in\bar{D}$, we have the following
observation.

By the study concerning the continuous dynamic complementarity
problem (DCP) in Bernard and El Kharroubi~\cite{berkha:regdet} (see
also the related discussions in Mandelbaum~\cite{man:dyncom}, Reiman
and Williams~\cite{reiwil:boupro}), Theorem~\ref{bsdeyI} (and its
proof) in the current paper, there is a 6-tuple
\begin{eqnarray}
&&((X^{n+1},Y^{n+1}),(V^{n+1},\bar{V}^{n+1},\tilde{V}^{n+1},F^{n+1}))
\nonumber\\
&&\in D^{2}_{{\cal F}}([0,T], R^{p})\times D^{2}_{{\cal F}}([0,T],
R^{b})\nonumber\\
&&\;\;\;\;\times D^{2}_{{\cal F}}([0,T], R^{q})\times D^{2}_{{\cal
F},p}([0,T], R^{q\times d}) \nonumber\\
&&\;\;\;\;\times D^{2}_{{\cal F},p}([0,T]\times{\cal Z}^{h},
R^{q\times h})\times D^{2}_{{\cal F}}([0,T], R^{q\times\bar{b}})
\nonumber
\end{eqnarray}
for each $n\in\{1,2,...\}$, satisfying the properties along each
sample path:
\begin{eqnarray}
X^{n+1}(t)&=&X(0)+Z^{n}(t)+RY^{n+1}(t)\in D, \elabel{bsdehjbIp}
\end{eqnarray}
with
\begin{eqnarray}
Z^{n}(t)&=&Z_{1}^{n}(t)+Z_{2}^{n}(t), \nonumber\\
Z_{1}^{n}(t)&=&\int_{0}^{t}b(s^{-},X^{n}(s^{-}),V^{n}(s^{-}),
\bar{V}^{n}(s^{-}),\tilde{V}^{n}(s^{-},\cdot),u(s^{-},X^{n}(s^{-}),\cdot)ds
\nonumber\\
Z_{2}^{n}(t)&=&\int_{0}^{t}\sigma(s^{-},X^{n}(s^{-}),V^{n}(s^{-}),
\bar{V}^{n}(s^{-}),\tilde{V}^{n}(s^{-},\cdot),u(t,X^{n}(s^{-})),z,\cdot)dW(s)
\nonumber\\
&& +\int_{0}^{t}\int_{{\cal
Z}^{h}}\eta(s^{-},X^{n}(s^{-}),V^{n}(s^{-}),
\bar{V}^{n}(s^{-}),\tilde{V}^{n}(s^{-},\cdot),u(s^{-},X^{n}(s^{-})),z,\cdot)
\tilde{N}(ds,dz); \nonumber
\end{eqnarray}
and
\begin{eqnarray}
V^{n+1}(t)&=&H(X^{n}(T))-SF^{n}(T)+U^{n}(t)+SF^{n+1}(t)\in\bar{D},
\elabel{vexpress}
\end{eqnarray}
with
\begin{eqnarray}
U^{n}(t)&=&U^{n}_{1}(t)-U^{n}_{2}(t)-U^{n}_{3}(t),\nonumber
\end{eqnarray}
where,
\begin{eqnarray}
U^{n}_{1}(t)&=&\int_{t}^{T}c(s^{-},X^{n}(s^{-}),V^{n}(s^{-}),
\bar{V}^{n}(s^{-}),\tilde{V}^{n}(s^{-},\cdot),u(s^{-},X^{n}(s^{-}),\cdot)ds,
\nonumber\\
U^{n}_{2}(t)&=&\int_{t}^{T}\left(\alpha(s^{-},X^{n}(s^{-}),V^{n}(s^{-}),
\bar{V}^{n}(s^{-}),\tilde{V}^{n}(s^{-},\cdot),\right.
\nonumber\\
&&\;\;\;\;\;\;\;\;\;\;\;\;\;\;\;\;\;\;\;
\left.u(s^{-},X^{n}(s^{-})),\cdot)-\bar{V}^{n}(s^{-})\right)dW(s)
\nonumber\\
&&+\int_{t}^{T}\int_{{\cal
Z}^{h}}\left(\zeta(s^{-},X^{n}(s^{-}),V^{n}(s^{-}),\bar{V}^{n}(s^{-}),
\tilde{V}^{n}(s^{-},z),\right.
\nonumber\\
&&\;\;\;\;\;\;\;\;\;\;\;\;\;\;\;\;\;\;\;
\left.u(s^{-},X^{n}(s^{-})),z,\cdot)-\tilde{V}^{n}(s^{-},z)\right)\tilde{N}(ds,dz),
\nonumber\\
U^{n}_{3}(t)&=&\int_{t}^{T}\bar{V}^{n+1}(s^{-})dW(s)
+\int_{t}^{T}\int_{{\cal
Z}^{h}}\tilde{V}^{n+1}(s^{-},z)\tilde{N}(ds,dz). \nonumber
\end{eqnarray}
Furthermore, $(X^{n+1},Y^{n+1})$ satisfies the property (3) in
Definition~\ref{srbm}. In other words, $Y^{n+1}$ is a
$b$-dimensional $\{{\cal F}_{t}\}$-adapted process such that the
$i$th component $Y^{n+1}_{i}$ of $Y^{n+1}$ for each
$i\in\{1,...,b\}$ ${\bf P}$-a.s. has the properties that
$Y^{n+1}_{i}(0)=0$, $Y^{n+1}_{i}$ is non-decreasing, and
$Y^{n+1}_{i}$ can increase only when $X^{n+1}$ is on the boundary
face $D_{i}$, i.e.,
\begin{eqnarray}
\int_{0}^{t}I_{D_{i}}(X^{n+1}(s))dY^{n+1}_{i}(s) =
Y^{n+1}_{i}(t)\;\;\mbox{for all}\;\;t\geq 0. \elabel{nxyzero}
\end{eqnarray}
Similarly, $(V^{n+1},F^{n+1})$ also satisfies the property (3) in
Definition~\ref{srbm}. More precisely, $F^{n+1}$ is a
$q$-dimensional $\{{\cal F}_{t}\}$-adapted process such that the
$i$th component $F^{n+1}_{i}$ of $F^{n+1}$ for each
$i\in\{1,...,\bar{b}\}$ ${\bf P}$-a.s. has the properties that
$F^{n+1}_{i}(0)=0$, $F^{n+1}_{i}$ is non-decreasing, and
$F^{n+1}_{i}$ can increase only when $V^{n+1}$ is on the boundary
face $\bar{D}_{i}$, i.e.,
\begin{eqnarray}
\int_{0}^{t}I_{\bar{D}_{i}}(V^{n+1}(s))dF^{n+1}_{i}(s) =
F^{n+1}_{i}(t)\;\;\mbox{for all}\;\;t\geq 0. \elabel{nxyzeroI}
\end{eqnarray}

Next, we prove that the following sequence of stochastic processes
along $n\in\{1,2,...\}$,
\begin{eqnarray}
&&\Xi^{n}=((X^{n+1},Y^{n+1}),(V^{n+1},\bar{V}^{n+1},
\tilde{V}^{n+1},F^{n+1})),\;\;
(X^{1},V^{1},\bar{V}^{1},\tilde{V}^{1})=0, \elabel{sevensp}
\end{eqnarray}
is relatively compact in the Skorohod topology over the space
\begin{eqnarray}
{\cal P}[0,T]&\equiv& D^{2}_{{\cal F}}([0,T], R^{p})\times
D^{2}_{{\cal F}}([0,T], R^{b})
\elabel{sevenp}\\
&&\times D^{2}_{{\cal F}}([0,T], R^{q})\times D^{2}_{{\cal
F},p}([0,T], R^{q\times d}) \nonumber\\
&&\times D^{2}_{{\cal F},p}([0,T]\times{\cal Z}^{h}, R^{q\times
h})\times D^{2}_{{\cal F}}([0,T], R^{q\times\bar{b}}). \nonumber
\end{eqnarray}
Along the line of Dai~\cite{dai:broapp,dai:optrat}, Dai and
Dai~\cite{daidai:heatra}, and by Corollary 7.4 in page 129 of Ethier
and Kurtz~\cite{ethkur:marpro}, it suffices to prove the following
two conditions to be true: First, for each $\epsilon>0$ and rational
$t>0$, there is a constant $C(\epsilon,t)$ such that
\begin{eqnarray}
\liminf_{n\rightarrow\infty}P\left\{\left\|\Xi^{n}\right\|^{2}\leq
C(\epsilon,t)\right\}\geq 1-\epsilon; \elabel{conda}
\end{eqnarray}
Second, for each $\epsilon>0$ and $T>0$, there is a constant
$\delta>0$ such that
\begin{eqnarray}
\limsup_{n\rightarrow\infty}P\left\{w(\Xi^{n},\delta,T)\geq\epsilon\right\}
\leq\epsilon. \elabel{condb}
\end{eqnarray}

To prove the two conditions stated in \eq{conda} and \eq{condb}, we
first define the norm along each sample path
\begin{eqnarray}
\left\|f\right\|_{[a,b]}&=&\sup_{a\leq t\leq b}\left\|f(t)\right\|
\nonumber
\end{eqnarray}
for each $f\in\{X^{n}, Z^{n}, U^{n},
(V^{n},\bar{V}^{n},\tilde{V}^{n})\}$ and each $a,b\in[0,T]$. Then,
we introduce the space for some constant $\gamma>0$ that will be
chosen and explained in the following proof,
\begin{eqnarray}
{\cal Q}_{\gamma}[0,T]&\equiv& D^{2}_{{\cal F}}([0,T], R^{p})\times
D^{2}_{{\cal F}}([0,T], R^{q})\times D^{2}_{{\cal F},p}([0,T],
R^{q\times d})
\elabel{mspace}\\
&&\times D^{2}_{{\cal F},p}([0,T]\times{\cal Z}^{h}, R^{q\times h})
\nonumber
\end{eqnarray}
endowed with the norm
\begin{eqnarray}
&&\left\|(X,V,\bar{V},\tilde{V})\right\|^{2}_{{\cal
Q}_{\gamma}[0,T]}
\elabel{fbnorm}\\
&\equiv&E\left[\sup_{t\in[0,T]}\left(\|X(t)\|^{2}+\|V(t)\|^{2}\right)e^{2\gamma
t}\right]+E\left[\int_{0}^{T}\left\|\bar{V}(t)\right\|^{2}e^{2\gamma
t}dt\right]
\nonumber\\
&&+E\left[\int_{0}^{T}\left\|\tilde{V}(t,\cdot)\right\|^{2}_{\nu}
e^{2\gamma t}dt\right] \nonumber
\end{eqnarray}
for each $(X,V,\bar{V},\tilde{V})\in{\cal Q}_{\gamma}[0,T]$. Thus,
by Lemma~\ref{osclemma}, there is a positive constant $C_{1}$ such
that
\begin{eqnarray}
&&\left\|(X^{n+1},Y^{n+1})(t)\right\|
\elabel{ineqI}\\
&\leq&\left\|(X^{n+1},Y^{n+1})(0)\right\|+\kappa\mbox{Osc}(Z^{n},[0,T])
\nonumber\\
&\leq&
C_{1}\left(\left\|X(0)\right\|+\left\|Z^{n}\right\|_{[0,T]}\right),
\nonumber
\end{eqnarray}
and
\begin{eqnarray}
&&\left\|(V^{n+1},\bar{V}^{n+1},\tilde{V}^{n+1}(\cdot),F^{n+1})(t)\right\|
\elabel{ineqII}\\
&\leq&\left\|(V^{n+1},\bar{V}^{n+1},\tilde{V}^{n+1}(\cdot))(t)\right\|+\left\|F^{n+1}(t)\right\|
\nonumber\\
&\leq&\left\|(V^{n+1},\bar{V}^{n+1},\tilde{V}^{n+1}(\cdot))(T)\right\|
+\left\|F^{n+1}(0)\right\| +2\kappa\mbox{Osc}(U^{n},[0,T])
\nonumber\\
&\leq&\bar{C}_{1}\left(\left\|(V^{n}(T)\right\|
+\left\|F^{n}(T)\right\| +\left\|U^{n}\right\|_{[0,T]}\right)
\nonumber\\
&\leq& \bar{C}_{2}\left(1+\left\|X^{n}(T)\right\|
+\left\|U^{n-1}\right\|_{[0,T]}
+\left\|U^{n}\right\|_{[0,T]}\right), \nonumber\\
&\leq& C_{1}\left(1+\|X(0)\|+\left\|Z^{n-1}\right\|_{[0,T]}
+\left\|U^{n-1}\right\|_{[0,T]}
+\left\|U^{n}\right\|_{[0,T]}\right), \nonumber
\end{eqnarray}
where, $\bar{C}_{1}$ and $\bar{C}_{2}$ are some nonnegative
constants. Furthermore, we have taken $(\bar{V}^{n+1}$,
$\tilde{V}^{n+1}(\cdot))(T)=0$ in the third equality of \eq{ineqII}
since the uniqueness for the Martingale representation is in the
sense of up to sets of measure zero in $(t,\omega)$ (see, e.g.,
Theorem 4.3.4 in page 53 of $\emptyset$ksendal~\cite{oks:stodif} and
Theorem 5.3.5 in page 266 of Applebaum~\cite{app:levpro}).

Thus, for each $n\in\{1,2,...\}$, the given linear growth constant
$L\geq 0$ in \eq{fbconI}, and any constant $K>LT$, it follows from
the Markov's inequality that
\begin{eqnarray}
P\left\{\left\|Z^{n}_{1}\right\|_{T}\geq K\right\}
&\leq&\frac{LT}{K-LT}E\left[\left\|(X^{n},V^{n},\bar{V}^{n},
\tilde{V}^{n}(\cdot))\right\|_{[0,T]}\right]. \elabel{ineqI}
\end{eqnarray}
Furthermore, by Lemma 4.2.8 in page 201 of
Applebaum~\cite{app:levpro} (or related theorem in page 20 of Gihman
and Skorohod~\cite{gihsko:stodif}) and the linear growth condition,
we know that
\begin{eqnarray}
P\left\{\left\|Z^{n}_{2}\right\|_{T}\geq K\right\}
&\leq&\frac{\bar{K}}{K^{2}}+\frac{L^{2}T}{\bar{K}-L^{2}T}
E\left[\left\|(X^{n},V^{n},\bar{V}^{n},\tilde{V}^{n}(\cdot))\right\|^{2}_{[0,T]}\right]
\elabel{ineqII}
\end{eqnarray}
for all nonnegative constant $\bar{K}>L^{2}T$. In addition, similar
to the illustration of \eq{ineqI}, we have that
\begin{eqnarray}
P\left\{\left\|U^{n}_{1}\right\|_{T}\geq K\right\}
&\leq&\frac{1}{K-LT}E\left[\left\|(X^{n},V^{n},\bar{V}^{n},
\tilde{V}^{n}(\cdot))\right\|_{[0,T]}\right]. \elabel{ineqIU}
\end{eqnarray}
Next, by the similar demonstration for \eq{ineqII} and the linear
growth condition, we know that
\begin{eqnarray}
P\left\{\left\|U^{n}_{2}\right\|_{T}\geq K\right\}
&\leq&\frac{\bar{K}}{K^{2}}+\frac{L^{2}T}{\bar{K}-L^{2}T}
E\left[\left\|(X^{n},V^{n},\bar{V}^{n},\tilde{V}^{n}(\cdot))\right\|^{2}_{[0,T]}\right].
\elabel{ineqIIUa}
\end{eqnarray}
Furthermore, by the proof of Theorem~\ref{bsdeyI}, we have that
\begin{eqnarray}
P\left\{\left\|U^{n}_{3}\right\|_{T}\geq K\right\}
&\leq&\frac{\bar{K}}{K^{2}}+\frac{\bar{K}_{1}T}{(\bar{K}-L^{2}T)^{2}}
\elabel{ineqIIU}\\
&&+\frac{\bar{K}_{2}}{(\bar{K}-L^{2}T)^{2}}
E\left[\left\|(X^{n},V^{n},\bar{V}^{n},\tilde{V}^{n})\right\|^{2}_{{\cal
Q}_{\gamma}[0,T]}\right]
\nonumber
\end{eqnarray}
for some nonnegative constants $\bar{K}_{1}$ and $\bar{K}_{2}$.
Therefore, for each given $\epsilon>0$, it follows from
\eq{ineqI}-\eq{ineqIIU}, suitably chosen constants $K$ and
$\bar{K}$, and the initial condition in \eq{sevensp} that there is a
nonnegative constant $C$ such that
\begin{eqnarray}
&&\inf_{n}P\left\{\left\|\Xi^{n}(t)\right\|\leq
C,\;0\leq t\leq T\right\} \elabel{ineqIIo}\\
&\geq&\inf_{n}\min\left\{P\left\{\left\|(X^{n+1},Y^{n+1})(t)\right\|\leq
C,\;0\leq t\leq
T\right\}\right.,\nonumber\\
&&\;\;\;\;\;\;\;\;\;\;\;\;\;
\left.P\left\{\left\|(V^{n+1},\bar{V}^{n+1},\tilde{V}^{n+1},F^{n+1})(t)\right\|\leq
C,\;0\leq t\leq T\right\}\right\}\nonumber\\
&\geq&1-\epsilon. \nonumber
\end{eqnarray}
Thus, the condition in \eq{conda} is satisfied by the sequence of
$\{\Xi^{n}\}$.

Now, for any $t\in[0,T]$, it follows from the proof of Proposition
18 for a BSDE with jumps in Dai~\cite{dai:meavar} and
Lemma~\ref{osclemma} that
\begin{eqnarray}
&&\left\|(U^{n},\bar{V}^{n}, \tilde{V}^{n})\right\|^{2}_{{\cal
Q}_{\gamma}[t,T]}
\elabel{IineqIqo}\\
&\leq&K_{\gamma}\left(2L^{2}(T-t)+\left\|(X^{n-1},V^{n-1},\bar{V}^{n-1},
\tilde{V}^{n-1})\right\|^{2}_{{\cal Q}_{\gamma}[t,T]}\right)
\nonumber\\
&\leq&K_{\gamma}\left(2L^{2}(T-t)+e^{2\gamma
T}E\left[\left\|V^{n-1}\right\|^{2}_{[t,T]}\right]+e^{2\gamma T}
\int_{t}^{T}E\left[\left\|X^{n-1}\right\|^{2}_{[0,s]}\right]ds\right)
\nonumber\\
&&+K_{\gamma}\left\|(U^{n-1},\bar{V}^{n-1},\tilde{V}^{n-1})\right\|^{2}_{{\cal
Q}_{\gamma}[t,T]},
\nonumber
\end{eqnarray}
where, $K_{\gamma}<1$ depending only on $L$, $T$, $d$, and $h$ for
some suitable chosen $\gamma>0$. Thus, by Lemma~\ref{osclemma}, the
It$\hat{o}$'s isometry formula, and \eq{IineqIqo}, we have that
\begin{eqnarray}
&&E\left[\left\|V^{n}\right\|^{2}_{[t,T]}\right]
\elabel{IineqIq}\\
&\leq&\bar{K}_{1}\left(E\left[\|V^{n}(T)\|^{2}\right]+E\left[\|F^{n-1}(T)\|^{2}\right]+\kappa^{2}
E\left[\mbox{Osc}(U^{n-1},[t,T])^{2}\right]\right)
\nonumber\\
&\leq&K_{1}\left(1+E\left[\|X^{n}\|^{2}_{[0,T]}\right]+\kappa^{2}
E\left[\mbox{Osc}(U^{n-2},[0,T])^{2}\right]+\kappa^{2}
E\left[\mbox{Osc}(U^{n-1},[t,T])^{2}\right]\right)
\nonumber\\
&\leq&K_{1}\left(1+24\kappa^{2}L^{2}T^{2}
+24\kappa^{2}L^{2}(T-t)^{2}\right)
+K_{1}E\left[\|X^{n}\|^{2}_{[0,T]}\right]
\nonumber\\
&&+24K_{1}\kappa^{2}L^{2}T\left(\int_{0}^{T}E\left[\left\|X^{n-2}\right\|^{2}_{[0,s]}\right]ds
+E\left[\left\|V^{n-2}\right\|^{2}_{[0,T]}\right]\right)
\nonumber\\
&&+24K_{1}\kappa^{2}L^{2}(T-t)\left(\int_{t}^{T}E\left[\left\|X^{n-1}\right\|^{2}_{[0,s]}\right]ds
+E\left[\left\|V^{n-1}\right\|^{2}_{[t,T]}\right]\right)
\nonumber\\
&&+24K_{1}\kappa^{2}L^{2}T\left\|(U^{n-2},\bar{V}^{n-2},
\tilde{V}^{n-2})\right\|^{2}_{{\cal Q}_{\gamma}[0,T]}
\nonumber\\
&&+4K_{1}\kappa^{2}\left\|(U^{n-1},\bar{V}^{n-1},
\tilde{V}^{n-1})\right\|^{2}_{{\cal Q}_{\gamma}[0,T]}
\nonumber\\
&&+24K_{1}\kappa^{2}L^{2}(T-t)\left\|(U^{n-1},\bar{V}^{n-1},
\tilde{V}^{n-1})\right\|^{2}_{{\cal Q}_{\gamma}[t,T]}
\nonumber\\
&&+4K_{1}\kappa^{2}\left\|(U^{n},\bar{V}^{n},
\tilde{V}^{n})\right\|^{2}_{{\cal Q}_{\gamma}[t,T]}
\nonumber\\
&\leq&K_{3}+K_{2}\left(\int_{0}^{T}E\left[\left\|X^{n-1}\right\|^{2}_{[0,s]}
\right]ds+\left\|(U^{n-2},\bar{V}^{n-2},
\tilde{V}^{n-2})\right\|^{2}_{{\cal Q}_{\gamma}[0,T]}\right.
\nonumber\\
&&\left.+\left\|(U^{n-1},\bar{V}^{n-1},
\tilde{V}^{n-1})\right\|^{2}_{{\cal Q}_{\gamma}[0,T]}
+\left\|(U^{n},\bar{V}^{n}, \tilde{V}^{n})\right\|^{2}_{{\cal
Q}_{\gamma}[0,T]}\right), \nonumber
\end{eqnarray}
where, $K_{i}$ for $i\in\{1,2,3\}$ are some nonnegative constants
depending only on $T$, $L$, $\kappa$, and
$E\left[\|V(T)\|^{2}\right]$. Furthermore, for any $t\in[0,T]$, we
have that
\begin{eqnarray}
E\left[\left\|X^{n}\right\|^{2}_{[0,t]}\right]
&\leq&2E\left[\left\|X(0)\right\|^{2}\right]+2\kappa^{2}
E\left[\mbox{Osc}(Z^{n-1},[0,t])^{2}\right]
\elabel{IineqIw}\\
&\leq&2E\left[\left\|X(0)\right\|^{2}\right] +6\kappa^{2}L^{2}t^{2}
\nonumber\\
&&+6\kappa^{2}L^{2}t\left(\int_{0}^{t}
E\left[\left\|X^{n-1}\right\|_{[0,s]}^{2}\right]ds
+E\left[\left\|V^{n-1}\right\|^{2}_{[0,T]}\right]\right)
\nonumber\\
&&+6\kappa^{2}L^{2}t\left\|(U^{n-1},\bar{V}^{n-1},
\tilde{V}^{n-1})\right\|^{2}_{{\cal Q}_{\gamma}[0,T]}
\nonumber\\
&\leq&2E\left[\left\|X(0)\right\|^{2}\right] +12\kappa^{4}L^{2}t^{2}
+6\kappa^{2}L^{2}tE\left[\left\|V^{2}(T)\right\|\right] \nonumber\\
&&+6\kappa^{2}L^{2}t\int_{0}^{t}
E\left[\left\|X^{n-1}\right\|_{[0,s]}^{2}\right]ds
\nonumber\\
&&+6\kappa^{2}L^{2}t\left(1+2\kappa^{2}\right)\left\|(U^{n-1},\bar{V}^{n-1},
\tilde{V}^{n-1})\right\|^{2}_{{\cal Q}_{\gamma}[0,T]}. \nonumber
\end{eqnarray}

Therefore, for any $\epsilon>0$ and a constant $\delta>0$, consider
a finite set $\{t_{l}\}$ of points satisfying
$0=t_{0}<t_{1}<...<t_{m}=T$ and $t_{l}-t_{l-1}=\delta<\epsilon/L$
with $l\in\{1,...m\}$. It follows from \eq{sevensp},
\eq{IineqIqo}-\eq{IineqIw}, and the similar explanation for
\eq{ineqI} that
\begin{eqnarray}
&&P\left\{w(Z^{n}_{1},\delta,T)\geq\epsilon\right\}
\elabel{ineqIw}\\
&\leq&\frac{3L^{2}\delta}{(\epsilon-L\delta)^{2}}
\left(E\left[\left\|X^{n}\right\|^{2}_{[0,T]}
+\left\|V^{n}\right\|^{2}_{[0,T]}\right]+\left\|(U^{n},\bar{V}^{n},
\tilde{V}^{n})\right\|^{2}_{{\cal Q}_{\gamma}[0,T]}\right)
\nonumber\\
&\leq&\frac{3L^{2}\delta}{(\epsilon-L\delta)^{2}}
\left(A_{0}+\sum_{k=1}^{n}\frac{A^{k+1}_{1}T^{k+1}}{(k+1)!}
\left(1+K^{k}_{\gamma}\right)+A_{2}\sum_{k=1}^{n}K^{k}_{\gamma}\right),
\nonumber
\end{eqnarray}
where $A_{0}$, $A_{1}$, and $A_{2}$ are some constants depending
only on $L$, $T$, $d$, and $h$. Furthermore, by Lemma 4.2.8 in page
201 of Applebaum~\cite{app:levpro} (or related theorem in page 20 of
Gihman and Skorohod~\cite{gihsko:stodif}) and the linear growth
condition, we know that
\begin{eqnarray}
&&P\left\{w(Z^{n}_{2},\delta,T)\geq\epsilon\right\}
\elabel{ineqIIw}\\
&\leq&\frac{\bar{\epsilon}}{\epsilon^{2}}+\frac{3L^{2}}{\bar{\epsilon}-3L^{2}\delta}
\left(\delta E\left[\left\|X^{n}\right\|_{T}^{2}\right] +\delta
E\left[\left\|V^{n}\right\|_{T}^{2}\right]
+E\left[\left\|(U^{n},\bar{V}^{n},\tilde{V}^{n})\right\|^{2}_{{\cal
Q}_{\gamma}[0,T]}\right]\right) \nonumber\\
&\leq&\frac{\bar{\epsilon}}{\epsilon^{2}}+\frac{3L^{2}}{\bar{\epsilon}-3L^{2}\delta}
\left(\delta\left(A_{0}+\sum_{k=1}^{n}\frac{A^{k+1}_{1}T^{k+1}}{(k+1)!}
\left(1+K^{k}_{\gamma}\right)+A_{2}\sum_{k=1}^{n}K^{k}_{\gamma}\right)
+A_{3}\sum_{k=1}^{n}K^{k}_{\gamma}\right) \nonumber
\end{eqnarray}
for all nonnegative constant $\bar{\epsilon}>3L^{2}\delta$, where
$A_{3}$ is some constant depending only on $L$, $T$, $d$, and $h$.

Similarly, there are some constants $B_{0}$, $B_{1}$, $B_{2}$, and
$B_{3}$ depending only on $L$, $T$, $d$, and $h$ such that
\begin{eqnarray}
&&P\left\{w(U^{n}_{1},\delta,T)\geq\epsilon\right\}
\elabel{ineqIws}\\
&\leq&\frac{3L^{2}\delta}{(\epsilon-L\delta)^{2}}
\left(B_{0}+\sum_{k=1}^{n}\frac{B^{k+1}_{1}T^{k+1}}{(k+1)!}
\left(1+K^{k}_{\gamma}\right)+B_{2}\sum_{k=1}^{n}K^{k}_{\gamma}\right),
\nonumber
\end{eqnarray}
and
\begin{eqnarray}
&&P\left\{w(Z^{n}_{2},\delta,T)\geq\epsilon\right\}
\elabel{ineqIIws}\\
&\leq&\frac{\bar{\epsilon}}{\epsilon^{2}}+\frac{3L^{2}}{\bar{\epsilon}-3L^{2}\delta}
\left(\delta\left(B_{0}+\sum_{k=1}^{n}\frac{B^{k+1}_{1}T^{k+1}}{(k+1)!}
\left(1+K^{k}_{\gamma}\right)+B_{2}\sum_{k=1}^{n}K^{k}_{\gamma}\right)
+B_{3}\sum_{k=1}^{n}K^{k}_{\gamma}\right). \nonumber
\end{eqnarray}
Hence, for each given $\epsilon>0$, it follows from
\eq{ineqIw}-\eq{ineqIIws} and suitably chosen constants
$\bar{\epsilon}$, $\delta$, and $\gamma$ that
\begin{eqnarray}
&&\limsup_{n\rightarrow\infty}
P\left\{w(\Xi^{n}),\delta,T)\geq\epsilon\right\}\leq \epsilon.
\elabel{ineqIII}
\end{eqnarray}
Thus, the condition in \eq{condb} is true for the sequence of
$\{\Xi^{n}\}$. Hence, by \eq{ineqII}, \eq{ineqIII}, and Corollary
7.4 in page 129 of Ethier and Kurtz~\cite{ethkur:marpro}, this
sequence is relatively compact. Therefore, there is a subsequence of
$\{\Xi^{n}\}$ that converges weakly to
$\Xi\equiv((X,Z,Y),(V,\bar{V},\tilde{V},F))$ over the space ${\cal
P}[0,T]$. For convenience, we suppose that the subsequence is the
sequence itself, i.e.,
\begin{eqnarray}
&&\Xi^{n}\Rightarrow\Xi. \elabel{conlimit}
\end{eqnarray}
Then, by the Skorohod representation theorem (see, e.g., Theorem 1.8
in page 102 of Ethier and Kurtz~\cite{ethkur:marpro}), we can assume
that the convergence in \eq{conlimit} is a.s. in the Skorohod
topology. Thus, by the claim (a) in Theorem 1.14 (or the claim (a)
in Proposition 2.1) of Jacod and Shiryaev~\cite{jacshi:limthe} and
the facts that $Y^{n+1}(0)=0$ and $Y^{n+1}$ is nondecreasing, we can
conclude that $Y(0)=0$ and $Y$ is nondecreasing. Furthermore, by
Lemma~\ref{complimentlemma} and \eq{nxyzero}
\begin{eqnarray}
&&\int_{0}^{t}I_{D_{i}}(X(s))dY_{i}(s)=Y_{i}(t)\;\mbox{for
all}\;\;t\geq 0,\;i\in\{1,...,b\}. \elabel{lxyzero}
\end{eqnarray}
Similarly, we know that $F(0)=0$, $F$ is non-decreasing, and
\begin{eqnarray}
&&\int_{0}^{t}I_{\bar{D}_{i}}(V(s))dF_{i}(s) = F_{i}(t)\;\;\mbox{for
all}\;\;t\geq 0,\;i\in\{1,...,\bar{b}\}.
\elabel{nxyzeroIl}
\end{eqnarray}
Therefore, by the Lipschitz condition in \eq{fbconII}, we know that
$((X,Y),(V,\bar{V},\tilde{V},F))$ satisfies the FB-SDEs in
\eq{bsdehjb} a.s. Thus, by the Skorohod representation theorem
again, it is a weak solution to the FB-SDEs in \eq{bsdehjb}.

\vskip 0.3cm {\bf Part A (Uniqueness).} Assume that
$((X,Y),(V,\bar{V},\tilde{V},F))$ is a weak solution to the FB-SDEs
in \eq{bsdehjb}. To prove its uniqueness, we introduce some
additional notations. Let $D_{\emptyset}=D$,
$\bar{D}_{\emptyset}=\bar{D}$, and define
\begin{eqnarray}
&&D_{K}\equiv\cap_{i\in
K}D_{i},\;\;\bar{D}_{\bar{K}}\equiv\cap_{i\in\bar{K}}\bar{D}_{i}
\elabel{ddI}
\end{eqnarray}
for each $\emptyset\neq K\subset\{1,...,b\}$ and each
$\emptyset\neq\bar{K}\subset\{1,...,\bar{b}\}$. In the sequel, we
call a set $K\in\{1,...,b\}$ ``maximal" if $K\neq\emptyset$,
$D_{K}\neq\emptyset$, and $D_{K}\neq D_{\tilde{K}}$ for any
$\tilde{K}\supset K$ such that $\tilde{K}\neq K$. Similarly, we can
define the maximal set corresponding to a set
$\bar{K}\in\{1,...,\bar{b}\}$. Furthermore, let $d(x,D_{K})$ and
$d(\bar{x},\bar{D}_{\bar{K}})$ respectively denote the Euclidean
distance between $x$ and $D_{K}$ for a point $x\in D$ and the
Euclidean distance between a point $\bar{x}\in\bar{D}$ and
$\bar{D}_{\bar{K}}$. Then, it follows from Lemma 3.2 in
Dai~\cite{dai:broapp} or Lemma B.1 in Dai and
Williams~\cite{daiwil:exiuni} that there exist two constants $C\geq
1$ and $\bar{C}\geq 1$ such that
\begin{eqnarray}
&&d(x,D_{K})\leq C\sum_{i\in K}(n_{i}\cdot
x-b_{i}),\;\;\bar{d}(\bar{x},\bar{D}_{\bar{K}})\leq
\bar{C}\sum_{i\in\bar{K}}(\bar{n}_{i}\cdot\bar{x}-\bar{b}_{i}).
\elabel{ddII}
\end{eqnarray}
Now, for each $\epsilon\geq 0$, $K\in\{1,...,b\}$, and
$\bar{K}\in\{1,...,\bar{b}\}$ (including the empty set), we let
\begin{eqnarray}
D^{\epsilon}_{K}&\equiv&\left\{x\in R^{q}:0\leq n_{i}\cdot
x-b_{i}\leq C_{\epsilon}\;\;\mbox{for all}\;\;i\in
K,\right.\elabel{ddIII}\\
&&\;\;\left.n_{i}\cdot x-b_{i}>\epsilon\;\;\mbox{for
all}\;\;i\in\{1,...,b\}\setminus K\right\},
\nonumber\\
\bar{D}^{\epsilon}_{\bar{K}}&\equiv&\left\{\bar{x}\in R^{q}:0\leq
\bar{n}_{i}\cdot\bar{x}-\bar{b}_{i}\leq\bar{C}_{\epsilon}\;\;\mbox{for
all}\;\;i\in\bar{K},\right.
\elabel{ddIV}\\
&&\;\;\left.\bar{n}_{i}\cdot\bar{x}-\bar{b}_{i}>\epsilon\;\;\mbox{for
all}\;\;i\in\{1,...,\bar{b}\}\setminus\bar{K}\right\}, \nonumber
\end{eqnarray}
where $C_{\epsilon}=Cp\epsilon$ and
$\bar{C}_{\epsilon}=\bar{C}q\epsilon$. Thus, by Lemmas 4.1-4.2 in
Dai and Williams~\cite{daiwil:exiuni}, we know that
\begin{eqnarray}
&&D=\cup_{K\in{\cal
G}}D^{\epsilon}_{K},\;\;\bar{D}=\cup_{\bar{K}\in\bar{\cal
G}}\bar{D}^{\epsilon}_{\bar{K}},
\elabel{ddV}
\end{eqnarray}
where, ${\cal G}$ is the collection of subsets of $\{1,...,b\}$
consisting of all maximal sets in $\{1,...,b\}$ and $\bar{{\cal G}}$
is defined in the same way in terms of subsets of
$\{1,...,\bar{b}\}$. For convenience, we order the sets in ${\cal
G}$ and $\bar{{\cal G}}$. Then, we can define a sequence of
3-dimensional points
$\{(r_{n},\bar{r}_{n},\tau_{n}),n\in\{1,2,...\}\}$ with $\tau_{0}=0$
by induction.

In fact, since $((X,Y),(V,\bar{V},\tilde{V},F))$ is a weak solution
to the FB-SDEs in \eq{bsdehjb}, both $X(0)$ and $V(0)$ are defined.
Thus, if $(r_{1},\bar{r}_{1})$ is the first
$K\times\bar{K}\in\{1,...,b\}\times\{1,...,\bar{b}\}$ such that
$(x,\bar{x})\in
D^{\epsilon}_{r_{1}}\times\bar{D}^{\epsilon}_{\bar{r}_{1}}$, we let
\begin{eqnarray}
&&\tau_{1}=\inf\left\{t\geq 0:\;(X(t),V(t))\notin
D^{\epsilon}_{r_{1}}\times
\bar{D}^{\epsilon}_{\bar{r}_{1}}\right\}\elabel{stopI}.
\end{eqnarray}
Furthermore, if $(r_{n},\bar{r}_{n},\tau_{n})$ has been defined on
$\{\tau_{n}<\infty\}$, we let $(r_{n+1},\bar{r}_{n+1})$ be the first
$K\times\bar{K}\in{\cal G}\times\bar{{\cal G}}$ such that
$(X(\tau_{n}),V(\tau_{n}))\in D^{\epsilon}_{K}\times
\bar{D}^{\epsilon}_{\bar{K}}$. Then, we can define
\begin{eqnarray}
&&\tau_{n+1}=\inf\left\{t\geq\tau_{n}:\;(X(t),V(t))\notin
D^{\epsilon}_{r_{n+1}}\times\bar{D}^{\epsilon}_{\bar{r}_{n+1}}\right\}
\elabel{stopnI}.
\end{eqnarray}
On $\{\tau_{n}=+\infty\}$, we define $r_{n+1}=r_{n}$,
$\bar{r}_{n+1}=\bar{r}_{n}$, and $\tau_{n+1}=\tau_{n}$. Due to the
right-continuity of the sample paths of solution $(X,V)$ by the
related property of L\'evy process driven stochastic integral (see,
e.g., Theorem 4.2.12 in page 204 of Applebaum~\cite{app:levpro}),
$\{\tau_{n}\}$ is a nondecreasing sequence of $\{{\cal
F}_{t}\}$-stopping times, satisfying $\tau_{n}\rightarrow\infty$
a.s. as $n\rightarrow\infty$.

Hence, it suffices to prove the weak uniqueness of
$((X,Y),(V,\bar{V},\tilde{V},F))(\cdot\wedge\tau_{n})$ for each $n$.
Note that both $D^{\epsilon}_{r_{n}}$ and
$\bar{D}^{\epsilon}_{\bar{r}_{n}}$ for each $n$ are subsets of
cones. Thus, without loss of generality, we assume that both $D$ and
$\bar{D}$ are cones. Therefore, we can prove the weak uniqueness by
induction in terms of the numbers of boundary faces of $D$ and
$\bar{D}$.

In fact, for the case that $b=\bar{b}=1$, it follows from the
uniqueness of the Skorohod mapping given by Lemma 3.1 in
Dai~\cite{dai:broapp} or Lemma 4.5 in Dai and
Dai~\cite{daidai:heatra} that the weak uniqueness is true. Now, we
suppose that the weak uniqueness is true for the case that
$b+\bar{b}=m\geq 2$ with $b\geq 1$ and $\bar{b}\geq 1$. Then, we can
prove the case for $b+\bar{b}=m+1$. In this case, we need to
consider two folds indexed by two pairs of $(b+1,\bar{b})$ and
$(b,\bar{b}+1)$. Both of the folds can be proved by the similar
discussion for Theorem 5.4 in Dai and Williams~\cite{daiwil:exiuni}.
Therefore, we finish the proof of weak uniqueness.

\vskip 0.3cm {\bf Part B.} We consider the case that $L(t,\omega)$
appeared in \eq{fbconI}-\eq{fbconII} is a constant and the spectral
radii of $S$ and each $p\times p$ sub-principal matrix of $N'R$ are
strictly less than one. In this case, we need to prove that there is
a unique strong adapted solution $((X,Y),(V,\bar{V},\tilde{V},F))$
to the system of in \eq{bsdehjb}.

In fact, it follows from the discussions in Reiman and
Harrison~\cite{harrei:refbro}, Dai~\cite{dai:optrat}, Lemma 7.1 and
Theorem 7.2 in pages 164-165 of Chen and Yao~\cite{cheyao:funque}
that there exist two Lipschitz continuous mappings $\Phi$ and $\Psi$
such that
\begin{eqnarray}
(X^{n+1},Y^{n+1})&=&\Phi(Z^{n})\elabel{mapI}\\
(V^{n+1},F^{n+1})&=&\Psi(U^{n})\elabel{mapII}
\end{eqnarray}
for each $n\in\{1,2,...\}$. Then, it follows from
\eq{mapI}-\eq{mapII}, the related estimates in Part A, and the
conventional Picard's iterative method, we can reach a proof for the
claim in Part B.

\vskip 0.3cm
 {\bf Part C.} We consider the case that $L(t,\omega)$
appeared in \eq{fbconI}-\eq{fbconII} is a constant and both of the
SDEs have no reflection boundaries. In this case, we need to prove
that there is a unique strong adapted solution
$((X,Y),(V,\bar{V},\tilde{V},F))$ to the system of in \eq{bsdehjb}.
In fact, by the related estimates in Part A, this case can be proved
by directly generalizing the conventional Picard's iterative method.
Actually, this case is a special one of Theorem~\ref{bsdeyI} or
Theorem~\ref{infdn}.

\vskip 0.3cm
{\bf Part D.} We consider the case that $L(t,\omega)$
appeared in \eq{fbconI}-\eq{fbconII} is a general adapted and
mean-squarely integrable stochastic process. The proofs
corresponding to the cases stated in Part A, Part B, and Part C can
be accomplished along the lines of proofs for Lemma 4.1 in
Dai~\cite{dai:meavar} associated with a forward SDE under random
environment and Proposition 18 in Dai~\cite{dai:meahed} for a
backward SDE under random environment. The key in the proofs is to
introduce the following sequence of $\{{\cal F}_{t}\}$-stopping
times, i.e.,
\begin{eqnarray}
&&\tau_{n}\equiv\inf\{t>0,\|L(t)\|>n\}\;\;\mbox{for
each}\;\;n\in\{1,2,...\}.
\end{eqnarray}
By the condition in \eq{fbconIII}, $\tau_{n}$ is nondecreasing and
a.s. tends to infinity as $n\rightarrow\infty$.

\vskip 0.3cm
Finally, by summarizing the cases presented in Part A
to Part D, we finish the proof of Theorem~\ref{fbexist}. $\Box$

\subsection{Proof of Theorem~\ref{gtheoremo}}

For a control process $u^{*}\in{\cal C}$, it follows from
Theorem~\ref{bsdeyI} that the $(r,q+1)$-dimensional FB-SPDEs in
\eq{fbspdef} with the partial differential operators $\{\bar{{\cal
L}},\bar{{\cal J}},\bar{{\cal I}}\}$ given by
\eq{paretoopto}-\eq{paretooptII} and terminal condition in
\eq{paretooptIII} indeed admits a well-posed 4-tuple solution
$(U(t,x)$, $V(t,x)$, $\bar{V}(t,x)$, $\tilde{V}(t,x,\cdot))$. Thus,
substituting
\begin{eqnarray}
(V(t),\bar{V}(t),\tilde{V}(t,\cdot))\equiv
(V(t,X(t)),\bar{V}(t,X(t)),\tilde{V}(t,X(t),\cdot)) \nonumber
\end{eqnarray}
into the system of coupled FB-SDEs in \eq{bsdehjb}, it follows from
Theorem~\ref{fbexist} that the claims in Theorem~\ref{gtheoremo} are
true. $\Box$

\section{Proof of Theorem~\ref{gtheorem}}\label{gtheoremproof}

The proof of part 1 is the direct extension of the
single-dimensional case (i.e., $p=q=1$) for the related optimal
control problem in $\emptyset$ksendal {\em et
al.}~\cite{okssul:stohjb}.

The proof of part 2 can be done as follows. For each $u\in{\cal C}$
and $\gamma(t,x)=\beta(t,x)\equiv 0$, it follows from
Theorem~\ref{gtheoremo} that the regulator processes $F(t)$ and
$Y(t)$ exist. Since they are nondecreasing with respect to time
variable $t$, the derivatives $\frac{dF}{dt}(t,x)$ and
$\frac{dY}{dt}(t,x)$ exist a.e. in terms of time variable $t$ along
each sample path a.s. Furthermore, if each $q\times q$ sub-principal
matrix of $\bar{N}'S$ and each $p\times p$ sub-principal matrix of
$N'R$ are invertible, these derivatives are uniquely determined
owing to the Skorohod mapping. Nevertheless, if only the general
completely-${\cal S}$ condition is imposed, these derivatives are
weakly unique in a probability distribution sense. In addition, it
follows from Proposition 7.1 in Ethier and
Kurtz~\cite{ethkur:marpro} that these derivatives can be
approximated by polynomials in terms of variable $x$ for each given
$t$, which are denoted by $\gamma(t,x)$ and $\beta(t,x)$. Then, the
proof for the claim in part 2 follows from the one for the claim in
part 1. Hence, we reach a proof for Theorem~\ref{gtheorem}. $\Box$







\begin{thebibliography}{10}




\bibitem{app:levpro}
D. Applebaum, L\'{e}vy Processes and Stochastic Calculus, Cambridge
University Press, Cambridge, 2004




\bibitem{berkha:regdet}
A. Bernard and A. El Kharroubi, R\'egulations d\'eterministes et
stochastiques dans le premier ``orthant" de $R^{n}$, Stochastics
Stochastics Rep. 34 (1991), 149-167.


\bibitem{bil:conpro}
P. Billingsley, Convergence of Probability Measures, Second Edition,
John Wiley $\&$ Sons, New York, 1999.

\bibitem{boudeb:stonon}
A. de Bouard and A Debussche, A stochastic nonlinear
Schr$\ddot{o}$dinger equation with multiplicative noise, Commun.
Math Phys. 205(1) (1999), 161-181.

\bibitem{boudeb:stononI}
A. de Bouard and A. Debussche, The stochastic nonlinear
Schr$\ddot{o}$dinger equation in $H^{1}$. Stoch. Anal. Appl. 21(1)
(2003), 97-126.


\bibitem{ber:levpro}
J. Bertoin, L\'{e}vy Processes. Cambridge University Press,
Cambridge, 1996.



\bibitem{cassap:vecmed}
V. Caselles, G. Sapiro, and D. H. Chung, Vector median filters,
morphology, and PDE's: theoretical connections, Proceedings of
International Conference on Image Processing, IEEE CS Press 4
(1999), 177-184.

\bibitem{cha:denfun}
J. D. Chai, Density functional theory with fractional orbit
occupations, Journal of Chemical Physics 136, 154104 (2012).

\bibitem{cha:theass}
J. D. Chai, Thermally-assisted-occupation density functional theory
with generalized-gradient approximations, Journal of Chemical
Physics 140, 18A521 (2014).

\bibitem{chazha:expobs}
C. Z. Chang, J. Zhang, X. Feng, J. Shen, Z. Zhang, M. Guo, K. Li, Y.
Ou, P. Wei, L. L. Wang, Z. Q. Ji, Y. Feng, S. Ji, X. Chen, J. Jia,
X. Dai, Z. Fang, S. C. Zhang, K. He, Y. Wang, L. Lu, X. C. Ma, and
Q. K. Xue, Experimental observation of the quantum anomalous Hall
effect in a magnetic topologic insulator, Science 340(167) (2013).

\bibitem{cheyao:funque}
H. Chen and D. D. Yao, Fundamental of Queueing Networks,
Springer-Verlag, New York, 2001.




\bibitem{daidai:heatra}
J. G. Dai and W. Dai, ``A heavy traffic limit theorem for a class of
open queueing networks with finite buffers", Queueing Systems Vol.
32, No. 1-3, pp. 5-40, 1999.

\bibitem{daiwil:exiuni}
J. G. Dai and R. J. Williams, Existence and uniqueness of
semimartingale reflecting Brownian motions in convex polyhedrons,
Theory of Probability and its Applications Vol. 40, No. 1, pp. 1-40,
1995.

\bibitem{dai:broapp}
W. Dai, Brownian Approximations for Queueing Networks with Finite
Buffers: Modeling, Heavy Traffic Analysis and Numerical
Implementations. Ph.D Thesis, Georgia Institute of Technology, 1996.
Also published in UMI Dissertation Services, A Bell $\&$ Howell
Company, Michigan, U.S.A., 1997.

\bibitem{dai:traneu}
W. Dai, On the traveling neuron nets (human brains) controlled by a
satellite communication system, Proceedings of International
Conference on Bioinformatics and Biomedical Engineering, pp. 1-4,
IEEE Computer Society Press, 2009.

\bibitem{dai:meavar}
W. Dai, Mean-variance portfolio selection based on a generalized BNS
stochastic volatility model, International Journal of Computer
Mathematics 88 (2011), 3521-3534.

\bibitem{dai:newcla}
W. Dai, A new class of backward stochastic partial differential
eqautions with jumps and applications, arxiv, 2011.


\bibitem{dai:optrat}
W. Dai, Optimal rate scheduling via utility-maximization for
$J$-user MIMO Markov fading wireless channels with cooperation,
Operations Research 61(6) (2013), 1450-1462 (with 26 page online
e-companion ( Supplemental)).

\bibitem{dai:profor}
W. Dai, Product-form solutions for integrated services packet
networks and cloud computing systems, Mathematical Problems in
Engineering, Volume 2014 (Regular Issue), Article ID 767651, 16
pages, 2015.


\bibitem{dai:meahed}
W. Dai, Mean-variance hedging based on an incomplete market with
external risk factors of non-Gaussian OU processes, Mathematical
Problems in Engineering, Volume 2015 (Regular Issue), Article ID
625289, 20 pages, 2015.

\bibitem{daijia:stoopt}
W. Dai and Q. Jiang, Stochastic optimal control of ATO systems with
batch arrivals via diffusion approximation, Probability in the
Engineering and Informational Sciences 21 (2007), 477-495.

\bibitem{daizar:dis}
W. Dai and T. Zariphopoulou, Conference discussion during 45 minutes
invited talk presented by Zariphopoulou in 2014 International
Congress of Mathematicians (ICM 2014), Seoul, Korea, 2014.



\bibitem{ethkur:marpro}
S. N. Ethier and T. G. Kurtz, Markov Process: Characterization and
Convergence, Wiley, New York, 1986.

\bibitem{gihsko:stodif}
I. I. Gihman and A. V. Skorohod, Stochastic Differential Equations,
Springer-Verlag, Berlin, 1972.



\bibitem{hal:newact}
E. Hall, On a new action of the magnet on electric currents,
American Journal of Mathematics 2(3) (1879), 287-292.

\bibitem{hai:solkpz}
M. Hairer, Solving the KPZ equation, Annals of Mathematics 178
(2013), 559-664.

\bibitem{harrei:refbro}
J. M. Harrison and M. I. Reiman, Reflected Brownian motion on an
orthant, Annals of Probability 9(2) (1981), 302-308.



\bibitem{jacshi:limthe}
J. Jacod, A.N. Shiryaev, Limit Theorems for Stochastic Processes,
Second Edition, Springer, Berlin, 2002.

\bibitem{jon:gamthe}
A. J. Jones, Game Theory: Mathematical Models of Conflict, Horwood
Publishing Limited, 2000.

\bibitem{kal:foumod}
O. Kallenberg, Foundations of Modern Probability, Springer-Verlag,
New York, 1997.

\bibitem{karli:bsdapp}
I. Karatzas and Q. Li, BSDE approach to non-zero-sum stochastic
differential games of control and stopping, in Stochastic Processes,
Finance and Control, pp. 105-153, World Scientific Publishers, 2012.

\bibitem{karlut:haleff}
R. Karplus and J. M. Luttinger, Hall effect in ferromagnetics, Phys.
Rev. 95(5) (1954), 1154-1160.



\bibitem{konlas:clalev}
T. Konstantopoulos, G. Last, S. J. Lin, On a class of L\'evy
stochastic networks, Queueing Systems 46 (2004), 409-437.


\bibitem{liosou:notaux}
P. L. Lions and T. Souganidis, Notes aux CRAS, t. 326 (1998) Ser. I
1085-1092; t. 327 (2000), Ser I, pp. 735-741; t. 331 (2000) Ser. I
617-624; t. 331 (2000) Ser. I 783-790.


\bibitem{man:dyncom}
A. Mandelbaum, The dynamic complementarity problem, preprint, 1989.

\bibitem{manmas:strapp}
A. Mandelbaum and W. A. Massey, Strong approximations for
time-dependent queues, Mathematics of Operations Research 20(1)
(1995), 33-64.

\bibitem{manpat:stadep}
A. Mandelbaum and G. Pats, State-dependent stochastic networks. Part
I: approximations and applications with continuous diffusion limits,
Annals of Applied Probability 8(2) (1998), 569-646.



\bibitem{muszar:stopar}
M. Musiela and T. Zariphopoulou, Stochastic partial differential
equations and portfolio choice, Preprint, 2009.




\bibitem{oks:stodif}
B. $\emptyset$ksendal, Stochastic Differential Equations, Sixth
Edition, Springer, New York, 2005.


\bibitem{okssul:appsto}
B. $\emptyset$ksendal, A. Sulem, Applied Stochastic Control of Jump
Diffusions, Springer-Verlag, Berlin, 2005.

\bibitem{okssul:stohjb}
B. $\emptyset$ksendal, A. Sulem, and T. Zhang, A stochastic HJB
equation for optimal control of forward-backward SDEs (2013,
available at http://arxiv.org/pdf/1312.1472.pdf)

\bibitem{par:stopar}
E. Pardoux, Stochastic Partial Differential Equations, Lectures
given in Fudan University, Shanghai, China, April, 2007.

\bibitem{parpen:bacsto}
E. Pardoux and S. Peng, Backward stochastic differential equations
and quasilinear parabolic partial differential equations, Lecture
Notes in CIS 176 (1992), 200-217, Springer-Verlag, New York.



\bibitem{pen:stoham}
S. Peng, Stochastic Hamilton-Jacobi-Bellman equations, SIAM J.
Control and Optimization 30(2) (1992), 284-304.

\bibitem{pesshi:optsto}
G. Peskir and A. Shiryaev, Optimal Stopping and Free-Boundary
Probelms, Birkh\"{a}user Verlag, Basel, 2006.

\bibitem{pro:stoint}
P. E. Protter, Stochastic Integration and Differential Equations,
Second Edition, Springer, New York, 2004.

\bibitem{reiwil:boupro}
M. I. Reiman and R. J. Williams, A boundary property of
semimartingale reflecting Brownian motions, Probab. Th. Rel. Fields
77 (1988), 87-97.

\bibitem{sat:levpro}
K. I. Sato, L\'{e}vy Processes and Infinite Divisibility. Cambridge
University Press, Cambridge, 1999.

\bibitem{shechedaidai:finele}
X. Shen, H. Chen, J. G. Dai, and W. Dai, The finite element method
for computing the stationary distribution of an SRBM in a hypercube
with applications to finite buffer queueing networks, Queueing
Systems \textbf{42} (2002), 33-62.

\bibitem{sit:solbac}
R. Situ, On solutions of backward stochastic differential equations
with jumps and applications, Stochastic Processes and Their
Applications 66 (1997), 209-236.

\bibitem{tho:quahal}
D. J. Thouless, The quantum Hall Effect and the Schr$\ddot{o}$dinger
equation with competing periods, Number Theory and Physics, Springer
Proceedings in Physics 47 (1990), 170-176.







\bibitem{yonzho:stocon}
J. Yong and X. Y. Zhou, Stochastic Controls: Hamiltonian Systems and
HJB Equations, Springer-Verlag, New York, 1999.




\end{thebibliography}
\end{document}